\documentclass[11pt, letterpaper]{article}
\pdfoutput=1
\usepackage[authoryear]{natbib}
\usepackage{amsmath,amssymb,amsfonts,amsthm,bbm,graphicx}
\usepackage{epic,eepic,epstopdf,longtable}
\usepackage{multirow,verbatim}
\usepackage{array}

\usepackage{epsfig}
\usepackage{subcaption}
\usepackage{setspace}
\usepackage{url}
\usepackage{color}
\usepackage{algorithm,algorithmic}

\makeatletter
\def\singlespace{\def\baselinestretch{1}\@normalsize}

\makeatletter
\def\singlespace{\def\baselinestretch{1}\@normalsize}

\numberwithin{equation}{section}

\renewcommand{\hat}{\widehat}

\renewcommand{\hat}{\widehat}

\newcommand{\bfm}[1]{\ensuremath{\mathbf{#1}}}

   \def\bA{\bfm A}  
   \def\bB{\bfm B}  
   \def\bC{\bfm C}  \def\CC{\mathbb{C}}
     
\def\be{\bfm e}     \def\EE{\mathbb{E}}
    
\def\bg{\bfm g}     
\def\bh{\bfm h}     
   \def\bI{\bfm I}

   \def\bL{\bfm L}  
   \def\bM{\bfm M}

   \def\bP{\bfm P}  \def\PP{\mathbb{P}}
     
\def\br{\bfm r}   \def\bR{\bfm R}  \def\RR{\mathbb{R}}
   \def\bS{\bfm S}  
     
\def\bu{\bfm u}   \def\bU{\bfm U}  
\def\bv{\bfm v}   \def\bV{\bfm V}  
\def\bw{\bfm w}   \def\bW{\bfm W}  
\def\bx{\bfm x}     
\def\by{\bfm y}   \def\bY{\bfm Y}  
\def\bz{\bfm z}   \def\bZ{\bfm Z}

\def\calE{{\cal  E}}

\def\calM{{\cal  M}} 
\def\calN{{\cal  N}}

\newcommand{\bfsym}[1]{\ensuremath{\boldsymbol{#1}}}

 \def\bbeta{\bfsym \beta}
              \def\bGamma{\bfsym \Gamma}

 \def\beps{\bfsym \varepsilon}          \def\bepsilon{\bfsym \varepsilon}

\DeclareMathOperator{\rank}{rank}

\def\eps{\varepsilon}
\def\beps{\mbox{\boldmath$\eps$}}

\def\today{\ifcase\month\or
  January\or February\or March\or April\or May\or June\or
  July\or August\or September\or October\or November\or December\fi
  \space\number\day, \number\year}

\newdimen\biblioindent    \biblioindent=30pt

 at 8truept

\def\eps{\varepsilon}

\newcommand{\beq}{\begin{equation}}
  \newcommand{\eeq}{\end{equation}}
\newcommand{\beqn}{\begin{eqnarray}}
  \newcommand{\eeqn}{\end{eqnarray}}
\newcommand{\beqnn}{\begin{eqnarray*}}
  \newcommand{\eeqnn}{\end{eqnarray*}}

\allowdisplaybreaks
\usepackage{verbatim}
\renewcommand{\baselinestretch}{1.66}
\newcounter{CondCounter}

\numberwithin{equation}{section}
\theoremstyle{plain}
\newtheorem{theorem}{Theorem}

\theoremstyle{definition}
\newtheorem{remark}{Remark}

\theoremstyle{definition}

\theoremstyle{definition}
\newtheorem{lemma}{Lemma}

\newtheorem{corollary}{Corollary}

\theoremstyle{definition}

\theoremstyle{definition}

\theoremstyle{definition}

\usepackage{mathtools}
\DeclarePairedDelimiter\ceil{\lceil}{\rceil}

\begin{document}
\title{The Sup-norm Perturbation of HOSVD and Low Rank Tensor Denoising}
\author{
Dong Xia\footnote{Most of the work was done when the first author was a Visiting Assistant Professor in Department of Statistics at University of Wisconsin-Madison and later a Post-doctoral Research Scientist in Department of Statistics at Columbia University.} and Fan Zhou\footnote{Fan Zhou is supported in part by NSF Grants DMS-1509739 and CCF-1523768.}\\ Hong Kong University of Science and Technology and Georgia Institute of Technology\\Email: madxia@ust.hk, fzhou40@math.gatech.edu}

\date{(\today)}
\maketitle

\begin{abstract}
The higher order singular value decomposition (HOSVD) of tensors is a generalization of matrix SVD. The perturbation analysis of HOSVD under random noise is more delicate than its matrix counterpart. Recently, polynomial time algorithms have been proposed where statistically optimal estimates of the singular subspaces and the low rank tensors 
are attainable in the Euclidean norm. In this article, we analyze the sup-norm perturbation bounds of HOSVD and introduce estimators of the singular subspaces with sharp deviation bounds in the sup-norm. We also investigate a low rank tensor denoising estimator and demonstrate its fast convergence rate with respect to the entry-wise errors. {The sup-norm perturbation bounds reveal unconventional phase transitions for statistical learning applications such as the exact clustering in high dimensional Gaussian mixture model and the exact support recovery in sub-tensor localizations.} 
 In addition, the bounds established for HOSVD also elaborate the one-sided sup-norm perturbation bounds for the singular subspaces of unbalanced (or fat) matrices.
\end{abstract}


\section{Introduction}\label{introsec}
A tensor is a multi-array of more than 2 dimensions, which can be viewed as a higher order generalization of matrices. Data of tensor types has been widely available in many fields, such as image and video processing (see~\cite{liu2013tensor}, \cite{westin2002processing}, \cite{hildebrand1997new}, \cite{li2010tensor}, \cite{vasilescu2002multilinear}); latent variable modelling (see \cite{anandkumar2014tensor}, \cite{cichocki2015tensor}, \cite{chaganty2013spectral}); genomic signal processing (\cite{omberg2007tensor}, \cite{muralidhara2011tensor} and \cite{ponnapalli2011higher}) and references therein. It is demanding to handle these datasets in order to take the most advantages of the tensor structures. The task is challenging due to the highly non-convexity of tensor related optimization problems. For instance, computing the tensor operator norm is generally NP-hard (see, e.g., \cite{hillar2013most}) while it can be implemented fast for matrices.

The higher order singular value decomposition (HOSVD) is one machinery to deal with tensors which generalizes the matrix SVD to higher order tensors, see \cite{zheng2015interpolating}, \cite{de2000multilinear}, \cite{bergqvist2010higher}, \cite{chen2009tensor} and \cite{kolda2009tensor}. The conceptual simplicity and computational efficiency make HOSVD popular.  It has been successfully applied on various statistical learning tasks, for instance,  face recognition (see \cite{vasilescu2002multilinear}),  genomic signal processing (see \cite{muralidhara2011tensor}) and more examples in a survey paper (\cite{acar2009unsupervised}). Basically,
the HOSVD unfolds a higher order tensor into matrices and treat it with standard matrix techniques to obtain the principal singular subspaces in each dimension (see more details in Section~\ref{sec:prelim}).
Although the HOSVD shows appealing effectiveness, there are several fundamental theoretic mysteries yet to be uncovered.

One particularly important question is related with the perturbation of HOSVD when a low rank tensor is contaminated with stochastic noise. The difficulty comes from both methodological and theoretical aspects. The computation of HOSVD is essentially reduced to matrix SVD which can be implemented efficiently.
 This naive estimator is actually statistically sub-optimal. It is well-known that 
 further power iterations can ameliorate the naive spectral initializations and thus deliver statistically optimal estimators, see more details in \cite{richard2014statistical}, \cite{zhang2017tensor}, \cite{hopkins2015tensor}, \cite{liu2017characterizing} and references therein. Another intriguing phenomenon is on the phase transitions of the signal-to-noise ratio (SNR). Actually, the SNR exhibits distinct computational and statistical phase transitions, while the differences do not exist for matrix SVD.
 In particular, there is a gap on SNR between statistical optimality and computational optimality for HOSVD, see  \cite{zhang2017tensor}. For introductory simplicity \footnote{More general results where $\bA$ is $d_1\times d_2\times d_3$ with multilinear ranks $(r_1,r_2,r_3)$ can be found in Section~\ref{sec:mainthm}. The results of this article can be easily generalized to higher order tensors.}, we focus on the third-order tensors.  Suppose that an unknown tensor $\bA\in\RR^{d\times d\times d}$ with multilinear ranks $(r,r,r)$ is planted in a noisy observation $\bY$ with
\begin{equation}\label{eq:Y=A+Z}
\bY=\bA+\bZ\in\RR^{d\times d\times d}.
\end{equation}
The noise tensor $\bZ$ has i.i.d. entries with $Z(i,j,k)\sim\calN(0,\sigma^2)$ for $i,j,k\in[d]$ and noise variance $\sigma^2>0$. Here, we denote by $[d]:=\{1,\ldots,d\}$. The signal strength $\underline{\Lambda}(\bA)$ is defined as the smallest nonzero singular values of the matrices unfolded from $\bA$ (see definitions in Section~\ref{subsec:denoising}). Let $\bU,\bV,\bW\in\RR^{d\times r}$ denote the singular vectors of $\bA$ in the corresponding dimensions. It was proved (see \cite{zheng2015interpolating}, \cite{zhang2017tensor} and \cite{liu2017characterizing}) that if the signal strength $\underline{\Lambda}(\bA)\geq D_1\sigma d^{3/4}$ for a large enough constant $D_1>0$, the following bound holds
\begin{align*}
r^{-1/2} \max\Big\{\big\|\hat{\bU}\hat{\bU}^{\top}-\bU\bU^{\top}\big\|_{\ell_2}, \big\|\hat{\bV}\hat{\bV}^{\top}-\bV\bV^{\top}\big\|_{\ell_2}, \big\|\hat{\bW}&\hat{\bW}^{\top}-\bW\bW^{\top}\big\|_{\ell_2}\Big\}\\
&=O_p\bigg(\frac{\sigma d^{1/2}}{\underline{\Lambda}(\bA)}+\frac{\sigma d^{3/2}}{\underline{\Lambda}^2(\bA)}\bigg),
\end{align*}
where $\hat{\bU}, \hat{\bV}, \hat{\bW}$ represent the naive SVD obtained from noisy tensor $\bY$ and $\|\cdot\|_{\ell_2}$ denotes the Euclidean norm. Power iterations (also called higher order orthogonal iterations, see \cite{de2000best}) can improve the estimate (denoted by $\widetilde{\bU}, \widetilde{\bV}, \widetilde{\bW}$) to
\begin{align}\label{eq:mle}
r^{-1/2}\max\Big\{ \big\|\widetilde{\bU}\widetilde{\bU}^{\top}-\bU\bU^{\top}\big\|_{\ell_2}, \big\|\widetilde{\bV}\widetilde{\bV}^{\top}-\bV\bV^{\top}\big\|_{\ell_2}, \big\|\widetilde{\bW}\widetilde{\bW}^{\top}-&\bW\bW^{\top}\big\|_{\ell_2}\Big\}\nonumber\\
&=O_p\bigg(\frac{\sigma d^{1/2}}{\underline{\Lambda}(\bA)}\bigg),
\end{align}
which is statistically optimal (see \cite{zhang2017tensor}). Moreover, it is 
demonstrated in \cite{zhang2017tensor}, built on a hardness conjecture of the hyper-graphical planted clique detection problem, that if $\underline{\Lambda}(\bA)=o\big(\sigma d^{3/4}\big)$, then all polynomial time algorithms deliver trivial estimates of $\bU, \bV, \bW$ in general.

One focus of this article is on estimating the linear forms of tensor singular vectors in model (\ref{eq:Y=A+Z}). More specifically, let $\bU=\big(\bu_1,\ldots,\bu_r\big)\in \RR^{d\times r}$ be $\bA$'s singular vectors in certain mode, our goal is to estimate $\langle\bu_j, \bx\rangle$ for fixed $\bx\in\RR^{d}$ and $j=1,\ldots,r$. Through choosing $\bx$ all over the canonical basis vectors in $\RR^d$, we end up with an estimate of $\bu_j$ whose component-wise perturbation bound can be attained. Unlike the $\ell_2$-norm perturbation bound, the $\ell_\infty$ bound can characterize the entry-wise sign consistency and entry-wise significance (i.e. entry-wise magnitude)  of singular vectors. The component-wise signs of singular vectors are critical in numerous applications such as community detection (see \cite{florescu2015spectral}, \cite{newman2004detecting}, \cite{mitra2009entrywise} and \cite{jin2015fast}). The entry-wise significance is advantageous in sub-matrix localizations, see \cite{cai2015computational}, \cite{ma2015computational} and references therein. 
In Section~\ref{sec:app}, we show that the sup-norm perturbation bounds reveal unconventional phase transitions for the exact clustering in high dimensional Gaussian mixture model. Put it simply, algorithms based on the sup-norm bounds require weaker SNR conditions than algorithms driven by the $\ell_2$-norm bounds to guarantee exact clustering. Furthermore, it enables us to construct a low rank denoising estimator of $\bA$  so that entry-wise denoising is fulfilled. To the best of our knowledge, ours is the first result concerning the low rank tensor denoising with sharp entry-wise deviation bounds. In Section~\ref{sec:app}, we show that a simple algorithm based on the $\ell_\infty$ bounds can exactly recover the supports for sub-tensor localizations (see Remark~\ref{rmk:subtensor_loc}). 

To better highlight our contributions, suppose that $\bA$ is an orthogonally decomposable third order tensor with (in particular, the CP decomposition of orthogonally decomposable tensors) 
\begin{equation}\label{eq:CP}
\bA=\sum_{k=1}^r \lambda_k \big(\bu_k\otimes \bv_k\otimes \bw_k\big),\quad \lambda_1\geq\ldots\geq \lambda_r>0
\end{equation}
where $\bU=(\bu_1,\ldots,\bu_r), \bV=(\bv_1,\ldots,\bv_r)$ and $\bW=(\bw_1,\ldots,\bw_r)$ are $d\times r$ matrices containing orthonormal columns. 
The $k$-th eigengap is written as $\bar{g}_k\big(\calM_1(\bA)\big)=\bar{g}_k\big(\calM_2(\bA)\big)=\bar{g}_k\big(\calM_3(\bA)\big)=\min\big(\lambda_{k-1}-\lambda_k, \lambda_k-\lambda_{k+1}\big)$ where $\calM_j(\bA)$ represents the matrices from unfoldings of $\bA$ (see Section~\ref{sec:prelim}). We preset $\lambda_0=+\infty$ and $\lambda_{r+1}=0$ for notational consistency. 
We show that, if $\bar{g}_k\big(\calM_1(\bA)\calM_1^{\top}(\bA)\big)\geq D_1\big(\sigma \lambda_1d^{1/2}+\sigma^2 d^{3/2}\big)$ for a large enough absolute constant $D_1>0$, then the following bound holds for any $\bx\in\RR^d$,
\begin{gather}\label{eq:hatuk_x}
\Big|\langle\hat{\bu}_k, \bx\rangle-(1+b_k)^{1/2}\langle\bu_k, \bx\rangle\Big|=O_p\bigg(\|\bx\|_{\ell_2}\frac{\lambda_1\sigma+d\sigma^2}{\bar{g}_k\big(\calM_1(\bA)\calM_1^{\top}(\bA)\big)}\bigg)=O_p\bigg(\frac{\|\bx\|_{\ell_2}}{d^{1/2}}\bigg).
\end{gather}
where $b_k\in[-1/2,0]$ is a constant which does not depend on $\bx$. The $d\times r$ matrix $\hat \bU=(\hat \bu_1,\cdots,\hat \bu_k)$ represent the empirical left singular vectors of mode-$1$ unfolding of $\bY$ satisfying model (\ref{eq:Y=A+Z}).

In the special case that  $r=1$ (rank one spiked tensor PCA model, see \cite{richard2014statistical}) such that $\underline{\Lambda}(\bA)=\bar{g}_1\big(\calM_1(\bA)\big)=\lambda_1$, we get from (\ref{eq:hatuk_x}) that 
\begin{equation}\label{eq:hatu1_x}
\Big|\langle\hat{\bu}_1, \bx\rangle-(1+b_1)^{1/2}\langle\bu_1, \bx\rangle\Big|=O_p\bigg(\frac{\sigma}{\underline{\Lambda}(\bA)}+\frac{\sigma^2d}{\underline{\Lambda}^2(\bA)}\bigg)\|\bx\|_{\ell_2}.
\end{equation}
By taking $\bx$ over the canonical basis vectors in $\RR^d$, the bounds in (\ref{eq:hatu1_x}) imply that
\begin{equation}\label{eq:hatu1_infty}
\big\|\hat{\bu}_1-(1+b_1)^{1/2}\bu_1\big\|_{\ell_{\infty}}=O_p\bigg(\Big(\frac{\log d}{d}\Big)^{1/2}\bigg)
\end{equation}
under the eigengap condition $\lambda_1\gg \sigma d^{3/4}$. It is the standard requirement in tensor PCA.  
\footnote{We shall point out that a similar result on matrix SVD has appeared in \cite{koltchinskii2016perturbation} which is sub-optimal for tensors or unbalanced matrices. Indeed, the result in \cite{koltchinskii2016perturbation} is established under the eigengap condition $\lambda_1\geq D_1\sigma d$. } Based on (\ref{eq:hatu1_infty}), we propose a low rank tensor estimator (denoted by $\hat{\bA}$) under the same SNR requirements such that 
\begin{align}\label{eq:hatA_infty}
\|\hat{\bA}-\bA\|_{\ell_\infty}=O_p\bigg(\Big(\frac{\sigma^2 d}{\lambda_1}+\sigma\Big)\big(\|\bu_1\|_{\ell_\infty}\|\bv_1\|_{\ell_\infty}+\|\bu_1\|_{\ell_\infty}\|\bw_1\|_{\ell_\infty}+\|\bv_1\|_{\ell_\infty}\|\bw_1\|_{\ell_\infty}\big)\bigg).
\end{align}
Equation (\ref{eq:hatA_infty}) shows that the entry-wise denoising bound of the novel estimator $\hat \bA$ is determined by the coherences of the singular vectors $\bu_1, \bv_1$ and  $\bw_1$. In particular, if $\bu_1, \bv_1, \bw_1$ are incoherent so that $\max\{\|\bu_1\|_{\ell_\infty}, \|\bv_1\|_{\ell_\infty}, \|\bw_1\|_{\infty}\}=O\big(\frac{1}{\sqrt{d}}\big)$, then equation (\ref{eq:hatA_infty}) implies that 
$$
\|\hat \bA -\bA\|_{\ell_\infty}=O_p\Big(\frac{\sigma^2}{\lambda_1}+\frac{\sigma}{d}\Big).
$$

Our main contribution is on the theoretical front. The HOSVD is essentially the standard SVD computed on an unbalanced matrix where the column size is much larger than the row size. The perturbation tools, such as Wedin's $\sin\Theta$ theorem (\cite{wedin1972perturbation}), characterize the $\ell_2$ bounds through the larger dimension, even when the left singular space lies in a low dimensional space. 
At the high level, the HOSVD is connected to the one-sided spectral analysis (see, e.g.,\cite{zheng2015interpolating}, \cite{wang2015singular}, \cite{cai2016rate} and references therein) which provide sharp perturbation bounds in $\ell_2$-norm. There are recent bounds  (see \cite{fan2016ell} and \cite{cape2017two}) in $\ell_\infty$-norm developed under additional constraint (incoherent singular subspaces) and structural noise (sparse noise). To obtain a sharp $\ell_\infty$-norm bound, we borrow the instruments invented by \cite{koltchinskii2016asymptotics} and extensively applied in  \cite{koltchinskii2016perturbation}. Our framework starts from a second order method of estimating the singular subspaces, which improves the eigengap condition than the first order method.
Similar techniques have been proposed for tensor completion (\cite{xia2017polynomial}) and tensor PCA (\cite{zheng2015interpolating} and \cite{liu2017characterizing}). The success of this seemingly natural treatment hinges upon delicate dealing with the correlations among higher order terms. We benefit from these $\ell_\infty$-norm perturbation bounds by proposing a low rank estimator for tensor denoising where the entry-wise deviation error is guaranteed by the tensor incoherence conditions. 

We organize our paper as follows. Tensor notations and preliminaries on HOSVD are explained in Section~\ref{sec:prelim}. Our main theoretical contributions are presented in Section~\ref{sec:mainthm} which includes the $\ell_\infty$-norm bound of the singular subspace perturbation and the entry-wise accuracy of a low rank tensor denoising estimator. In Section~\ref{sec:app}, we apply our theoretical results on applications including high dimensional clustering and sub-tensor localizations to manifest the advantages of utilizing $\ell_\infty$ bounds, where algorithms driven by the $\ell_\infty$-norm bounds are designed. Results of numerical experiments are displayed in Section~\ref{sec:numeric}. The proofs are provided in Section~\ref{sec:proofs}. 

\section{Preliminaries on Tensor and HOSVD}\label{sec:prelim}
\subsection{Notations}
We first review some notations that will be used through the paper. We use boldfaced upper-case letters to denote tensors or matrices, and use the same letter in normal font with indices to denote its entries. We use boldfaced lower-case letters to represent vectors, and the same letter in normal font with indices to represent its entries. For notationally simplicity, our main context is focused on third-order tensors, while our results can be easily generalized to higher order tensors.

Given a third-order tensor $\bA\in\RR^{d_1\times d_2\times d_3}$, define a linear mapping $\calM_1: \RR^{d_1\times d_2\times d_3}\mapsto \RR^{d_1\times (d_2d_3)}$ such that
$$
\calM_1(\bA)\big(i_1, (i_2-1)d_3+i_3\big)=A(i_1,i_2,i_3),\quad i_1\in[d_1], i_2\in[d_3], i_3\in[d_3]
$$
which is conventionally called the unfolding (or matricization) of tensor $\bA$. It is also called the mode-$1$ unfolding of $\bA$. The columns of matrix $\calM_1(\bA)$ are called the mode-$1$ fibers of $\bA$. The corresponding matricizations $\calM_2(\bA)$ and $\calM_3(\bA)$ can be defined in a similar fashion. The multilinear ranks of $\bA$ are then defined by:
\begin{align*}
r_1(\bA):=\rank\big(\calM_1(\bA)\big),\quad
r_2(\bA):=\rank\big(\calM_2(\bA)\big),\quad
r_3(\bA):=\rank\big(\calM_3(\bA)\big)
\end{align*}
Note that $r_1(\bA), r_2(\bA), r_3(\bA)$ are unnecessarily equal with each other in general. We write $\br(\bA):=\big(r_1(\bA), r_2(\bA), r_3(\bA)\big)$ which are also called the Tucker ranks of $\bA$.

The marginal product $\times_1: \RR^{r_1\times r_2\times r_3}\times \RR^{d_1\times r_1}\mapsto \RR^{d_1\times r_2\times r_3}$ is given by
$$
\bC\times_1 \bU=\bigg(\sum_{j_1=1}^{r_1} C(j_1,j_2,j_3) U(i_1,j_1)\bigg)_{i_1\in[d_1], j_2\in[r_2], j_3\in[r_3]},
$$
and $\times_2$ and $\times_3$ are defined similarly. Therefore, we write the multilinear product of tensors $\bC\in\RR^{r_1\times r_2\times r_3}, \bU\in\RR^{d_1\times r_1}, \bV\in\RR^{d_2\times r_2}$ and $\bW\in\RR^{d_3\times r_3}$ as
$$
\bC\cdot(\bU,\bV,\bW)=\bC\times_1\bU\times_2\bV\times_3\bW\in\RR^{d_1\times d_2\times d_3}.
$$
We use $\|\cdot\|$ to denote the operator norm of matrices and $\|\cdot\|_{\ell_2}$ and $\|\cdot\|_{\ell_\infty}$ to denote $\ell_2$ and $\ell_{\infty}$ norms of vectors, or vectorized matrices and tensors.

\subsection{HOSVD and Eigengaps}
For a tensor $\bA\in\RR^{d_1\times d_2\times d_3}$ with multilinear ranks $\br(\bA)=\big(r_1(\bA),r_2(\bA), r_3(\bA)\big)$, 
let $\bU\in \RR^{d_1\times r_1(\bA)}, \bV\in\RR^{d_2\times r_2(\bA)}$ and $\bW\in\RR^{d_3\times r_3(\bA)}$ be the left singular vectors of $\calM_1(\bA), \calM_2(\bA)$ and $\calM_3(\bA)$ respectively, which can be computed efficiently via matricization followed by thin singular value decomposition.
The higher order singular value decomposition (HOSVD) refers to the decomposition
\begin{equation}\label{eq:HOSVD}
\bA=\bC\times_1 \bU\times_2 \bV\times_3 \bW
\end{equation}
where the $r_1(\bA)\times r_2(\bA)\times r_3(\bA)$ core tensor $\bC$ is obtained by
$
\bC:=\bA\times_1 \bU^{\top}\times_2 \bV^{\top}\times_3 \bW^{\top}.
$

Suppose that a noisy version of $\bA$ is observed as in model (\ref{eq:Y=A+Z}) so that 
$$
\bY=\bA+\bZ
$$
where $\bZ\in\mathbb{R}^{d_1\times d_2\times d_3}$ is an unknown noise tensor with i.i.d. entries satisfying
$Z(i,j,k)\sim\calN(0,\sigma^2)$. By observing $\bY$, our goal is to estimate $\bU,\bV$ and $\bW$. An immediate approach is to compute HOSVD of $\bY$. To this end,
let $\hat{\bU}\in\RR^{d_1\times r_1}, \hat{\bV}\in\RR^{d_2\times r_2}, \hat{\bW}\in\RR^{d_3\times r_3}$ be the corresponding top singular vectors of $\calM_1(\bY), \calM_2(\bY)$ and $\calM_3(\bY)$. 
The key factor characterizing the perturbation bounds of $\hat{\bU}, \hat\bV$ and $\hat\bW$  is the so-called eigengap. 

Since the computing of $\hat{\bU}$ is essentially via the matrix SVD on $\calM_1(\bA)$, it suffices to consider the eigengaps of matrices.
Given a rank $r$ matrix $\bM\in\RR^{m_1\times m_2}$ with SVD:
$$
\bM=\sum_{k=1}^{r}\lambda_k\big(\bg_k\otimes \bh_k\big)
$$
where $\bM$'s singular values are $\lambda_1\geq \lambda_2\geq\ldots\geq \lambda_r> 0$ and $\{\bg_1,\ldots,\bg_r\}$ are the corresponding left singular vectors and $\{\bh_{1},\ldots,\bh_r\}$ are $\bM$'s corresponding right singular vectors. We further introduce $\lambda_0=+\infty$ and $\lambda_{r+1}=0$.
The $k$-th eigengap of matrix $\bM$ is then defined by
$$
\bar{g}_k(\bM):=\min\big(\lambda_k-\lambda_{k+1}, \lambda_{k-1}-\lambda_{k}\big),\quad \forall\ 1\leq k\leq r.
$$
Recall that $\bU, \hat\bU\in\RR^{d_1\times r_1}$ are the top-$r_1$ left singular vectors of $\calM_1(\bA)$ and $\calM_1(\bY)$ respectively. By Davis-Kahan Theorem (\cite{davis1970rotation}) or Wedin's $\sin\Theta$ theorem (\cite{wedin1972perturbation}), we get
\begin{equation}\label{eq:wedin}
\|\hat{\bU}\hat{\bU}^{\top}-\bU\bU^{\top}\|=O\bigg(\frac{\|\calM_1(\bZ)\|}{\bar{g}_{r_1}\big(\calM_1(\bA)\calM_1^{\top}(\bA)\big)}\bigg),
\end{equation}
which is generally sub-optimal especially when $\calM_1(\bZ)\in\RR^{d_1\times (d_2d_3)}$ is unbalanced such that $d_2d_3\gg d_1$. Sharper bounds in $\ell_2$-norm concerning one sided perturbation have been derived in \cite{zheng2015interpolating}, \cite{wang2015singular} and \cite{cai2016rate}. In this paper, we derive sharp perturbation bounds of $\hat\bU, \hat\bV, \hat\bW$ in $\ell_\infty$-norm which illustrate unconventional phase transitions for various statistical learning applications. More generally, we 
will investigate the perturbation bounds of linear forms $\langle \hat{\bu}_k, \bx\rangle$ for any fixed vector $\bx\in\RR^{d_1}$. Similar results can be obtained for singular vectors $\hat{\bV}$ and $\hat{\bW}$.

\section{Main Results}\label{sec:mainthm}
\subsection{Second Order Method for One-sided Spectral Analysis}
The $\ell_\infty$-norm perturbation bounds for singular subspaces of balanced matrices has been developed in \cite{koltchinskii2016perturbation}.
Recall that $\bu_k$ denotes the $k$-th left singular vector of $\calM_1(\bA)$ and $\hat{\bu}_k$ denotes the $k$-th left singular vector of $\calM_1(\bY)$ where $\calM_1(\bA)$ is of size $d_1\times (d_2d_3)$.
The operator norm $\|\calM_1(\bZ)\|$ is generally determined by the larger dimension $(d_1\vee d_2d_3)$, see Section~\ref{sec:proofs}.
It turns out that the machinery in \cite{koltchinskii2016perturbation} is sub-optimal concerning the SNR requirement.  Indeed, the eigengap requirement in \cite{koltchinskii2016perturbation} becomes $\bar{g}_k\big(\calM_1(\bA)\calM_1^{\top}(\bA)\big)\gg \sigma\big(d_1\vee d_2d_3\big)^{1/2}$, which shall is unnecessarily strong in view of the recent results in \cite{zheng2015interpolating}, \cite{cai2016rate}, \cite{zhang2017tensor} and \cite{liu2017characterizing}.

To bridge such gaps, we conduct a second order spectral analysis for $\hat{\bU}$. 
The key observation is that the top left singular vectors of $\calM_1(\bY)$ are also the top eigenvectors of $\calM_1(\bY)\calM_1^{\top}(\bY)$. The second order method seeks the eigenspace perturbation on $\calM_1(\bY)\calM_1^{\top}(\bY)$ instead of singular space perturbation on $\calM_1(\bY)$.
 Clearly, we have
$$
\calM_1(\bY)\calM_1^{\top}(\bY)=\calM_1(\bA)\calM_1^{\top}(\bA)+\bGamma\in\mathbb{R}^{d_1\times d_1}
$$
where $\bGamma=\calM_1(\bA)\calM_1^{\top}(\bZ)+\calM_1(\bZ)\calM_1^{\top}(\bA)+\calM_1(\bZ)\calM_1^{\top}(\bZ)$. Note that $\bU$ are the leading eigenvectors of $\calM_1(\bA)\calM_1^{\top}(\bA)$ and $\hat{\bU}$ are the top-$r_1$ eigenvectors of $\calM_1(\bY)\calM_1^{\top}(\bY)$. Moreover, the following relation on eigengaps is obvious:
$$
\bar{g}_{r_1}\Big(\calM_1(\bA)\calM_1^{\top}(\bA)\Big)\geq \bar{g}_{r_1}^2\big(\calM_1(\bA)\big).
$$
The advantage of second order method comes from the observation that even though
$\mathbb{E}\big\|\calM_1(\bZ)\calM_1^{\top}(\bZ)\big\|$ is of the order $\sigma^2(d_1\vee d_2d_3)$, the symmetric matrix $\calM_1(\bZ)\calM_1^{\top}(\bZ)$ is concentrated at $d_2d_3\sigma^2\bI_{d_1}$ such that (see more details in Section~\ref{sec:proofs})
$$
\big\|\calM_1(\bZ)\calM_1^{\top}(\bZ)-\sigma^2d_2d_3\bI_{d_1}\big\|=O_p\Big(\sigma^2\big(d_1d_2d_3\big)^{1/2}\Big).
$$
Note that subtracting by an identity matrix does not affect the eigen-structure. The second order method introduces the additional term $\calM_1(\bA)\calM_1^{\top}(\bZ)$ whose operator norm is bounded by $O_p\big(\sigma\sqrt{d_1}\big\|\calM_1(\bA)\big\|\big)$, which creates a constraint on the condition number of $\calM_1(\bA)$. However, in order to characterize sharp perturbation bounds of linear forms $\langle\hat{\bu}_k, \bx\rangle$, we need to pay more attention to dealing with correlations among the higher order terms than the first order method in \cite{koltchinskii2016perturbation}. {We note that the idea of second order method is already existing in the literature (see, e.g., \cite{zheng2015interpolating} for the $\ell_2$-norm perturbation bounds). The second order moment method is only the starting point of our technical analysis which significantly reduces the SNR requirements. Our most fundamental contribution is about the sup-norm characterization of the empirical singular vectors. Basically, we observe that the empirical singular vectors are biased and the bias is nicely aligned with the true singular vectors. After subtracting the bias, the empirical singular vectors exhibit the so-called {\it delocalization} property where all the entry-wise perturbations have comparable magnitudes. Such {\it delocalization} property is universal meaning that no conditions on the true singular vectors are needed.
In Section~\ref{sec:app}, we show that the sup-norm perturbation bounds indeed reveal unconventional phase transitions in statistical learning applications such as the exact clustering in high dimensional Gaussian mixture models and the exact support recovery in sub-tensor localizations. }

\subsection{Perturbation of Linear Forms of Singular Vectors}
In this section, we present our main theorem characterizing the perturbation of linear forms $\langle\hat{\bu}_k, \bx\rangle$ for any $\bx\in\RR^{d_1}$, where $\hat{\bu}_k$ is the $k$-th left singular vector of $\calM_1(\bY)$. Our results have similar implications as the previous work \cite{koltchinskii2016perturbation}, meaning that the bias $\mathbb{E}\hat{\bu}_k\hat\bu_k^{\top}-\bu_k\bu_k^{\top}$ is well aligned with $\bu_k\bu_k^{\top}$. Therefore, after correcting the bias term, we are able to obtain a sharper estimation of linear forms $\langle\bu_k, \bx\rangle$. To this end, we denote the condition number of the matrix $\calM_1(\bA)$ by
$$
\kappa\big(\calM_1(\bA)\big)=\frac{\lambda_{\max}\big(\calM_1(\bA)\big)}{\lambda_{\min}\big(\calM_1(\bA)\big)}
$$
where $\lambda_{\max}(\cdot)$ and $\lambda_{\min}(\cdot)$ return the largest and smallest nonzero singular values, respectively. Since $\hat\bu_k$ is up to the switch of signs, we choose $\hat \bu_k$ in the following theorems, remarks and corollaries so that $\big<\hat \bu_k, \bu_k\big>>0$. 

\begin{theorem}\label{thm:linearform}
Let\footnote{Observe that if we set $d_3=1$ and consider the case with $d_1\ll d_2$, then Theorem~\ref{thm:linearform} elaborates the one-sided perturbation bounds in $\ell_\infty$-norm for singular vectors  of unbalanced (or fat) matrices.}
$\bM:=\calM_1(\bA)$ and $\delta(d_1,d_2,d_3):=\sigma d_1^{1/2}\|\bM\|+\sigma^2(d_1d_2d_3)^{1/2}$ and suppose $d_2d_3e^{-d_1/2}\leq 1$. There exist absolute constants $D_1, D_2>0$ such that the following fact holds. Let $\bu_k$ be $\bM$'s $k$-th left singular vector with multiplicity $1$. If $\bar{g}_k\big(\bM\bM^{\top}\big)\geq D_1\delta(d_1,d_2,d_3)$, there exist a constant $b_k\in[-1/2,0]$ with $|b_k|\leq \frac{\sqrt{2}\delta(d_1,d_2,d_3)}{\bar{g}_k(\bM\bM^{\top})}$ such that for any $\bx$, the following bound holds with probability at least $1-e^{-t}$,
\begin{align}
\big|\langle\hat{\bu}_k, \bx\rangle &-(1+b_k)^{1/2}\langle\bu_k,\bx \rangle \big|\nonumber\\
\leq& D_2\bigg(t^{1/2}\frac{\sigma\|\bM\|+\sigma^2(d_2d_3)^{1/2}}{\bar{g}_k(\bM\bM^{\top})}+\frac{\sigma^2d_1}{\bar{g}_k(\bM\bM^{\top})}\Big(\frac{\delta(d_1,d_2,d_3)}{\bar{g}_k(\bM\bM^{\top})}\Big)\bigg)\|\bx\|_{\ell_2}\label{eq:thm-hatuk_x}
\end{align}
for all $\log8\leq t\leq d_1$.  In particular, if $\bx=\pm \bu_k$, then with the same probability,
\begin{align*}
\big||\langle \hat \bu_k, \bu_k\rangle|-1\big|&\leq \big|\sqrt{1+b_k}-1\big|\\
+&D_2\bigg(t^{1/2}\frac{\sigma\|\bM\|+\sigma^2(d_2d_3)^{1/2}}{\bar{g}_k(\bM\bM^{\top})}+\frac{\sigma^2d_1}{\bar{g}_k(\bM\bM^{\top})}\Big(\frac{\delta(d_1,d_2,d_3)}{\bar{g}_k(\bM\bM^{\top})}\Big)\bigg).
\end{align*}
\end{theorem}
By Theorem~\ref{thm:linearform}, it is easy to check that the condition $\bar{g}_k\big(\calM_1(\bA)\calM_1^{\top}(\bA)\big)\geq D_1\delta(d_1,d_2,d_3)$ holds whenever
$$
\bar{g}_k\big(\calM_1(\bA)\big)\geq D_1\Big(\sigma (d_1d_2d_3)^{1/4}+\sigma d_1^{1/2}\kappa\big(\calM_1(\bA)\big)\Big).
$$
If $\kappa\big(\calM_1(\bA)\big)\leq \big(\frac{d_2d_3}{d_1}\big)^{1/4}$, the above bound becomes $\bar{g}_k\big(\calM_1(\bA)\big)\gg \sigma (d_1d_2d_3)^{1/4}$ which is a standard requirement in tensor SVD or PCA, see \cite{zheng2015interpolating}, \cite{zhang2017tensor}, \cite{hopkins2015tensor} and \cite{richard2014statistical}. By taking $\bx$ over the standard basis vectors in $\RR^{d_1}$ and choosing $t\geq D_3 \log d_1$, we end up with a $\ell_{\infty}$-norm perturbation bound for empirical singular vector $\hat{\bu}_k$.


\begin{corollary}\label{cor:linfty}
Under the conditions in Theorem~\ref{thm:linearform}, there exists a universal constant $D_1>0$ such that the following bound holds with probability at least $1-\frac{1}{d_1}$,
$$
\big\|\hat{\bu}_k-(1+b_k)^{1/2}\bu_k\big\|_{\ell_{\infty}}\leq D_1\bigg(\Big(\frac{\log d_1}{d_1}\Big)^{1/2}+\Big(\frac{d_1}{d_2d_3}\Big)^{1/2}\bigg).
$$
\end{corollary} 
If $d_1\asymp d_2\asymp d_3\asymp d$, we obtain
$$
\PP\Big(\big\|\hat{\bu}_k-(1+b_k)^{1/2}\bu_k\big\|_{\ell_{\infty}}\geq D_1\Big(\frac{\log d}{d}\Big)^{1/2}\Big)\leq \frac{1}{d}
$$
which has an analogous form to the perturbation bound in \cite{koltchinskii2016perturbation} implying a famous {\it delocalization} phenomenon in random matrix theory, see \cite{rudelson2015delocalization} and \cite{vu2015random} and references therein. 
\begin{remark}\label{rmk:l2l_infty}
Let's compare with the $\ell_2$-norm bound in \cite{zheng2015interpolating} in the case that rank $r=1$, $d_1=d_2=d_3=d$ and signal strength $\bar{g}_1(\bM\bM^{\top})=\lambda^2$. By \cite[Theorem~$1$]{zheng2015interpolating}, if $\lambda\gg \sigma d^{3/4}$, then 
\begin{align}\label{eq:rmk_hatu_l2}
\|\hat \bu_1 -\bu_1\|_{\ell_2}=O_p\Big(\frac{d^{1/2}\sigma}{\lambda}+\frac{\sigma^2d^{3/2}}{\lambda^2}\Big).
\end{align}
By Theorem~\ref{thm:linearform}, if $\lambda\gg \sigma d^{3/4}$, then we get
\begin{align}\label{eq:rmk_hatu_linf}
\big\|\hat{\bu}_1-(1+b_1)^{1/2}\bu_1\big\|_{\ell_{\infty}}=O_p\Big(\frac{\sigma \log^{1/2} d}{\lambda}+\frac{\sigma^2 d\log^{1/2}d}{\lambda^2}\Big)
\end{align}
for a constant $b_1\in[-1/2, 0]$ depending on $\bu_1$ and $\lambda$ only. By (\ref{eq:rmk_hatu_linf}) and (\ref{eq:rmk_hatu_l2}), we observe that, after subtracting the bias, the entry-wise deviation of the empirical left singular vector $\hat \bu_1$ is about $\sqrt{\frac{\log d}{d}}$ of the $\ell_2$-norm perturbation bound of $\hat\bu_1$.  It means that, after subtracting the bias, the deviations of all $\hat \bu_1$'s entries have comparable magnitudes, namely the so-called {\it delocalization} property. Interestingly, if $|u_1(j)|\gg \frac{1}{\sqrt{d}}$, then eq. (\ref{eq:rmk_hatu_linf}) implies that $\hat u_1(j)$ has the same sign as $u_1(j)$ as long as $\lambda \gg\sigma d^{3/4}$. This sign consistency is crucial for guaranteeing the exact clustering of high dimensional mixture model, see more details in Section~\ref{sec:app}. 
\end{remark}

The bias $b_k$ is usually unknown and we borrow the idea in \cite{koltchinskii2016perturbation} to estimate $b_k$ based on two independent samples. It happens in the application of tensor decomposition for gene expression data where usually multiple independent copies are available, see more details in \cite{hore2016tensor}.

Suppose that two independent noisy version of $\bA\in\RR^{d_1\times d_2\times d_3}$ are observed with $\bY^{(1)}=\bA+\bZ^{(1)}$ and $\bY^{(2)}=\bA+\bZ^{(2)}$ where $\bZ^{(1)}$ and $\bZ^{(2)}$ have i.i.d. centered Gaussian entries with variance $\sigma^2$ as in (\ref{eq:Y=A+Z}). Let $\hat{\bu}_k^{(1)}$ and $\hat{\bu}_k^{(2)}$ denote the $k$-th left singular vector of $\calM_1\big(\bY^{(1)}\big)$ and $\calM_1\big(\bY^{(2)}\big)$, respectively. The signs of $\hat{\bu}_k^{(1)}$ and $\hat{\bu}_k^{(2)}$ are chosen such that $\langle\hat{\bu}_k^{(1)}, \hat{\bu}_k^{(2)} \rangle\geq 0$. Define the estimator of $b_k$ by
$$
\hat{b}_k:=\langle\hat{\bu}_k^{(1)}, \hat{\bu}_k^{(2)}\rangle-1.
$$
Define the scaled version of empirical singular vector
$
\widetilde{\bu}_k:=\frac{\hat{\bu}_k}{(1+\hat{b}_k)^{1/2}}
$
, which is not necessarily a unit vector.
\begin{theorem}\label{thm:tildeu}
Under the assumptions in Theorem~\ref{thm:linearform}, there exists an absolute constant $D_1>0$ such that for any $\bx\in\RR^{d_1}$, the follow bound holds with probability at least $1-e^{-t}$ for all $\log 8\leq t\leq d_1$,
$$
\big|\hat{b}_k-b_k\big|\leq D_1\bigg(t^{1/2}\frac{\sigma\|\bM\|+\sigma^2(d_2d_3)^{1/2}}{\bar{g}_k(\bM\bM^{\top})}+\frac{\sigma^2d_1}{\bar{g}_k(\bM\bM^{\top})}\Big(\frac{\delta(d_1,d_2,d_3)}{\bar{g}_k(\bM\bM^{\top})}\Big)\bigg)
$$
and
$$
\big|\big<\widetilde{\bu}_k-\bu_k, \bx\big>\big|\leq D_1\bigg( t^{1/2}\frac{\sigma\|\bM\|+\sigma^2(d_2d_3)^{1/2}}{\bar{g}_k(\bM\bM^{\top})}+\frac{\sigma^2d_1}{\bar{g}_k(\bM\bM^{\top})}\Big(\frac{\delta(d_1,d_2,d_3)}{\bar{g}_k(\bM\bM^{\top})}\Big)\bigg)\|\bx\|_{\ell_2}
$$
where $\bM=\calM_1(\bA)$.
\end{theorem}
\begin{remark}
By Theorem~\ref{thm:tildeu}, if $d/2\leq \min_k d_k\leq \max_k d_k\leq 2d$, we get
$$
\PP\Big(\|\widetilde{\bu}_k-\bu_k\|_{\ell_\infty}\geq D_1\Big(\frac{\log d}{d}\Big)^{1/2}\Big)\leq \frac{1}{d}.
$$
\end{remark}

\subsection{Low Rank Tensor Denoising and Entry-wise Deviation Bound}
\label{subsec:denoising}
In this section, we study a low rank estimate of $\bA$ through the projection of $\bY$. Let $\widetilde{\bU}=(\tilde{\bu}_1,\ldots,\tilde{\bu}_{r_1})\in \RR^{d_1\times r_1}$ be scaled singular vectors each of which is computed as in Theorem~\ref{thm:tildeu}.
Similarly, let $\widetilde{\bV}\in\RR^{d_2\times r_2}$ and $\widetilde{\bW}\in\RR^{d_3\times r_3}$ be the corresponding scaled singular vectors computed from $\calM_2(\bY)$ and $\calM_3(\bY)$. Define the low rank estimate
$$
\widetilde{\bA}:=\bY\times_1 \bP_{\widetilde{\bU}}\times_2 \bP_{\widetilde{\bV}}\times_3 \bP_{\widetilde{\bW}}
$$
where $\bP_{\widetilde{\bU}}$ represents the scaled projector $\bP_{\widetilde{\bU}}:=\widetilde{\bU}\widetilde{\bU}^{\top}$. Clearly, $\rank(\widetilde{\bA})=(r_1,r_2,r_3)$ which serves as a low rank estimate of $\bA$. We characterize the entry-wise accuracy of $\widetilde{\bA}$, namely, the upper bound of $\|\widetilde{\bA}-\bA\|_{\ell_\infty}$ in terms of the coherence of $\bU, \bV$ and $\bW$.
Our $\|\widetilde{\bA}-\bA\|_{\ell_{\infty}}$ bound relies on the simultaneous $\ell_{\infty}$-norm perturbation bounds of $\{\tilde{\bu}_{k_1}\}_{k_1=1}^{r_1}$, $\{\tilde{\bv}_{k_2}\}_{k_2=1}^{r_2}$ and  $\{\tilde{\bw}_{k_3}\}_{k_3=1}^{r_3}$. We impose the following conditions on the eigengaps: for a large enough constant $D_1>0$,
\begin{equation}\label{eq:denoisegk1}
\bar{g}_{k_1}\big(\calM_1(\bA)\calM_1^{\top}(\bA)\big)\geq D_1\Big(\sigma d_1^{1/2}\overline{\Lambda}(\bA)+\sigma^2(d_1d_2d_3)^{1/2}\Big),\quad 1\leq k_1\leq r_1,
\end{equation}
\begin{equation}\label{eq:denoisegk2}
\bar{g}_{k_2}\big(\calM_2(\bA)\calM_2^{\top}(\bA)\big)\geq D_1\Big(\sigma d_2^{1/2}\overline{\Lambda}(\bA)+\sigma^2(d_1d_2d_3)^{1/2}\Big),\quad 1\leq k_2\leq r_2,
\end{equation}
\begin{equation}\label{eq:denoisegk3}
\bar{g}_{k_3}\big(\calM_3(\bA)\calM_3^{\top}(\bA)\big)\geq D_1\Big(\sigma d_3^{1/2}\overline{\Lambda}(\bA)+\sigma^2(d_1d_2d_3)^{1/2}\Big),\quad 1\leq k_3\leq r_3,
\end{equation}
where we denote by
$$
\overline{\Lambda}(\bA):=\max\big\{\lambda_{\max}\big(\calM_1(\bA)\big), \lambda_{\max}\big(\calM_2(\bA)\big), \lambda_{\max}\big(\calM_3(\bA)\big) \big\}.
$$
Similarly, we define
$$
\underline{\Lambda}(\bA):=\min\Big\{\lambda_{\min}\big(\calM_1(\bA)\big),\ \lambda_{\min}\big(\calM_2(\bA)\big),\ \lambda_{\min}\big(\calM_3(\bA)\big)\Big\}
$$
and the overall eigengap
\begin{eqnarray*}
\bar{g}_{\min}\big(\bA\big):=\min\bigg\{\bar{g}_{k_1}^{1/2}\big(\calM_1(\bA)\calM_1^{\top}(\bA)\big), \bar{g}_{k_2}^{1/2}\big(\calM_2(\bA)\calM_2^{\top}(\bA)\big), \bar{g}_{k_3}^{1/2}\big(\calM_3(\bA)\calM_3^{\top}(\bA)\big)\\
,1\leq k_1\leq r_1, 1\leq k_2\leq r_2, 1\leq k_3\leq r_3
\bigg\}.
\end{eqnarray*}
By definition, it is clear that $\underline{\Lambda}(\bA)\geq \bar{g}_{\min}(\bA)$.


\begin{theorem}\label{thm:linfty}
Suppose conditions (\ref{eq:denoisegk1}) (\ref{eq:denoisegk2}) (\ref{eq:denoisegk3}) hold and assume that for all $i\in[d_1], j\in[d_2], k\in[d_3]$,
$$
\|\bU^{\top}\be_i\|_{\ell_2}\leq \mu_\bU\sqrt{\frac{r_1}{d_1}},\quad \|\bV^{\top}\be_j\|_{\ell_2}\leq \mu_\bV\sqrt{\frac{r_2}{d_2}},\quad \|\bW^{\top}\be_k\|_{\ell_2}\leq \mu_\bW\sqrt{\frac{r_3}{d_3}}
$$
for some constants $\mu_\bU, \mu_\bV,\mu_\bW\geq 0$.
Suppose that $\frac{d}{2}\leq \min_{1\leq k\leq 3} d_k\leq \max_{1\leq k\leq 3} d_k\leq 2d$ and $\frac{r}{2}\leq \min_{1\leq k\leq 3} r_k\leq \max_{1\leq k\leq 3} r_k\leq 2r$.
Then, there exists an absolute constant $D_2>0$ such that, with probability at least $1-\frac{1}{d}$,
\begin{align*}
&\big\|\tilde{\bA}-\bA\big\|_{\ell_\infty}\\
&\leq D_2
\sigma r^3\bigg(\frac{\widetilde{\kappa}(\bA)\sigma}{\bar{g}_{\min}(\bA)}+\frac{\widetilde{\kappa}^2(\bA)}{d}\bigg)\big(\mu_\bU\mu_\bV+\mu_\bU\mu_\bW+\mu_\bV\mu_\bW\big)\log^{3/2}d
\end{align*}
where $\widetilde{\kappa}(\bA)=\overline{\Lambda}(\bA)/\bar{g}_{\min}(\bA)$.
\end{theorem}
\begin{remark}
To highlight the contribution of Theorem~\ref{thm:linfty}, let $r=O(1)$ and $\widetilde{\kappa}(\bA)=O(1)$.
Note that if the coherence constants $\mu_\bU,\mu_\bV,\mu_\bW=d^{(\frac{3}{4}-\varepsilon)/2}$ for $\varepsilon\in(0,3/4)$, i.e., $\bU,\bV,\bW$ can be almost spiked, under the minimal eigengap $\bar{g}_{\min}(\bA)\gg \sigma d^{3/4}$, we obtain
$$
\|\tilde\bA-\bA\|_{\ell_\infty}=O_p\Big(\frac{\sigma}{d^{\varepsilon}}\log^{3/2}d\Big).
$$
It worths to point out that the minimax optimal bound of estimating $\bA$ in $\ell_2$-norm is $O\big(\sigma d^{1/2}\big)$, see \cite{zhang2017tensor}.
Theorem~\ref{thm:linfty} is more interesting when $\bA$ is incoherent such that $\mu_\bU, \mu_\bV, \mu_\bW= O(1)$ where we can conclude that 
\begin{align}\label{eq:rmk_tildeA_linf}
\|\tilde{\bA}-\bA\|_{\ell_\infty}=O_p\bigg(\Big(\frac{\sigma^2}{\bar{g}_{\min}(\bA)}+\frac{\sigma}{d}\Big)\log^{3/2} d\bigg)=O_p\Big(\frac{\sigma}{d^{3/4}}\log^{3/2}d\Big).
\end{align}
By (\ref{eq:rmk_tildeA_linf}), if the entry $|A(j_1,j_2,j_3)|\gg \frac{\sigma\log^{3/2}d}{d^{3/4}}$, then the entry $\tilde A(j_1,j_2,j_3)$ maintains the same sign as $A(j_1,j_2,j_3)$. In Section~\ref{sec:app} and Remark~\ref{rmk:subtensor_loc}, we show that the sup-norm bound of $\tilde\bA-\bA$ is useful for the exact support recovery of sub-tensor localizations, under minimal signal strength requirements (that is the support size). 
\end{remark}

\section{Applications}\label{sec:app}
{
In this section, we review two applications of $\ell_\infty$-norm perturbation bound. In these applications, we note that it is unnecessary to estimate the bias $b_k$. We show that the sup-norm perturbation bounds reveal unconventional phase transitions in these statistical learning applications. Meanwhile, novel
yet simple statistical algorithms can be designed based on the sup-norm perturbation bounds. }
\subsection{High Dimensional Clustering}
Many statistical and machine learning tasks are associated with clustering high dimensional data, see \cite{mccallum2000efficient}, \cite{parsons2004subspace}, \cite{fan2008high}, \cite{hastie2009unsupervised}, \cite{friedman1989regularized} and references therein. We consider a two-class Gaussian mixture model such that each data point $\by_i\in \RR^{p}$ can be represented by
\begin{align}\label{eq:cluster_model}
\by_i=-\ell_i\bbeta+(1-\ell_i)\bbeta+\bepsilon_i\in\RR^{p}
\end{align}
where the associated label $\ell_i\in\{0,1\}$ for $i=1,2,\ldots,n$ is unknown and the noise vector $\bepsilon_i\sim\calN({\bf 0}, \bI_p)$. The vector $\bbeta\in\RR^{p}$ is unknown with $p\gg n$. We denote the true clusters by 
$$
\calN_0:=\{1\leq i\leq n: \ell_i=0\}\quad {\rm and}\quad \calN_1:=\{1\leq i\leq n: \ell_i=1\}.
$$

Given the data matrix
$$
\bY=\big(\by_1,\ldots,\by_n\big)^{\top}\in \RR^{n\times p},
$$
our goal is  bi-clustering the $n$ data points. Let $n_{k+1}:={\rm Card}\big(\calN_k\big)$ for $k=0,1$ such that $n_1+n_2=n$. Observe that $\EE\bY$ has rank $1$ and its leading left singular vector $\bu\in\RR^n$ with
$$
u(i)=\frac{1-\ell_i}{n^{1/2}}-\frac{\ell_i}{n^{1/2}},\quad 1\leq i\leq n.
$$
The signs of $\bu$ immediately suggest the cluster memberships of each data points. Moreover, the leading singular value of $\EE\bY$ is $n^{1/2}\|\bbeta\|_{\ell_2}$. Let $\hat{\bu}$ denote the leading left singular vector of $\bY$. By Corollary~\ref{cor:linfty}, if $\|\bbeta\|_{\ell_2}\geq D_1\big(1\vee(p/n)^{1/4}\big)$ such that $|(1+b_k)^{-1/2}-1|\leq 1/2$, then
$$
\PP\bigg(\big\|\hat{\bu}-(1+b_k)^{1/2}\bu\big\|_{\ell_\infty}\leq D_2\Big(\frac{1}{\|\bbeta\|_{\ell_2}}+\frac{(p/n)^{1/2}}{\|\bbeta\|_{\ell_2}^2}\Big)\Big(\frac{1}{\|\bbeta\|^2_{\ell_2}}+\sqrt{\frac{\log n}{n}}\Big) \bigg)\geq 1-\frac{1}{n}.
$$
On this event, if $\|\bbeta\|_{\ell_2}\geq D_1\Big(n^{1/6}\vee p^{1/8}\vee \big(p\log(n)/n\big)^{1/4}\Big)$
\begin{gather}
\|\hat{\bu}-\bu\|_{\ell_\infty}\leq \|\hat{\bu}-(1+b_k)^{1/2}\bu\|_{\ell_{\infty}}+\big|(1+b_k)^{-1/2}-1\big|\|\bu\|_{\ell_\infty}\nonumber\\
\leq \|\hat{\bu}-(1+b_k)^{1/2}\bu\|_{\ell_{\infty}}+\frac{1}{2n^{1/2}}\leq \frac{3}{4n^{1/2}}\label{eq:hatu-u_linf}
\end{gather}
implying that if $\ell_i=\ell_j$, then ${\rm sign}\big(\hat{u}(i)\big)={\rm sign}\big(\hat{u}(j)\big)$ for all $1\leq i, j\leq n$. Therefore, we propose a simple clustering algorithm by entry-wise signs of $\hat\bu$ in Algorithm~\ref{alg:clustering}.

\begin{algorithm}
\caption{{High dimensional bi-clustering by entry-wise signs.}}
\label{alg:clustering}
\begin{algorithmic}[2]
\STATE {\bf Input:} Data matrix $\bY\in\RR^{n\times p}$
\STATE Calculate the leading left singular vector of $\bY$, denoted by $\hat \bu\in\RR^n$
\STATE Initiate $\hat\calN_0=\{\}$ and $\hat\calN_1=\{\}$
\FOR{$i=1,\cdots,n$}
 \IF{$\hat u(i)\geq 0$}
\STATE $\hat\calN_0\rightarrow \hat\calN_0\cup\{i\}$
\ELSE
\STATE $\hat\calN_1\rightarrow \hat\calN_1\cup \{i\}$
\ENDIF
\ENDFOR
\STATE {\bf Output:} $\hat\calN_0$ and $\hat\calN_1$.
\end{algorithmic}
\end{algorithm}
By the bound (\ref{eq:hatu-u_linf}), Algorithm~\ref{alg:clustering} can guarantee exact clustering as follows. 
\begin{theorem}\label{thm:exact_clustering}
Suppose model (\ref{eq:cluster_model}) holds with noise vector $\beps\sim\calN(0,\bI_p)$.  Let $\hat\calN_0$ and $\hat\calN_1$ be the output of Algorithm~\ref{alg:clustering}. There exists an absolute constant $D_1>0$ such that if $\|\bbeta\|_{\ell_2}\geq D_1\Big(n^{1/6}\vee p^{1/8}\vee \big(p\log(n)/n\big)^{1/4}\Big)$, then with probability at least $1-\frac{1}{n}$,
$$
\hat\calN_0=\calN_0 \quad {\rm or }\quad  \hat\calN_0=\calN_1. 
$$
\end{theorem}
The proof of Theorem~\ref{thm:exact_clustering} is straightforward based on eq. (\ref{eq:hatu-u_linf}). We note that eq. (\ref{eq:hatu-u_linf})
 also implies that it is unnecessary to estimate $b_k$ in this application, since scaling switch the entry-wise signs simultaneously and thus maintains the clustering outputs. 
 
 \begin{remark}
 Theorem~\ref{thm:exact_clustering} reveal unconventional phase transition thresholds for the exact clustering of Gaussian mixture model (\ref{eq:cluster_model}). Indeed, by Theorem~\ref{thm:exact_clustering}, the sup-norm based clustering algorithm (Algorithm~\ref{alg:clustering}) will exactly recover the memberships with high probability when the signal strength satisfies 
 $$
 \|\bbeta\|_{\ell_2}\gg \Big(n^{1/6}\vee p^{1/8}\vee \big(p\log(n)/n\big)^{1/4}\Big).
 $$
In comparison, the $\ell_2$-norm based clustering algorithm in \cite{cai2016rate} and \cite{zheng2015interpolating} requires
$$
\|\bbeta\|_{\ell_2}\gg \big(n^{1/2}\vee p^{1/4}\big)
$$
for exact clustering. Clearly, with respect to exact recovery, the sup-norm based clustering algorithm requires much weaker SNR conditions. 
\end{remark}

\begin{remark}
The above framework can be directly generalized to Gaussian mixture model with $k$-clusters. Suppose that the $j$-th cluster has mean vector $\bbeta_j$ and size $n_j$, then without loss of generality, the data matrix $\bY=\bM+\bZ$
$$
\bM=\big(\underbrace{\bbeta_1,\cdots,\bbeta_1}_{n_1},\cdots,\underbrace{\bbeta_j,\cdots,\bbeta_j}_{n_j},\cdots,\underbrace{\bbeta_k,\cdots,\bbeta_k}_{n_k}\big)^{\top}\in \RR^{N\times p}
$$
with $N=\sum_{j=1}^k n_j$ and $\bZ\in\RR^{N\times p}$ having i.i.d. standard Gaussian entries. Observe that $\rank(\bM)\leq k$, it suffices to consider the top-$k$ left singular vectors of $\bM$. However, it requires nontrivial effort to investigate the eigengaps of $\bM$ without further assumptions on $\{\bbeta_j\}_{j=1}^k$. In the case that $n_j= n$ and $\bbeta_1,\ldots,\bbeta_k$ are mutually orthogonal such that $\|\bbeta_1\|_{\ell_2}\geq\ldots\geq\|\bbeta_{k}\|_{\ell_2}$, then $\bM$'s top-$k$ singular values are $\lambda_j=\sqrt{n_j}\|\bbeta_j\|_{\ell_2}, 1\leq j\leq k$. Clearly, the non-zero entries of $\bM$'s top-$k$ left singular vectors provide the cluster membership of each data points. By Theorem~\ref{thm:linearform}, if $\Delta_{j}\geq C_1\sqrt{k}\|\bbeta_1\|_{\ell_2}+C_2(kp/n)^{1/2}$ where $\Delta_j=\min\{\big(\|\bbeta_j\|_{\ell_2}^2-\|\bbeta_{j+1}\|_{\ell_2}^2\big),\big(\|\bbeta_{j-1}\|_{\ell_2}^2-\|\bbeta_{j}\|_{\ell_2}^2\big)\}$, then
$$
\|\hat\bu_j-\sqrt{1+b_j}\bu_j\|_{\ell_\infty}=O_p\bigg(\Big(\frac{\|\bbeta_1\|_{\ell_2}}{\Delta_j}+\frac{(p/n)^{1/2}}{\Delta_j}\Big)\Big(\frac{k^{3/2}}{\Delta_j}+\sqrt{\frac{k\log n}{n}}\Big)\bigg)
$$
for all $1\leq j\leq k$.
\end{remark}

\subsection{Sub-tensor Localization}
In gene expression association analysis (see \cite{hore2016tensor}, \cite{xiong2012integrating}, \cite{kolar2011minimax} and \cite{ben2003discovering}) and planted clique detection (see \cite{brubaker2009random}, \cite{anandkumar2013tensor} and \cite{gauvin2014detecting}), the goal is equivalent to locating a sub-tensor whose entries are statistically more significant than the others. One simple model characterizing this type of tensor data is as
\begin{equation}\label{eq:subtensor_model}
\bY=\lambda {\bf 1}_{C_1}\otimes {\bf 1}_{C_2}\otimes {\bf 1}_{C_3}+\bZ\in\RR^{d_1\times d_2\times d_3}
\end{equation}
with $C_k=\cup_{j=1}^{s_k}C_k^{(j)}\subset [d_k]$ where $\big\{C_k^{(1)},\ldots,C_k^{(s_k)}\big\}$ are disjoint subsets of $[d_k]$
 for $k=1,2,3$, i.e., there are $s_k\geq 1$ dense blocks in the $k$-th direction. Then, in total, there are $s_1s_2s_3$ dense blocks in $\EE \bY$.
The vector ${\bf 1}_{C_k}\in \RR^{d_k}$ is a zero-or-one vector whose entry equals $1$ only when the index belongs to $\bC_k$. The noise tensor $\bZ$ has i.i.d. entries such that $Z(i,j,k)\sim\calN(0,1)$. Given the noisy observation $\bY$, the goal is to locate the unknown subsets $\{C_1^{(j)}\}_{j=1}^{s_1}, \{C_2^{(j)}\}_{j=1}^{s_2}$ and $\{C_3^{(j)}\}_{j=1}^{s_3}$. The appealing scenario is $\lambda=O(1)$, since otherwise the signal is so strong that the problem can be easily solved by just looking at each entry. The tensor $\EE\bY$ has rank $1$ with leading singular value $\lambda |C_1|^{1/2}|C_2|^{1/2}|C_3|^{1/2}$ and corresponding singular vectors
$$
\bu=\frac{1}{|C_1|^{1/2}}{\bf 1}_{C_1},\quad \bv=\frac{1}{|C_2|^{1/2}}{\bf 1}_{C_2}\quad {\rm and}\quad \bw=\frac{1}{|C_3|^{1/2}}{\bf 1}_{C_3},
$$
where $|C|$ denotes the cardinality of $C$.
By Theorem~\ref{thm:linearform}, if $\lambda\geq D_1\frac{(d_1d_2d_3)^{1/4}}{|C_1|^{1/2}|C_2|^{1/2}|C_3|^{1/2}}$ for a large enough constant $D_1>0$ and $d_{\max}\leq (d_1d_2d_3)^{1/2}$ where $d_{\max}:=(d_1\vee d_2\vee d_3)$, then with probability at least $1-\frac{1}{d_{\max}}$,  we obtain
\begin{align}
\|\hat{\bu}&-(1+b_1)^{1/2}\bu\|_{\ell_{\infty}}\nonumber\\
&\leq \frac{D_1\log^{1/2}d_{\max}}{\lambda|C_1|^{1/2}|C_2|^{1/2}|C_3|^{1/2}}+\frac{D_1(d_2d_3\log d_{\max})^{1/2}}{\lambda^2|C_1||C_2||C_3|}+\frac{D_1d_1}{\lambda^2|C_1||C_2||C_3|}\bigg(\frac{(d_1d_2d_3)^{1/2}}{\lambda^2|C_1||C_2||C_3|}\bigg),\label{eq:subtensor_hatu}
\end{align}
where $b_1\in[-0.5,0]$ is a constant depending on $\bu, \bv,\bw$ and $\lambda$ only. Similar bounds can be also derived for $\hat \bv$ and $\hat \bw$. 
By eq. (\ref{eq:subtensor_hatu}), we propose a simple algorithm (Algorithm~\ref{alg:subtensor}) for the support recovery of sub-tensor model (\ref{eq:subtensor_model}). 
\begin{algorithm}
\caption{{Sub-tensor localizations by entry-wise magnitudes.}}
\label{alg:subtensor}
\begin{algorithmic}[2]
\STATE {\bf Input:} Data matrix $\bY\in\RR^{d_1\times d_2\times d_3}$
\STATE Calculate the leading left singular vectors of $\{\calM_k(\bY)\}_{k=1}^3$, denoted by $\hat \bu\in\RR^{d_1}, \hat \bv\in\RR^{d_2}$ and $\hat\bw\in\RR^{d_3}$, respectively. 
\STATE Take entry-wise magnitudes $\{|\hat u(j_1)|\}_{j_1=1}^{d_1}$ and arrange them in a non-increasing order,
\STATE Record the top-$|C_1|$ locations and denote them by $\hat C_1$;
\STATE Take entry-wise magnitudes $\{|\hat v(j_2)|\}_{j_2=1}^{d_2}$ and arrange them in a non-increasing order,
\STATE Record the top-$|C_2|$ locations and denote them by $\hat C_2$;
\STATE Take entry-wise magnitudes $\{|\hat w(j_3)|\}_{j_3=1}^{d_3}$ and arrange them in a non-increasing order,
\STATE Record the top-$|C_3|$ locations and denote them by $\hat C_3$;
\STATE {\bf Output:} $\hat C_1$, $\hat C_2$ and $\hat C_3$.
\end{algorithmic}
\end{algorithm}
By bound (\ref{eq:subtensor_hatu}), we can immediately guarantee the exact support recovery by Algorithm~\ref{alg:subtensor}. The proof is straightforward and is omitted here. 
\begin{theorem}\label{thm:subtensor}
Suppose model (\ref{eq:subtensor_model}) holds and $(d_1+d_2+d_3)\leq 2(d_1d_2d_3)^{1/2}$. There exist absolute constants $D_1, D_2>0$ such that if $\lambda\geq D_1\frac{(d_1d_2d_3)^{1/4}}{(|C_1||C_2||C_3|)^{1/2}}$ and 
$$
\max\bigg\{\sqrt{\frac{|C_1|}{d_1}}, \sqrt{\frac{|C_2|}{d_2}}, \sqrt{\frac{|C_3|}{d_3}} \bigg\}\cdot \frac{(d_1d_2d_3\log d_{\max})^{1/2}}{\lambda^2|C_1||C_2||C_3|}\leq \frac{1}{D_2},
$$
then, with probability at least $1-\frac{1}{d_1+d_2+d_3}$, we get
$$
\hat C_1=C_1\quad {\rm and}\quad \hat C_2=C_2\quad {\rm and}\quad \hat C_3=C_3
$$
where $\{\hat C_k\}_{k=1}^3$ are the output of Algorithm~\ref{alg:subtensor}. 
\end{theorem}
Note that in Algorithm~\ref{alg:subtensor} and Theorem~\ref{thm:subtensor}, it is also unnecessary to estimate the bias $b_1$ because we are interested in the top-$|C_1|$ largest entries of $|\hat{\bu}|$ and scaling does not affect the ordering of the entry-wise magnitudes. 

\begin{remark}\label{rmk:subtensor}
The phase transition of Algorithm~\ref{alg:subtensor} and model (\ref{eq:subtensor_model}) is intriguing. Note that the support localizations are trivial when $\lambda\gg 1$. Therefore, we only focus on the case $\lambda=1$. 
 Now, let $|C_1|\asymp |C_2|\asymp |C_3|=K$ and $d_1\asymp d_2\asymp d_3=d$. By Theorem~\ref{thm:subtensor}, we conclude that Algorithm~\ref{alg:subtensor} can exactly recover the supports $C_1,C_2,C_3$ with high probability if the support size $K\gg d^{\frac{1}{2}}$. Meanwhile, by the lower bound arguments in \cite{zhang2017tensor}, we know that if $K\ll d^{\frac{1}{2}}$, then there exist no polynomial time algorithms which can recover $C_1$ consistently. Put it differently, phase transition occurs at the threshold $O(d^{\frac{1}{2}})$ such that if $K\ll d^{\frac{1}{2}}$, the problem is unsolvable by polynomial time algorithms; if $K\gg d^{\frac{1}{2}}$, the problem can be perfectly solved by Algorithm~\ref{alg:subtensor}. 
 In comparison, the $\ell_2$-norm based algorithms can only guarantee the consistency of support recovery when $K\gg d^{\frac{1}{2}}$, rather than the exact recovery.  
\end{remark}

\begin{remark}\label{rmk:subtensor_loc}
We could also investigate the entry-wise denoising of model (\ref{eq:subtensor_model}). Suppose that $|C_1|\asymp |C_2|\asymp |C_3|=K$ and $d_1\asymp d_2\asymp d_3=d$. 
We denote by $\bA= {\bf 1}_{C_1}\otimes {\bf 1}_{C_2}\otimes {\bf 1}_{C_3}$ where we fix $\lambda=1$ and we focus only on the support sizes $\{|C_k|\}_{k=1}^3$. Let $\hat \bu, \hat\bv $ and $\hat\bw$ be the empirical singular vectors as in Algorithm~\ref{alg:subtensor}. Define the projection estimator
$$
\hat \bA = \bY\times_1 (\hat \bu \hat\bu^{\top})\times_2(\hat\bv \hat\bv^{\top})\times_3 (\hat\bw\hat\bw^{\top}).
$$
Similarly as in Theorem~\ref{thm:linfty}, we can show that there exists a constant $b\in[\sqrt{2}/4,1]$ such that with probability at least $1-\frac{1}{d}$, 
\begin{equation}\label{eq:subtensor_inf_1}
\|\hat\bA-b\cdot \bA \|_{\ell_\infty}\leq D_1\cdot \Big(\frac{1}{K}+\frac{d}{K^{5/2}}\Big)\log^{3/2}d
\end{equation}
for some absolute constant $D_1>0$. 
Recall from model (\ref{eq:subtensor_model}) that $A(j_1,j_2,j_3)=1$ if $(j_1,j_2,j_3)\in C_1\times C_2\times C_3$. From eq. (\ref{eq:subtensor_inf_1}), we conclude that if $K\geq D_2\big(\sqrt{d}+d^{0.4}\log^{0.6}d\big)$ for a large enough absolute constant $D_2>0$ (note that the threshold $\sqrt{d}$ comes from SNR requirement as in eq. (\ref{eq:subtensor_hatu})), then 
$$
\big| \hat A(j_1,j_2,j_3)\big|>|\hat A(j_1',j_2',j_3')|
$$
for all $(j_1,j_2,j_3)\in C_1\times C_2\times C_3$ and $(j_1', j_2', j_3')\notin C_1\times C_2\times C_3$. As a result, we can choose the locations of $\hat \bA$'s entries with the largest-$|C_1||C_2||C_3|$ magnitudes and recover $\bA$'s supports exactly. 
\end{remark}

\subsection{Numerical Experiments}\label{sec:numeric}
We present simulation results of experiments for the applications in Section~\ref{sec:app}. For high dimensional clustering in model (\ref{eq:cluster_model}), we randomly sample a vector $\bbeta\in\RR^{p}$ with $p=3200$. Fixed a $\bbeta$, we sample $n_1=n/2=800$ random vectors from distribution $\calN(\bbeta,\bI_{p})$ and $n_2=n/2=800$ random vectors from distribution $\calN(-\bbeta,\bI_p)$. 
Then, we calculate the top left singular vector of $\bY$ as in (\ref{eq:cluster_model}) and apply Algorithm~\ref{alg:clustering} to cluster the $1600$ points into two disjoint groups. For each $\bbeta$, we repeat the experiments for $50$ times and the average mis-clustering rate is recorded. The signal strengths are chosen so that $\|\bbeta\|_{\ell_2}=n^{\alpha}$ with $\alpha=0.06*k-0.5$ for $1\leq k\leq 20$. The average mis-clustering rates with respect to signal strengths are displayed in Figure~(\ref{fig:sub1}). Moreover, in Figure~(\ref{fig:sub1}), we also compare the average mis-clustering rates when two clusters have different sizes such as $3n_1=n_2=1200$ and $9n_1=n_2=1440$. As shown in Figure~(\ref{fig:sub1}), there exists a threshold around $\alpha=0.18$ such that the mis-clustering rates by Algorithm~\ref{alg:clustering} decreases extremely fast when the signal strength exceeds the threshold. Meanwhile, Figure~(\ref{fig:sub1}) also shows that the size balances of two clusters does not affect the threshold.  Both these numerical observations from Figure~(\ref{fig:sub1}) are consistent with the theoretic guarantees from Theorem~\ref{thm:exact_clustering}.

For sub-tensor localizations in model (\ref{eq:subtensor_model}), we fix $\lambda=1$ because the support localization task is trivial if $\lambda\gg 1$.  Similarly as in Remark~\ref{rmk:subtensor}, it then suffices to investigate the efficiency of Algorithm~\ref{alg:subtensor} with respect to the support sizes.
For simplicity, we choose
 $d_1=d_2=d_3$ and $C_1=C_2=C_3=[|C_1|]$, that is, the sub-tensor is in the bottom-left-front corner of $\EE \bY$. For each $d_1=150, d_1=200$ and $d_1=300$, we show the average mis-localization rates by Algorithm~\ref{alg:subtensor} with respect to the support size $|C_1|$. The average mis-localization rates are calculated from $50$ independent experiments. The support sizes are chosen as $|C_1|=\ceil{d_1^{\alpha}}$ with $0.06\leq \alpha\leq 1$.   The results of mis-localization rates are displayed in Figure~(\ref{fig:sub2}). Indeed, Figure~(\ref{fig:sub2}) shows that the mis-localization rates by Algorithm~\ref{alg:subtensor} starts to decrease extremely fast when the support size is around $|C_1|\asymp d_1^{0.6}$. The exponent $0.6$ is somewhat larger than the threshold $0.5$ claimed in Remark~\ref{rmk:subtensor}. Note that the dimension  size $d$ is moderately large (only $300$) in our simulations due to the heavy computational cost. 
\begin{figure}
\centering
\begin{subfigure}{.5\textwidth}
  \centering
  \includegraphics[width=0.95\linewidth]{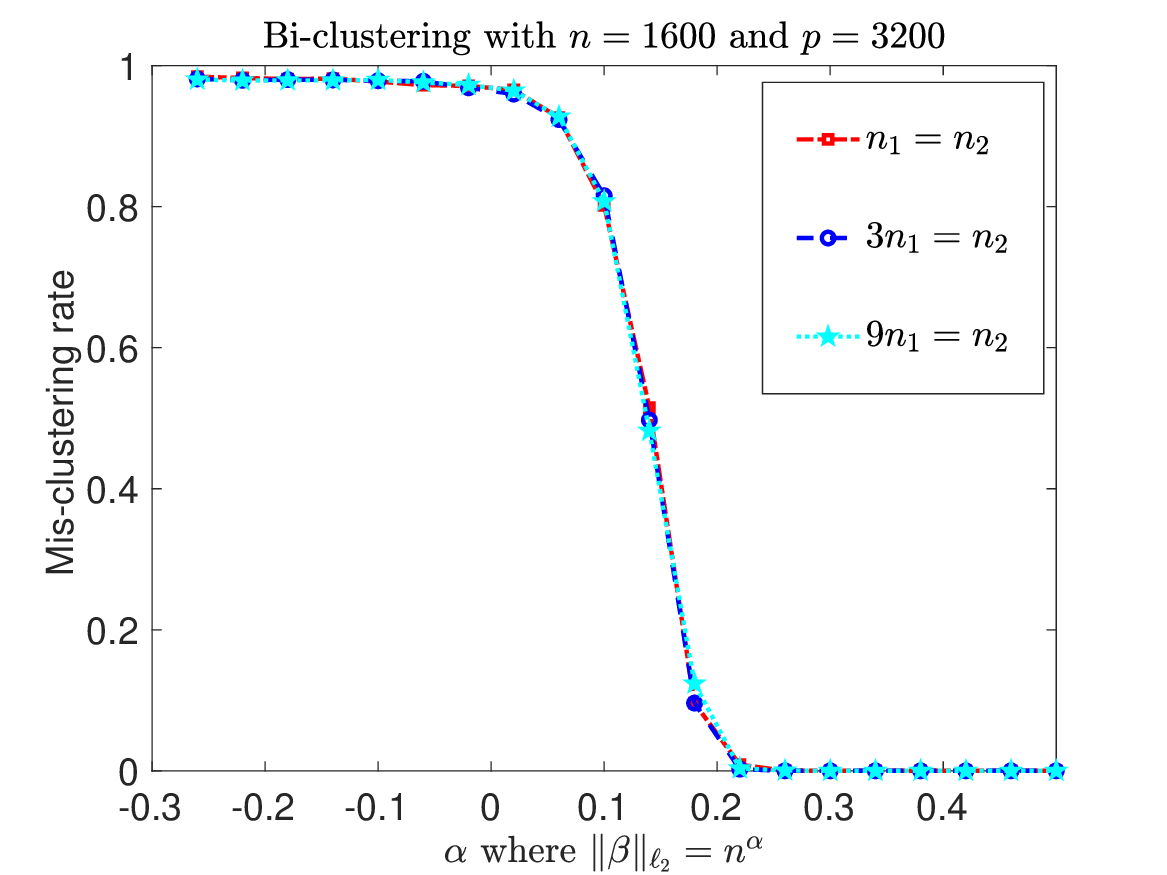}
  \caption{Mis-clustering rates of Algorithm~\ref{alg:clustering}}
  \label{fig:sub1}
\end{subfigure}%
\begin{subfigure}{.5\textwidth}
  \centering
  \includegraphics[width=0.95\linewidth]{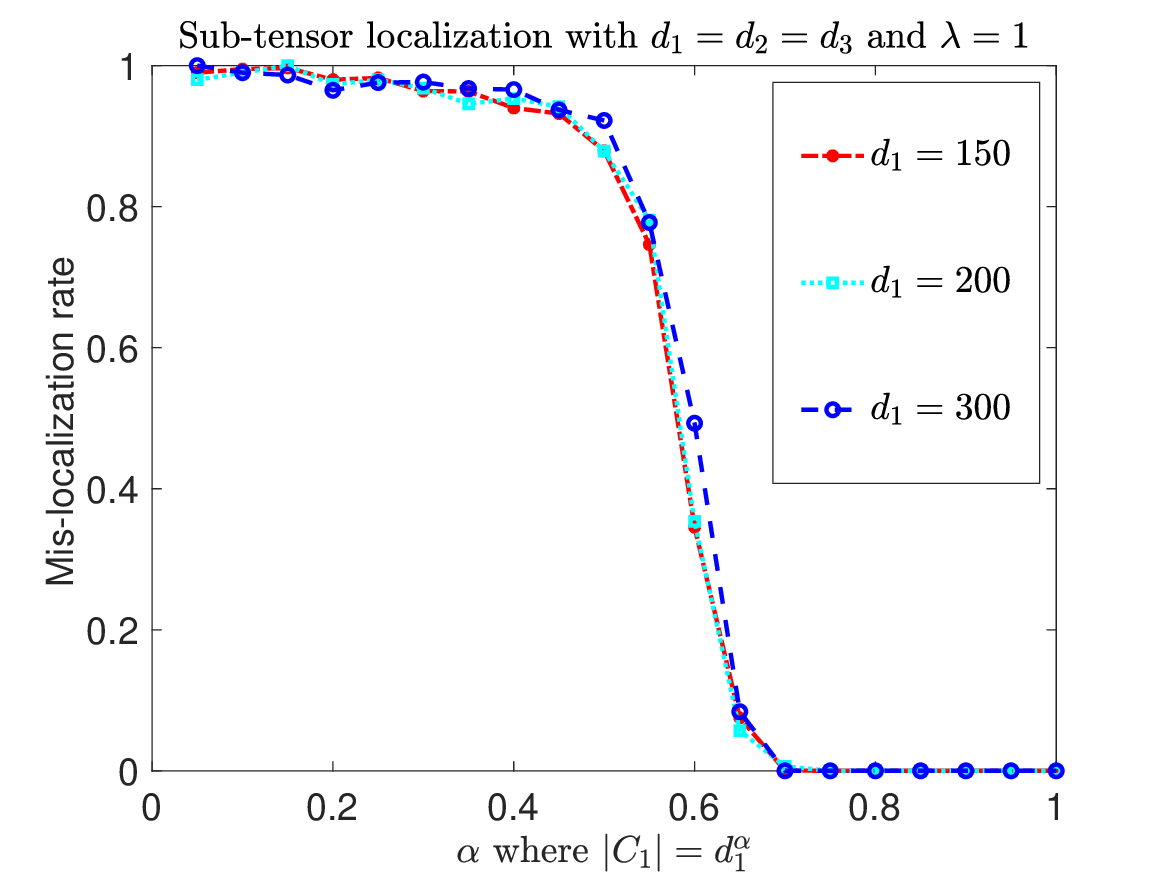}
  \caption{Mis-localization rates of Algorithm~\ref{alg:subtensor}}
  \label{fig:sub2}
\end{subfigure}
\caption{Simulation results for the performances of Algorithm~\ref{alg:clustering} and Algorithm~\ref{alg:subtensor}. In Figure~(\ref{fig:sub1}), the phase transition happens around the signal strength $\|\bbeta\|_{\ell_2}\approx n^{0.18}$ which coincides with Theorem~\ref{thm:exact_clustering}. Figure~(\ref{fig:sub1}) shows that Algorithm~\ref{alg:clustering} can exactly recover the true clusters when signal strength exceeds the aforementioned threshold. Figure~(\ref{fig:sub1}) also shows that the efficiency of Algorithm~\ref{alg:clustering} is unaffected when two clusters have unbalanced sizes. In Figure~(\ref{fig:sub2}), the phase transition happens when the support $C_1$ has size around $d_1^{0.6}$. It shows that Algorithm~\ref{alg:subtensor} can exactly locate the sub-tensor when the support size exceeds the aforementioned threshold. }
\label{fig:sim}
\end{figure}


\section{Proofs}\label{sec:proofs}
For notational brevity, we write $A\lesssim B$ if there exists an absolute constant $D_1$ such that $A\leq D_1B$. A similar notation would be $\gtrsim$ and $A\asymp B$ means that $A\lesssim B$ and $A\gtrsim B$ simultaneously. If the constant $D_1$ depends on some parameter $\gamma$, we shall write $\lesssim_{\gamma}, \gtrsim_{\gamma}$ and $\asymp_{\gamma}$.

Recall that the HOSVD is translated directly from SVD on $\calM_1(\bA)$ and the matrix perturbation model $\calM_1(\bY)=\calM_1(\bA)+\calM_1(\bZ)$. Without loss of generality, it suffices to focus on matrices with unbalanced sizes. In the remaining context, we write $\bA, \bZ, \bY\in\RR^{m_1\times m_2}$ instead of $\calM_1(\bA), \calM_1(\bZ), \calM_1(\bY)\in\RR^{m_1\times m_2}$, where $m_1=d_1$ and $m_2=d_2d_3$ such that $m_1\ll m_2$. The second order spectral analysis begins with
$$
\bY\bY^{\top}=\bA\bA^{\top}+\bGamma,\quad{\rm where}\quad \bGamma=\bA\bZ^{\top}+\bZ\bA^{\top}+\bZ\bZ^{\top}.
$$
Suppose that $\bA$ has the thin singular value decomposition
$$
\bA=\sum_{k=1}^{r_1}\lambda_k\big(\bu_k\otimes \bh_k\big)\in\RR^{m_1\times m_2}
$$
where $\{\bh_1,\ldots,\bh_{r_1}\}\subset {\rm span}\big\{\bv_j\otimes \bw_k^{\top}: j\in[r_2], k\in[r_3]\big\}$ are the right singular vectors of $\bA$. Moreover, $\bA\bA^{\top}$ admits the eigen-decomposition:
$$
\bA\bA^{\top}=\sum_{k=1}^{r_1}\lambda_k^2 \big(\bu_k\otimes \bu_k\big).
$$
In an identical fashion, denote the eigen-decomposition of $\bY\bY^{\top}$ by
$$
\bY\bY^{\top}=\sum_{k=1}^{m_1} \hat{\lambda}_k^2\big(\hat{\bu}_k\otimes \hat{\bu}_k\big).
$$
Even though Theorem~\ref{thm:linearform} and Theorem~\ref{thm:tildeu} are stated when the singular value $\lambda_k$ has multiplicity $1$, we present more general results in this section. Note that when there are repeated singular values, the singular vectors are not uniquely defined. In this case, let $\mu_1>\mu_2>\ldots>\mu_{s}>0$ be distinct singular values of $\bA$ with $s\leq r_1$. Denote $\Delta_k:=\{j: \lambda_j=\mu_k\}$ for $1\leq k\leq s$ and $\nu_k:={\rm Card}(\Delta_k)$ the multiplicity of $\mu_k$. Let $\mu_{s+1}=0$ which is a trivial eigenvalue of $\bA\bA^{\top}$ with multiplicity $m_1-r_1$. Then, the spectral decomposition of $\bA\bA^{\top}$ can be represented as
$$
\bA\bA^{\top}=\sum_{k=1}^{s+1}\mu_k^2 \bP_k^{uu}
$$
where the spectral projector $\bP_k^{uu}:=\sum_{j\in\Delta_k}\bu_j\otimes \bu_j$ which is uniquely defined. Correspondingly, define the empirical spectral projector based on eigen-decomposition of $\bY\bY^{\top}$,
$$
\hat{\bP}_k^{uu}:=\sum_{j\in\Delta_k}\hat{\bu}_j\otimes \hat{\bu}_j.
$$
We develop a sharp concentration bound for bilinear forms $\big<\hat{\bP}_k^{uu}\bx, \by\big>$ for $\bx, \by \in \RR^{m_1}$. Observe that $\bY\bY^{\top}$ has an identical eigen-space as $\bY\bY^{\top}-m_2\sigma^2\bI_{m_1}$. Let $\hat{\bGamma}:=\bGamma-m_2\sigma^2\bI_{m_1}$ and the spectral analysis shall be realized on $\bA\bA^{\top}+\hat{\bGamma}$. 

Several preliminary facts are introduced as follows. It is clear that the $k$-th eigengap is $\bar{g}_k\big(\bA\bA^{\top}\big):=\min\big(\mu_{k-1}^2-\mu_k^2, \mu_k^2-\mu_{k+1}^2\big)$ for $1\leq k\leq s$, where we set $\mu_0=+\infty$. The proof of Lemma~\ref{lemma:gammabound} is provided in the Appendix.
\begin{lemma}\label{lemma:gammabound}
For any deterministic matrix $\bB\in\RR^{m_3\times m_2}$, the following bounds hold
\begin{align}\label{eq:EAZ}
\mathbb{E}\|\bB\bZ^{\top}\|&\lesssim \sigma\|\bB\|\Big(m_1^{1/2}+m_3^{1/2}+(m_1m_3)^{1/4}\Big)\\
 \big\|\mathbb{E}\bZ\bZ^{\top}-m_2\sigma^2\bI_{m_1}\big\|&\lesssim \sigma^2(m_1m_2)^{1/2}.\nonumber
\end{align}
For any $t>0$, the following inequalities hold with probability at least $1-e^{-t}$,
\begin{align}\label{eq:PAZ}
\|\bB\bZ^{\top}\|&\lesssim \sigma\|\bB\|\Big(m_1^{1/2}+m_3^{1/2}+(m_1m_3)^{1/4}+t^{1/2}+(m_1t)^{1/4}\Big)\\
\big\|\bZ\bZ^{\top}-m_2\sigma^2\bI_{m_1}\big\|&\lesssim \sigma^2m_2^{1/2}\big(m_1^{1/2}+t^{1/2}\big).\nonumber
\end{align}
\end{lemma}

\subsection{Proof of Theorem~\ref{thm:linearform}}
To this end, define
$$
\bC_k^{uu}:=\sum_{s\neq k}\frac{1}{\mu_s^2-\mu_k^2}\bP_{s}^{uu}
$$
and
$$
\bP_k^{hh}:=\sum_{j\in\Delta_k} \bh_j\otimes \bh_j.
$$
Theorem~\ref{thm:linearform} is decomposed of two separate components. Theorem~\ref{thm:xPky} provides the concentration bound for $\big|\langle\bP_k\bx, \by\rangle-\EE\langle\bP_k\bx, \by\rangle\big|$ by Gaussian isoperimetric inequality and the proof is postponed to the Appendix. In Theorem~\ref{thm:bias}, we characterize the bias $\EE\hat{\bP}_k^{uu}-\bP_k^{uu}$.
\begin{theorem}\label{thm:xPky}
Let $\delta(m_1,m_2):=\mu_1\sigma m_1^{1/2}+\sigma^2(m_1m_2)^{1/2}$ and suppose that $\bar{g}_k\big(\bA\bA^{\top}\big)\geq D_1\delta(m_1,m_2)$ for a large enough constant $D_1>0$. Then, for any $\bx,\by\in\RR^{m_1}$, there exists an absolute constant $D_2>0$ such that for all $\log8\leq t\lesssim m_1$, the following bound holds with probability at least $1-e^{-t}$,
\begin{eqnarray*}
\big|\langle\hat{\bP}_k^{uu}\bx, \by\rangle-\EE\langle\hat{\bP}_k^{uu}\bx, \by\rangle\big|\leq D_2t^{1/2}\bigg(\frac{\sigma\mu_1 +\sigma^2m_2^{1/2}}{\bar{g}_k\big(\bA\bA^{\top}\big)}\bigg)\|\bx\|_{\ell_2}\|\by\|_{\ell_2}.
\end{eqnarray*}
\end{theorem}
The following spectral representation formula is needed whose proof can be found in \cite{koltchinskii2016asymptotics}.
\begin{lemma}\label{lemma:hatPk}
The following bound holds
$$
\|\hat{\bP}_k^{uu}-\bP_k^{uu}\|\leq \frac{4\|\hat{\bGamma}\|}{\bar{g}_k(\bA\bA^{\top})}.
$$
Moreover, $\hat{\bP}_k^{uu}$ can be represented as
$$
\hat{\bP}_k^{uu}-\bP_k^{uu}=\bL_k(\hat{\bGamma})+\bS_k(\hat{\bGamma})
$$
where $\bL_k(\hat{\bGamma})=\bP_k^{uu}\hat{\bGamma}\bC_k^{uu}+\bC_k^{uu}\hat{\bGamma}\bP_k^{uu}$ and 
$$
\|\bS_k(\hat{\bGamma})\|\leq 14\bigg(\frac{\|\hat{\bGamma}\|}{\bar{g}_k(\bA\bA^{\top})}\bigg)^2.
$$
\end{lemma}

\begin{theorem}\label{thm:bias}
Let $\delta(m_1,m_2):=\mu_1\sigma m_1^{1/2}+\sigma^2(m_1m_2)^{1/2}$ and suppose that $\bar{g}_k\big(\bA\bA^{\top}\big)\geq D_1\delta(m_1,m_2)$ for a large enough constant $D_1>0$ and $m_2e^{-m_1/2}\leq 1$. Then there exists an absolute constant $D_2>0$ such that
\begin{eqnarray*}
\big\|\EE\hat{\bP}_k^{uu}-\bP_k^{uu}-\bP_k^{uu}\big(\EE\hat{\bP}_k^{uu}-\bP_k^{uu}\big)\bP_k^{uu} \big\|\leq D_2\nu_k\frac{\sigma^2m_1+\sigma^2m_2^{1/2}+\sigma\mu_1}{\bar{g}_k\big(\bA\bA^{\top}\big)}\bigg(\frac{\delta(m_1,m_2)}{\bar{g}_k\big(\bA\bA^{\top}\big)}\bigg).
\end{eqnarray*}
\end{theorem}

\begin{proof}[Proof of Theorem~\ref{thm:linearform}]
Combining Theorem~\ref{thm:xPky} and Theorem~\ref{thm:bias}, we conclude that for any $\bx,\by\in\RR^{m_1}$ with probability at least $1-e^{-t}$ for all $\log8\leq t\leq m_1$, 
\begin{align*}
\big|\big<\hat{\bP}_k^{uu}\bx,\by\big>-&\big<\bP_k^{uu}\bx,\by\big>-\big<\bP_k^{uu}(\EE\hat\bP_k^{uu}-\bP_k^{uu})\bP_k^{uu}\bx,\by\big>\big|\\
&\lesssim \bigg(t^{1/2}\frac{\sigma\mu_1+\sigma^2m_2^{1/2}}{\bar{g}_k(\bA\bA^{\top})}+\frac{\sigma^2m_1\delta(m_1,m_2)}{\bar{g}_k^2(\bA\bA^{\top})}\bigg)\|\bx\|_{\ell_2}\|\by\|_{\ell_2}
\end{align*}
where we used the fact $\frac{\delta(m_1,m_2)}{\bar{g}_k(\bA\bA^{\top})}\leq 1$ and $\nu_k=1$.
Since $\nu_k=1$ such that $\bP_k^{uu}=\bu_k\otimes\bu_k$ and $\hat{\bP}_k^{uu}=\hat{\bu}_k\otimes \hat{\bu}_k$, we can write
$$
\bP_k^{uu}(\EE\hat\bP_k^{uu}-\bP_k^{uu})\bP_k^{uu}=b_k\bP_k^{uu}
$$
where
$$
b_k=\mathbb{E}\langle\hat\bu_k, \bu_k\rangle^2-1\in[-1,0].
$$
Moreover, a simple fact is $b_k\leq \EE\|\hat{\bP}_k^{uu}-\bP_k^{uu}\|\lesssim \frac{\delta(m_1,m_2)}{\bar{g}_k(\bA\bA^{\top})}$ by Wedin's sin$\Theta$ theorem (\cite{wedin1972perturbation}). If $\bar{g}_k(\bA\bA^{\top})\geq D\delta(m_1,m_2)$ for a large enough constant $D>0$, we can ensure $b_k\in[-1/2,0]$. Then, with probability at least $1-e^{-t}$,
$$
\big|\big<\big(\hat{\bP}_k^{uu}-(1+b_k)\bP_k^{uu}\big)\bx,\by\big>\big|
\lesssim \bigg(t^{1/2}\frac{\sigma\mu_1+\sigma^2m_2^{1/2}}{\bar{g}_k(\bA\bA^{\top})}+\frac{\sigma^2m_1\delta(m_1,m_2)}{\bar{g}_k^2(\bA\bA^{\top})}\bigg)\|\bx\|_{\ell_2}\|\by\|_{\ell_2}.
$$
By choosing $\bx=\by=\bu_k$, we obtain for all $\log8\leq t\leq m_1$,
$$
\PP\bigg(\big|\langle\hat\bu_k,\bu_k \rangle^2-(1+b_k)\big|\gtrsim t^{1/2}\frac{\sigma\mu_1+\sigma^2m_2^{1/2}}{\bar{g}_k(\bA\bA^{\top})}+\frac{\sigma^2m_1\delta(m_1,m_2)}{\bar{g}_k^2(\bA\bA^{\top})}\bigg)\leq e^{-t}.
$$
Denote this event by $\calE_1$. Observe that if the constant $D>0$ is large enough and $m_1\ll m_2$, we conclude that on event $\calE_1$, $\langle\hat{\bu}_k, \bu_k\rangle^2\geq \frac{1}{4}$. Then, on event $\calE_1$,
\begin{align*}
\big|\langle\hat{\bu}_k,\bx \rangle-&\sqrt{1+b_k}\langle\bu_k,\bx\rangle\big|\\
\leq& \Big|\frac{1+b_k}{\langle\hat{\bu}_k, \bu_k\rangle}-\sqrt{1+b_k}\Big||\langle\bu_k,\bx \rangle|\\
+&\frac{1}{|\langle\hat{\bu}_k,\bu_k \rangle|}\Big|\langle\hat{\bu}_k,\bu_k\rangle\langle\hat{\bu}_k, \bx\rangle-(1+b_k)\langle\bu_k,\bx \rangle\Big|\\
=&\frac{\sqrt{1+b_k}\big|1+b_k-\langle\hat{\bu}_k, \bu_k\rangle^2\big||\langle\bu_k,\bx\rangle|}{|\langle\hat\bu_k, \bu_k\rangle|\big(\sqrt{1+b_k}+\langle\hat{\bu}_k,\bu_k\rangle\big)}+\frac{1}{|\langle\hat{\bu}_k, \bu_k\rangle|}\big|\big<\big(\hat{\bP}_k^{uu}-(1+b_k)\bP_k^{uu}\big)\bu_k,\bx\big>\big|\\
\lesssim& t^{1/2}\frac{\sigma\mu_1+\sigma^2m_2^{1/2}}{\bar{g}_k(\bA\bA^{\top})}\|\bx\|_{\ell_2}+\frac{\sigma^2m_1}{\bar{g}_k(\bA\bA^{\top})}\bigg(\frac{\delta(m_1,m_2)}{\bar{g}_k(\bA\bA^{\top})}\bigg)\|\bx\|_{\ell_2},
\end{align*}
which concludes the proof after replacing $\bA$ with $\calM_1(\bA)$ and $\mu_1$ with $\|\calM_1(\bA)\|$.
\end{proof}

\begin{proof}[Proof of Theorem~\ref{thm:bias}]
Recall the representation formula of $\hat{\bP}_k^{uu}$ in Lemma~\ref{lemma:hatPk} that 
$$
\mathbb{E}\hat{\bP}_k^{uu}=\bP_k^{uu}+\mathbb{E}\bS_k(\hat{\bGamma})
$$
where $\hat{\bGamma}:=\bA\bZ^{\top}+\bZ\bA^{\top}+\bZ\bZ^{\top}-m_2\sigma^2\bI_{m_1}$. To this end, define
$$
\widetilde{\bGamma}:=\hat{\bGamma}-\big(\bZ\bP_k^{hh}\bZ^{\top}-\nu_k\sigma^2\bI_{m_1}\big)
$$
such that we can write $\EE\hat{\bP}_k^{uu}=\bP_k^{uu}+\EE\bS_k(\widetilde{\bGamma})+\Big(\EE\bS_k(\hat{\bGamma})-\EE\bS_k(\widetilde{\bGamma})\Big)$. We derive an upper bound on $\big\|\EE\bS_k(\widetilde{\bGamma})-\EE\bS_k(\hat{\bGamma})\big\|$ and the proof can be found in the Appendix. Lemma~\ref{lemma:hatS-tildeS} implies that our analysis can be proceeded by replacing $\hat{\bGamma}$ with $\widetilde{\bGamma}$.
\begin{lemma}\label{lemma:hatS-tildeS}
There exists a universal constant $D_1>0$ such that if $m_2e^{-m_1/2}\leq 1$, then
$$
\big\|\EE\bS_k(\widetilde{\bGamma})-\EE\bS_k(\hat{\bGamma})\big\|\leq D_1 \frac{\sigma\mu_1+\sigma^2m_1}{\bar{g}_k(\bA\bA^{\top})}\bigg( \frac{\delta(m_1,m_2)}{\bar{g}_k(\bA\bA^{\top})}\bigg).
$$
\end{lemma} 
Let $\delta_t=\EE\|\hat{\bGamma}\|+D_1\sigma\mu_1 t^{1/2}+D_2\sigma^2m_2^{1/2}t^{1/2}$ for $0<t\leq m_1$ to be determined later and large enough constants $D_1,D_2>0$ such that $\PP\big(\|\hat{\bGamma}\|\geq \delta_t\big)\leq e^{-t}$. We write
\begin{align}
\EE\hat{\bP}_k^{uu}&-\bP_k^{uu}-\bP_k^{uu}\EE\bS_k(\widetilde{\bGamma})\bP_k^{uu}\nonumber\\
=&\EE\bS_{k}(\hat{\bGamma})-\EE\bS_k(\widetilde{\bGamma})\nonumber\\
+&\EE\Big(\bP_k^{uu}\bS_k(\widetilde{\bGamma})(\bP_k^{uu})^{\perp}+(\bP_k^{uu})^{\perp}\bS_k(\widetilde{\bGamma})\bP_k^{uu}+(\bP_k^{uu})^{\perp}\bS_k(\widetilde{\bGamma})(\bP_k^{uu})^{\perp}\Big){\bf 1}\big(\|\widetilde{\bGamma}\|\leq \delta_t\big)\nonumber\\
+&\EE\Big(\bP_k^{uu}\bS_k(\widetilde{\bGamma})(\bP_k^{uu})^{\perp}+(\bP_k^{uu})^{\perp}\bS_k(\widetilde{\bGamma})\bP_k^{uu}+(\bP_k^{uu})^{\perp}\bS_k(\widetilde{\bGamma})(\bP_k^{uu})^{\perp}\Big){\bf 1}\big(\|\widetilde{\bGamma}\|> \delta_t\big).\label{eq:hatP-P}
\end{align}
We prove an upper bound for $\mathbb{E}\big\langle\bx, (\bP_{k}^{uu})^{\perp}\bS_k(\widetilde{\bGamma})\bP_k^{uu}\by\big\rangle{\bf 1}\big(\|\widetilde{\bGamma}\|\leq\delta_t\big)$ for $\bx,\by\in\RR^{m_1}$. Similar to the approach in \cite{koltchinskii2016perturbation}, under the assumption $\|\widetilde{\bGamma}\|\leq \delta_t$, $\bS_k(\widetilde{\bGamma})$ is represented in the following analytic form,
$$
\bS_k(\widetilde{\bGamma})=-\frac{1}{2\pi i}\oint_{\gamma_k}\sum_{r\geq 2}(-1)^r\Big(\bR_{\bA\bA^{\top}}(\eta)\widetilde{\bGamma}\Big)^r\bR_{\bA\bA^{\top}}(\eta)d\eta
$$
where $\gamma_k$ is a circle on the complex plane with center $\mu_k^2$ and radius $\frac{\bar{g}_k(\bA\bA^{\top})}{2}$, and $\bR_{\bA\bA^{\top}}(\eta)$ is the resolvent of the operator $\bA\bA^{\top}$ with $\bR_{\bA\bA^{\top}}(\eta)=(\bA\bA^{\top}-\eta\bI_{m_1})^{-1}$ 
which can be explicitly written as
$$
\bR_{\bA\bA^{\top}}(\eta):=(\bA\bA^{\top}-\eta\bI_{m_1})^{-1}=\sum_{s}\frac{1}{\mu_s^2-\eta}\bP_s^{uu}.
$$
We also denote
$$
\widetilde{\bR}_{\bA\bA^{\top}}(\eta):={\bR}_{\bA\bA^{\top}}(\eta)-\frac{1}{\mu_k^2-\eta}\bP_k^{uu}=\sum_{s\neq k}\frac{1}{\mu_s^2-\eta}\bP_s^{uu}.
$$
It is easy to check that 
\begin{align*}
(\bP_k^{uu})^{\perp}	&\big(\bR_{\bA\bA^{\top}}(\eta)\widetilde{\bGamma}\big)^r\bR_{\bA\bA^{\top}}(\eta)\bP_k^{uu}\\
=&(\bP_k^{uu})^{\perp}\big(\bR_{\bA\bA^{\top}}(\eta)\widetilde{\bGamma}\big)^r\frac{1}{\mu_k^2-\eta}\bP_k^{uu}\\
&=\bigg(\frac{1}{(\mu_k^2-\eta)^2}
\sum_{s=2}^r\big(\widetilde{\bR}_{\bA\bA^{\top}}(\eta)\widetilde{\bGamma}\big)^{s-1}\big(\bP_k^{uu}\widetilde{\bGamma}\big)\big(\bR_{\bA\bA^{\top}}(\eta)\widetilde{\bGamma}\big)^{r-s}\bP_k^{uu}\bigg)\\
&+\frac{1}{\mu_k^2-\eta}\big(\widetilde{\bR}_{\bA\bA^{\top}}(\eta)\widetilde{\bGamma}\big)^{r}\bP_k^{uu},
\end{align*}
where we used the formula $(a+b)^r=b^r+\sum_{s=1}^rb^{s-1}a(a+b)^{r-s}$.
As a result, 
\begin{align}
&(\bP_k^{uu})^{\perp}\bS_k(\widetilde{\bGamma})\bP_k^{uu}\nonumber\\
&=-\sum_{r\geq 2}(-1)^r\frac{1}{2\pi i}\oint_{\gamma_k}\bigg(\frac{1}{(\mu_k^2-\eta)^2}
\sum_{s=2}^r\big(\widetilde{\bR}_{\bA\bA^{\top}}(\eta)\widetilde{\bGamma}\big)^{s-1}\big(\bP_k^{uu}\widetilde{\bGamma}\big)\big(\bR_{\bA\bA^{\top}}(\eta)\widetilde{\bGamma}\big)^{r-s}\bP_k^{uu}\nonumber\\
&\quad +\frac{1}{\mu_k^2-\eta}\big(\widetilde{\bR}_{\bA\bA^{\top}}(\eta)\widetilde{\bGamma}\big)^{r}\bP_k^{uu}\bigg)d\eta.\label{eq:PkSkPk}
\end{align}
For any $\bx,\by\in\RR^{m_1}$, we shall derive an upper bound for
$$
\EE\Big<\bx, \big(\widetilde{\bR}_{\bA\bA^{\top}}(\eta)\widetilde{\bGamma}\big)^{s-1}\big(\bP_k^{uu}\widetilde{\bGamma}\big)\big(\bR_{\bA\bA^{\top}}(\eta)\widetilde{\bGamma}\big)^{r-s}
\bP_k^{uu}\by\Big>{\bf 1}\big(\|\widetilde{\bGamma}\|\leq \delta_t\big),\quad s=2,\ldots,r.
$$
Recall that $\rank(\bP_k^{uu})=\nu_k$ and $\bP_k^{uu}=\sum_{j\in\Delta_k}\bu_{j}\otimes \bu_{j}$. Then,
\begin{align*}
 \Big<\bx, \big(\widetilde{\bR}_{\bA\bA^{\top}}&(\eta)\widetilde{\bGamma}\big)^{s-1}\big(\bP_k^{uu}\widetilde{\bGamma}\big)\big(\bR_{\bA\bA^{\top}}(\eta)\widetilde{\bGamma}\big)^{r-s}\bP_k^{uu}\by\Big>\\
  =& \sum_{j\in\Delta_k}\Big<\bx,\big(\widetilde{\bR}_{\bA\bA^{\top}}(\eta)\widetilde{\bGamma}\big)^{s-1}\big(\bu_{j}\otimes \bu_{j}\widetilde{\bGamma}\big)\big(\bR_{\bA\bA^{\top}}(\eta)\widetilde{\bGamma}\big)^{r-s}\bP_k^{uu}\by\Big>\\
 =&\sum_{j\in\Delta_k}\big<\widetilde{\bGamma}\big(\bR_{\bA\bA^{\top}}(\eta)\widetilde{\bGamma}\big)^{r-s}\bP_k^{uu}\by,\bu_{j}\big> \big<\big(\widetilde{\bR}_{\bA\bA^{\top}}(\eta)\widetilde{\bGamma}\big)^{s-2}\widetilde{\bR}_{\bA\bA^{\top}}(\eta)\widetilde{\bGamma}\bu_{j},\bx\big>.
 \end{align*}
Observe that 
\begin{align*}
\big|\big<\widetilde{\bGamma}\big(\bR_{\bA\bA^{\top}}(\eta)\widetilde{\bGamma}\big)^{r-s}\bP_k^{uu}\by,\bu_{j}\big>\big|\leq& \|\bR_{\bA\bA^{\top}}(\eta)\|^{r-s}\|\widetilde{\bGamma}\|^{r-s+1}\|\by\|_{\ell_2}\\
\leq& \Big(\frac{2}{\bar{g}_k(\bA\bA^{\top})}\Big)^{(r-s)}\|\widetilde{\bGamma}\|^{r-s+1}\|\by\|_{\ell_2}.
\end{align*}
Therefore,
\begin{align}
 &\mathbb{E}\Big<\bx, \big(\widetilde{\bR}_{\bA\bA^{\top}}(\eta)\widetilde{\bGamma}\big)^{s-1}\big(\bP_k^{uu}\widetilde{\bGamma}\big)\big(\bR_{\bA\bA^{\top}}(\eta)\widetilde{\bGamma}\big)^{r-s}\bP_k^{uu}\by\Big>{\bf 1}\Big(\|\widetilde{\bGamma}\|\leq \delta_t\Big)\nonumber\\
 &=\sum_{j\in\Delta_k}\mathbb{E}\big<\widetilde{\bGamma}\big(\bR_{\bA\bA^{\top}}(\eta)\widetilde{\bGamma}\big)^{r-s}\bP_k^{uu}\by,\bu_{j}\big> \big<\big(\widetilde{\bR}_{\bA\bA^{\top}}(\eta)\widetilde{\bGamma}\big)^{s-1}\bu_{j},\bx\big>{\bf 1}\Big(\|\widetilde{\bGamma}\|\leq \delta_t\Big)\nonumber\\
 &\leq \sum_{j\in\Delta_k}\EE^{1/2}\Big|\big<\widetilde{\bGamma}\big(\bR_{\bA\bA^{\top}}(\eta)\widetilde{\bGamma}\big)^{r-s}\bP_k^{uu}\by,\bu_{j}\big>{\bf 1}\Big(\|\widetilde{\bGamma}\|\leq \delta_t\Big)\Big|^2\nonumber\\
 &\hspace{3cm}\times\EE^{1/2}\Big|\big<\big(\widetilde{\bR}_{\bA\bA^{\top}}(\eta)\widetilde{\bGamma}\big)^{s-1}\bu_{j},\bx\big>{\bf 1}\Big(\|\widetilde{\bGamma}\|\leq \delta_t\Big)\Big|^2\nonumber\\
 &\leq \Big(\frac{2\delta_t}{\bar{g}_k(\bA\bA^{\top})}\Big)^{r-s}\delta_t\|\by\|_{\ell_2}\sum_{j\in\Delta_k}\EE^{1/2}\Big|\big<\big(\widetilde{\bR}_{\bA\bA^{\top}}(\eta)\widetilde{\bGamma}\big)^{s-2}\widetilde{\bR}_{\bA\bA^{\top}}(\eta)\widetilde{\bGamma}\bu_{j},\bx\big>{\bf 1}\Big(\|\widetilde{\bGamma}\|\leq \delta_t\Big)\Big|^2. \label{eq:xRy}
\end{align}
It then remains to bound, for each $j\in\Delta_k$,
$$
\mathbb{E}^{1/2}\Big|\Big<\big(\widetilde{\bR}_{\bA\bA^{\top}}(\eta)\widetilde{\bGamma}\big)^{s-2}\widetilde{\bR}_{\bA\bA^{\top}}(\eta)\widetilde{\bGamma}\bu_{j},\bx\Big>\Big|^2{\bf 1}\Big(\|\widetilde{\bGamma}\|\leq \delta_t\Big).
$$
Recall that we can write
$$
\widetilde{\bGamma}=\bA\bZ^{\top}+\bZ\bA^{\top}+\bZ\sum_{k'\neq k}\bP_{k'}^{hh}\bZ^{\top}-\sigma^2(m_2-\nu_k)\bI_{m_1}
$$
and correspondingly
$$
\widetilde{\bGamma}\bu_{j}=\bA\bZ^{\top}\bu_{j}+\bZ\bA^{\top}\bu_{j}+\bZ\sum_{k'\neq k}\bP_{k'}^{hh}\bZ^{\top}\bu_{j}-\sigma^2(m_2-\nu_k)\bu_{j}.
$$
We write
\begin{align}
\Big<\big(\widetilde{\bR}_{\bA\bA^{\top}}&(\eta)\widetilde{\bGamma}\big)^{s-2}\widetilde{\bR}_{\bA\bA^{\top}}(\eta)\widetilde{\bGamma}\bu_{j},\bx\Big>\nonumber\\
=&\Big<\big(\widetilde{\bR}_{\bA\bA^{\top}}(\eta)\widetilde{\bGamma}\big)^{s-2}\widetilde{\bR}_{\bA\bA^{\top}}(\eta)\bZ\bA^{\top}\bu_{j},\bx\Big>\label{eq:ZA}\\
&+\Big<\big(\widetilde{\bR}_{\bA\bA^{\top}}(\eta)\widetilde{\bGamma}\big)^{s-2}\widetilde{\bR}_{\bA\bA^{\top}}(\eta)\bA\bZ^{\top}\bu_{j},\bx\Big>\label{eq:AZ}\\
&+\Big<\big(\widetilde{\bR}_{\bA\bA^{\top}}(\eta)\widetilde{\bGamma}\big)^{s-2}\widetilde{\bR}_{\bA\bA^{\top}}(\eta)\Big(\bZ\sum_{k'\neq k}\bP_{k'}^{hh}\bZ^{\top}\bu_{j}-\sigma^2(m_2-\nu_k)\bu_{j}\Big),\bx\Big>.\label{eq:ZZ}
\end{align}
The upper bounds of (\ref{eq:ZA}), (\ref{eq:AZ}), and (\ref{eq:ZZ})  shall be obtained separately via different representations.

\paragraph{Bound of\ $\mathbb{E}^{1/2}\big|\big<\big(\widetilde{\bR}_{\bA\bA^{\top}}(\eta)\widetilde{\bGamma}\big)^{s-2}\widetilde{\bR}_{\bA\bA^{\top}}(\eta)\bZ\bA^{\top}\bu_{j},\bx\big>\big|^2{\bf 1}\Big(\|\widetilde{\bGamma}\|\leq \delta_t\Big)$.} Observe that $\bA^{\top}\bu_{j}=\mu_k\bh_{j}\in\RR^{m_2}$ for $j\in\Delta_k$ such that
$$
\bZ\bA^{\top}\bu_{j}=\mu_k\bZ\bh_{j}=\mu_k\sum_{i=1}^{m_1}\langle \bz_i, \bh_{j}\rangle\be_i
$$
where $\{\be_1,\ldots,\be_{m_1}\}$ denote the canonical basis vectors in $\RR^{m_1}$ and $\{\bz_1^{\top},\ldots,\bz_{m_1}^{\top}\}$ denote the rows of $\bZ$. Therefore,
\begin{align*}
\big<\big(\widetilde{\bR}_{\bA\bA^{\top}}(\eta)\widetilde{\bGamma}\big)^{s-2}\widetilde{\bR}_{\bA\bA^{\top}}(\eta)&\bZ\bA^{\top}\bu_{j},\bx\big>\\
=&\mu_k\sum_{i=1}^{m_1}\langle\bz_i,\bh_{j}\rangle\big<\big(\widetilde{\bR}_{\bA\bA^{\top}}(\eta)\widetilde{\bGamma}\big)^{s-2}\widetilde{\bR}_{\bA\bA^{\top}}(\eta)\be_i,\bx\big>.
\end{align*}
It is clear that $\langle\bz_i, \bh_{j}\rangle, i=1,\ldots,m_1$ are i.i.d. and $\langle\bz_i,\bh_{j}\rangle\sim\calN(0,\sigma^2)$. Recall that $\widetilde{\bR}_{\bA\bA^{\top}}(\eta)=\sum_{k'\neq k}\frac{\bP_{k'}^{uu}}{\mu_{k'}^2-\eta}$, implying that $\big(\widetilde{\bR}_{\bA\bA^{\top}}(\eta)\widetilde{\bGamma}\big)^{s-2}\widetilde{\bR}_{\bA\bA^{\top}}(\eta)$ can be viewed as a linear combination of operators
$$
(\bP_{t_1}^{uu}\widetilde{\bGamma}\bP_{t_2}^{uu})(\bP_{t_2}^{uu}\widetilde{\bGamma}\bP_{t_3}^{uu})\ldots(\bP_{t_{s-2}}^{uu}\widetilde{\bGamma}\bP_{t_{s-1}}^{uu})
$$
where $t_1,\ldots,t_{s-1}\neq k$. For each $\bP_{t_1}^{uu}\widetilde{\bGamma}\bP_{t_2}^{uu}$, we have
$$
\bP_{t_1}^{uu}\widetilde{\bGamma}\bP_{t_2}^{uu}=\bP_{t_1}^{uu}\bA\bZ^{\top}\bP_{t_2}^{uu}+\bP_{t_1}^{uu}\bZ\bA^{\top}\bP_{t_2}^{uu}+\bP_{t_1}^{uu}\big(\bZ\sum_{k'\neq k}\bP_{k'}^{hh}\bZ^{\top}\big)\bP_{t_2}^{uu}-\sigma^2(m_2-\nu_k)\bP_{t_1}^{uu}\bP_{t_2}^{uu}.
$$
Clearly, $\bP_{t_1}^{uu}\bA\bZ^{\top}$ is a function of random vectors $\bP_{t_1}^{uu}\bA\bz_i, i=1,\ldots,m_1$; $\bZ\bA^{\top}\bP_{t_2}^{uu}$ is a function of random vectors $\bP_{t_2}^{uu}\bA\bz_i, i=1,\ldots,m_1$; $\bZ\sum_{k'\neq k}\bP_{k'}^{hh}\bZ^{\top}=\bZ\sum_{k'\neq k}(\bP_{k'}^{hh})^2\bZ^{\top}$ is a function of random vectors $\bP_{k'}^{hh}\bz_i, i=1,\ldots,m_1$. The following facts are obvious
$$
\mathbb{E}\langle \bz_i, \bh_{j}\rangle \bP_{t_1}^{uu}\bA\bz_i=\bP_{t_1}^{uu}\bA(\mathbb{E}\bz_i\otimes \bz_i)\bh_{j}=\sigma^2\bP_{t_1}^{uu}\bA\bh_{j}=\sigma^2\mu_k\bP_{t_1}^{uu}\bu_{j}={\bf 0},\quad \forall t_1\neq k
$$
and
$$
\mathbb{E}\langle\bz_i,\bh_{j} \rangle\bP_{k'}^{hh}\bz_i=\bP_{k'}^{hh}(\mathbb{E}\bz_i\otimes \bz_i)\bh_{j}=\sigma^2\bP_{k'}^{hh}\bh_{j}={\bf 0},\quad \forall k'\neq k.
$$
Since $\big\{\langle\bz_i,\bh_{j}\rangle, i=1,\ldots, m_1\big\}$ are Gaussian random variables and $\big\{\bP_{t_1}^{uu}\bA\bz_i, \bP_{k'}^{hh}\bz_i, i=1,\ldots, m_1\big\}$ are (complex) Gaussian random vectors, uncorrelations indicate that $\big\{\langle\bz_i,\bh_{j} \rangle: i=1,\ldots,m_1\big\}$ are independent with $\big\{\bP_{t_1}^{uu}\bA\bz_i, \bP_{k'}^{hh}\bz_i: t_1\neq k, k'\neq k, i=1,\ldots,m_1\big\}$. We conclude that $\big\{\langle\bz_i,\bh_{j} \rangle: i=1,\ldots,m_1\big\}$ are independent with 
$$
\big\{\big<\big(\widetilde{\bR}_{\bA\bA^{\top}}(\eta)\widetilde{\bGamma}\big)^{s-2}\widetilde{\bR}_{\bA\bA^{\top}}(\eta)\be_i, \bx\big>, i=1,\ldots,m_1\big\}.
$$ 
To this end, define the complex random variables
$$
\omega_i(\bx)=\big<\big(\widetilde{\bR}_{\bA\bA^{\top}}(\eta)\widetilde{\bGamma}\big)^{s-2}\widetilde{\bR}_{\bA\bA^{\top}}(\eta)\be_i, \bx\big>=\omega_i^{(1)}(\bx)+\omega_i^{(2)}(\bx){\rm Im}\in \CC,\quad i=1,\ldots,m_1
$$
where ${\rm Im}$ denotes the imaginary number.
Then,
\begin{align*}
\big<\big(\widetilde{\bR}_{\bA\bA^{\top}}(\eta)\widetilde{\bGamma}\big)^{s-2}\widetilde{\bR}_{\bA\bA^{\top}}(\eta)&\bZ\bA^{\top}\bu_{j},\bx\big> \\
&=\mu_k\sum_{i=1}^{m_1}\langle\bz_i,\bh_{j} \rangle\omega_i^{(1)}(\bx)+\Big(\mu_k\sum_{i=1}^{m_1}\langle\bz_i,\bh_{j} \rangle\omega_i^{(2)}(\bx)\Big){\rm Im}\\
& =:\kappa_1(\bx)+\kappa_2(\bx){\rm Im}\in \CC.
\end{align*}
Conditioned on $\big\{\bP_{t_1}^{uu}\bA\bz_i, \bP_{k'}^{hh}\bz_i: t_1\neq k, k'\neq k, i=1,\ldots,m_1\big\}$, we get
$$
\mathbb{E}\kappa_1^2(\bx)=\mu_k^2\sigma^2\sum_{i=1}^{m_1}\Big(\omega_i^{(1)}(\bx)\Big)^2
$$
and
$$
\mathbb{E}\kappa_1(\bx)\kappa_2(\bx)=\mu_k^2\sigma^2\sum_{i=1}^{m_1}\omega_i^{(1)}(\bx)\omega_i^{(2)}(\bx)
$$
implying that
the centered Gaussian random vector $(\kappa_1(\bx), \kappa_2(\bx))$ has covariance matrix:
$$
\bigg(\mu_k^2\sigma^2\sum_{i=1}^{m_1}\omega_i^{(k_1)}(\bx)\omega_i^{(k_2)}(\bx)\bigg)_{k_1,k_2=1,2}.
$$
Finally, 
\begin{align*}
\EE^{1/2}\big|\big<\big(\widetilde{\bR}_{\bA\bA^{\top}}(\eta)&\widetilde{\bGamma}\big)^{s-2}\widetilde{\bR}_{\bA\bA^{\top}}(\eta)\bZ\bA^{\top}\bu_{j},\bx\big>\big|^2{\bf 1}\Big(\|\widetilde{\bGamma}\|\leq \delta_t\Big)\\
=&\mathbb{E}^{1/2}\big(\kappa_1^2(\bx)+\kappa_2^2(\bx)\big){\bf 1}\Big(\|\widetilde{\bGamma}\|\leq \delta_t\Big)\\
&=\sigma\mu_k\mathbb{E}^{1/2}\Big(\sum_{i=1}^{m_1}\big(\omega_i^{(1)}(\bx)\big)^2+\big(\omega_i^{(2)}(\bx)\big)^2\Big){\bf 1}\Big(\|\widetilde{\bGamma}\|\leq \delta_t\Big)\\
&=\sigma\mu_k\mathbb{E}^{1/2}\sum_{i=1}^{m_1}\big|\omega_i(\bx)\big|^2{\bf 1}\Big(\|\widetilde{\bGamma}\|\leq \delta_t\Big).
\end{align*}
Moreover,
\begin{eqnarray*}
\sum_{i=1}^{m_1}\big|\omega_i(\bx)\big|^2=\sum_{i=1}^{m_1}\big|\big<\widetilde{\bR}_{\bA\bA^{\top}}(\eta)\big(\widetilde{\bR}_{\bA\bA^{\top}}(\eta)\widetilde{\bGamma}\big)^{s-2}\bx, \be_j\big>\big|^2\leq \big\|\widetilde{\bR}_{\bA\bA^{\top}}(\eta)\big(\widetilde{\bR}_{\bA\bA^{\top}}(\eta)\widetilde{\bGamma}\big)^{s-2}\bx\big\|_{\ell_2}^2\\
\leq \|\widetilde{\bR}_{\bA\bA^{\top}}(\eta)\|^{2(s-1)}\|\widetilde{\bGamma}\|^{2(s-2)}\|\bx\|_{\ell_2}^2\leq \Big(\frac{2}{\bar{g}_k(\bA\bA^{\top})}\Big)^{2(s-1)}\|\widetilde{\bGamma}\|^{2(s-2)}\|\bx\|_{\ell_2}^2.
\end{eqnarray*}
As a result,
\begin{align*}
\EE^{1/2}\big|\big<\big(\widetilde{\bR}_{\bA\bA^{\top}}&(\eta)\widetilde{\bGamma}\big)^{s-2}\widetilde{\bR}_{\bA\bA^{\top}}(\eta)\bZ\bA^{\top}\bu_{j},\bx\big>\big|^2{\bf 1}\Big(\|\widetilde{\bGamma}\|\leq \delta_t \Big)\\
 &\leq \sigma\mu_k\EE^{1/2} \Big(\frac{2}{\bar{g}_k(\bA\bA^{\top})}\Big)^{2(s-1)}\|\widetilde{\bGamma}\|^{2(s-2)}\|\bx\|_{\ell_2}^2{\bf 1}\Big(\|\widetilde{\bGamma}\|\leq \delta_t \Big)\\
&\leq \frac{\sigma\mu_k}{\bar{g}_k(\bA\bA^{\top})}\Big(\frac{2\delta_t}{\bar{g}_k(\bA\bA^{\top})}\Big)^{s-2}\|\bx\|_{\ell_2}.
\end{align*}

\paragraph{Bound of $\ \EE^{1/2}\big|\big<\big(\widetilde{\bR}_{\bA\bA^{\top}}(\eta)\widetilde{\bGamma}\big)^{s-2}\widetilde{\bR}_{\bA\bA^{\top}}(\eta)\bA\bZ^{\top}\bu_{j}, \bx\big>\big|^2{\bf 1}\big(\|\widetilde{\bGamma}\|\leq \delta_t\big)$.} With a little abuse on the notations, we denote by $\bz_1,\ldots,\bz_{m_2}\in \RR^{m_1}$ the corresponding columns of $\bZ$ in this paragraph. Then,
$$
\big<\big(\widetilde{\bR}_{\bA\bA^{\top}}(\eta)\widetilde{\bGamma}\big)^{s-2}\widetilde{\bR}_{\bA\bA^{\top}}(\eta)\bA\bZ^{\top}\bu_{j}, \bx\big>=\sum_{i=1}^{m_2}\langle\bz_i, \bu_{j}\rangle\Big<\big(\widetilde{\bR}_{\bA\bA^{\top}}(\eta)\widetilde{\bGamma}\big)^{s-2}\widetilde{\bR}_{\bA\bA^{\top}}(\eta)\bA\be_i,\bx\Big>.
$$
Similarly, $\big(\widetilde{\bR}_{\bA\bA^{\top}}(\eta)\widetilde{\bGamma}\big)^{s-2}\widetilde{\bR}_{\bA\bA^{\top}}(\eta)$ can be represented as linear combination of operators
$$
\big(\bP_{t_1}^{uu}\widetilde{\bGamma}\bP_{t_2}^{uu}\big)\big(\bP_{t_2}^{uu}\widetilde{\bGamma}\bP_{t_3}^{uu}\big)\ldots\big(\bP_{t_{s-2}}^{uu}\widetilde{\bGamma}\bP_{t_{s-1}}^{uu}\big),\quad t_1,\ldots,t_{s-1}\neq k.
$$
To this end, we write
$$
\bP_{t_1}^{uu}\widetilde{\bGamma}\bP_{t_2}^{uu}=\bP_{t_1}^{uu}\bA\bZ^{\top}\bP_{t_2}^{uu}+\bP_{t_1}^{uu}\bZ\bA^{\top}\bP_{t_2}^{uu}+\bP_{t_1}^{uu}\big(\bZ\sum_{k'\neq k}\bP_{k'}^{hh}\bZ^{\top}\big)\bP_{t_2}^{uu}-\sigma^2(m_2-\nu_k)\bP_{t_1}^{uu}\bP_{t_2}^{uu}.
$$
Observe that $\bP_{t_1}^{uu}\bA\bZ^{\top}\bP_{t_2}^{uu}$, $\bP_{t_1}^{uu}\bZ\bA^{\top}\bP_{t_2}^{uu}$ and $\bP_{t_1}^{uu}\big(\bZ\sum_{k'\neq k}\bP_{k'}^{hh}\bZ^{\top}\big)\bP_{t_2}^{uu}$ are functions of random vectors $\{\bP_{t_1}^{uu}\bz_i, \bP_{t_2}^{uu}\bz_i:  t_1, t_2\neq k, i=1,\ldots,m_2\}$. Moreover,
$$
\mathbb{E}\langle\bz_i, \bu_{j}\rangle\bP_{t_1}^{uu}\bz_i=\bP_{t_1}^{uu}\big(\EE \bz_i\otimes \bz_i\big)\bu_{j}=\sigma^2\bP_{t_1}^{uu}\bu_{j}={\bf 0},\quad \forall\ t_1\neq k
$$
which implies that $\{\langle\bz_i, \bu_{j}\rangle: i=1,\ldots,m_2\}$ and $\Big\{\big<\big(\widetilde{\bR}_{\bA\bA^{\top}}(\eta)\widetilde{\bGamma}\big)^{s-2}\widetilde{\bR}_{\bA\bA^{\top}}(\eta)\bA\be_i,\bx\big>: i=1,\ldots,m_2\Big\}$ are independent. Following an identical analysis as above, we get
$$
\EE^{1/2}\big|\big<\big(\widetilde{\bR}_{\bA\bA^{\top}}(\eta)\widetilde{\bGamma}\big)^{s-2}\widetilde{\bR}_{\bA\bA^{\top}}(\eta)\bA\bZ^{\top}\bu_{j}, \bx\big>\big|^2{\bf 1}\big(\|\widetilde{\bGamma}\|\leq \delta_t\big)\leq \frac{\sigma \mu_1}{\bar{g}_k(\bA\bA^{\top})}\Big(\frac{2\delta_t}{\bar{g}_k(\bA\bA^{\top})}\Big)^{s-2}\|\bx\|_{\ell_2}.
$$

\paragraph{Bound of $\ \EE^{1/2}\Big|\Big<\big(\widetilde{\bR}_{\bA\bA^{\top}}(\eta)\widetilde{\bGamma}\big)^{s-2}\widetilde{\bR}_{\bA\bA^{\top}}(\eta)\big(\bZ\sum_{k'\neq k}\bP_{k'}^{hh}\bZ^{\top}\big)\bu_{j},\bx\Big>\Big|^2{\bf 1}\big(\|\widetilde{\bGamma}\|\leq \delta_t\big)$.} Note that we used the fact $\widetilde{\bR}_{\bA\bA^{\top}}(\eta)\bu_{j}={\bf 0}$ in (\ref{eq:ZZ}). Again, let $\{\bz_1,\ldots,\bz_{m_2}\}\subset \RR^{m_1}$ denote the corresponding columns of $\bZ$. We write
\begin{align*}
\big<\big(\widetilde{\bR}_{\bA\bA^{\top}}(\eta)\widetilde{\bGamma}\big)^{s-2}\widetilde{\bR}_{\bA\bA^{\top}}(\eta)&\big(\bZ\sum_{k'\neq k}\bP_{k'}^{hh}\bZ^{\top}\big)\bu_{j},\bx\big>\\
=&\sum_{i=1}^{m_2}\langle\bz_i,\bu_{j} \rangle\big<\big(\widetilde{\bR}_{\bA\bA^{\top}}(\eta)\widetilde{\bGamma}\big)^{s-2}\widetilde{\bR}_{\bA\bA^{\top}}(\eta)\bZ\big(\sum_{k'\neq k}\bP_{k'}^{hh}\big)\be_i,\bx\big>.
\end{align*}
In a similar fashion, we show that $\big(\widetilde{\bR}_{\bA\bA^{\top}}(\eta)\widetilde{\bGamma}\big)^{s-2}\widetilde{\bR}_{\bA\bA^{\top}}(\eta)\bZ$ is a function of random vectors $\big\{\bP^{uu}_{t}\bz_i: t\neq k, i=1,\ldots,m_2\big\}$ which are independent with $\big\{\langle\bz_i, \bu_{j}\rangle: i=1,\ldots,m_2\big\}$. Then,
\begin{align*}
\EE^{1/2}\Big|\Big<\big(\widetilde{\bR}_{\bA\bA^{\top}}&(\eta)\widetilde{\bGamma}\big)^{s-2}\widetilde{\bR}_{\bA\bA^{\top}}(\eta)\big(\bZ\sum_{k'\neq k}\bP_{k'}^{hh}\bZ^{\top}\big)\bu_{j},\bx\Big>\Big|^2{\bf 1}\big(\|\widetilde{\bGamma}\|\leq \delta\big)\\
\leq& \mathbb{E}^{1/2}\sigma^2\|\widetilde{\bR}_{\bA\bA^{\top}}(\eta)\|^{2(s-1)}\|\widetilde{\bGamma}\|^{2(s-2)}\|\bZ\sum_{k'\neq k}\bP_{k'}^{hh}\|^2\|\bx\|_{\ell_2}^2{\bf 1}\big(\|\widetilde{\bGamma}\|\leq \delta_t\big)\\
\lesssim& \frac{\sigma^2m_2^{1/2}}{\bar{g}_k(\bA\bA^{\top})}\Big(\frac{\delta_t}{\bar{g}_k(\bA\bA^{\top})}\Big)^{s-2}\|\bx\|_{\ell_2}.
\end{align*}
where we used the fact $\mathbb{E}^{1/2}\big\|(\sum_{k'\neq k}\bP_{k'}^{hh})\bZ^{\top}\big\|^2\lesssim \sigma m_2^{1/2}$ from Lemma~\ref{lemma:gammabound}.

\paragraph{Finalize the proof of Theorem.} Combining the above bounds into (\ref{eq:AZ}), (\ref{eq:ZA}) and (\ref{eq:ZZ}), we conclude that
\begin{align*}
\mathbb{E}^{1/2}\Big|\big<\big(\widetilde{\bR}_{\bA\bA^{\top}}&(\eta)\widetilde{\bGamma}\big)^{s-2}\widetilde{\bR}_{\bA\bA^{\top}}(\eta)\widetilde{\bGamma}\bu_{j},\bx\big>\Big|^2{\bf 1}\big(\|\widetilde{\bGamma}\|\leq \delta_t\big)\\
&\lesssim \frac{\sigma^2m_2^{1/2}+\sigma\mu_1}{\bar{g}_k(\bA\bA^{\top})}\Big(\frac{2\delta_t}{\bar{g}_k(\bA\bA^{\top})}\Big)^{s-2}\|\bx\|_{\ell_2}.
\end{align*}
Continue from (\ref{eq:xRy}) and we end up with
\begin{align*}
\mathbb{E}\big<\bx, \big(\widetilde{\bR}_{\bA\bA^{\top}}(\eta)\widetilde{\bGamma}\big)^{s-1}(\bP_k^{uu}\widetilde{\bGamma})&\big(\bR_{\bA\bA^{\top}}(\eta)\widetilde{\bGamma}\big)^{r-s}\bP_k^{uu}\by\big>{\bf 1}\big(\|\widetilde{\bGamma}\|\leq \delta_t\big)\\
\lesssim& \nu_k\delta_t\frac{\sigma^2m_2^{1/2}+\sigma\mu_1}{\bar{g}_k(\bA\bA^{\top})}\Big(\frac{2\delta_t}{\bar{g}_k(\bA\bA^{\top})}\Big)^{r-2}\|\bx\|_{\ell_2}\|\by\|_{\ell_2}.
\end{align*}
Plug the bounds into (\ref{eq:PkSkPk}), 
\begin{align*}
\big|\mathbb{E}&\big<(\bP_k^{uu})^{\perp}\bS_k(\widetilde{\bGamma})\bP_k^{uu}\by,\bx\big>{\bf 1}\big(\|\widetilde{\bGamma}\leq \delta_t\|\big)\big|\\
&\lesssim \sum_{r\geq 2}\frac{\pi\bar{g}_k(\bA\bA^{\top})}{2\pi}\Big(\frac{2}{\bar{g}_k(\bA\bA^{\top})}\Big)^2(r-1)\nu_k\delta_t\frac{\sigma^2m_2^{1/2}+\sigma\mu_1}{\bar{g}_k(\bA\bA^{\top})}\Big(\frac{2\delta_t}{\bar{g}_k(\bA\bA^{\top})}\Big)^{r-2}\|\bx\|_{\ell_2}\|\by\|_{\ell_2}\\
&\leq D_1\nu_k\frac{\sigma^2m_2^{1/2}+\sigma\mu_1}{\bar{g}_k(\bA\bA^{\top})}\|\bx\|_{\ell_2}\|\by\|_{\ell_2}\sum_{r\geq 2}(r-1)\Big(\frac{2\delta_t}{\bar{g}_k(\bA\bA^{\top})}\Big)^{r-1}
\end{align*}
where we used the fact $\oint_{\gamma_k}\big(\widetilde{\bR}_{\bA\bA^{\top}}(\eta)\widetilde{\bGamma}\big)^r\bP_k^{uu}d\eta={\bf 0}$. By the inequality $\sum_{r\geq 1}rq^r=\frac{q}{(1-q)^2}, \forall q<1$ and the fact $D_1\delta_t\leq \bar{g}_k(\bA\bA^{\top})$ for some large constant $D_1>0$ and $t\leq m_1$, we conclude with
\begin{align*}
\big|\mathbb{E}\big<(\bP_k^{uu})^{\perp}&\bS_k(\widetilde{\bGamma})\bP_k^{uu}\by,\bx\big>{\bf 1}\big(\|\widetilde{\bGamma}\leq \delta_t\|\big)\big|\\
&\lesssim \nu_k\frac{\sigma^2m_2^{1/2}+\sigma\mu_1}{\bar{g}_k(\bA\bA^{\top})}\Big(\frac{2\delta_t}{\bar{g}_k(\bA\bA^{\top})}\Big)\|\bx\|_{\ell_2}\|\by\|_{\ell_2},\quad \forall \bx, \by\in\RR^{m_1}
\end{align*}
implying that
$$
\Big\|\EE (\bP_k^{uu})^{\perp}\bS_k(\widetilde{\bGamma})\bP_k^{uu}{\bf 1}\big(\|\widetilde{\bGamma}\|\leq \delta_t\big)\Big\|\lesssim \nu_k\frac{\sigma^2m_2^{1/2}+\sigma\mu_1}{\bar{g}_k(\bA\bA^{\top})}\Big(\frac{2\delta_t}{\bar{g}_k(\bA\bA^{\top})}\Big).
$$
The same bound holds for 
$$
\big\|\EE \bP_k^{uu}\bS_k(\widetilde{\bGamma})(\bP_k^{uu})^{\perp}{\bf 1}\big(\|\widetilde{\bGamma}\|\leq \delta_t\big)\big\|\quad{\rm and}\quad \big\|\EE (\bP_k^{uu})^{\perp}\bS_k(\widetilde{\bGamma})(\bP_k^{uu})^{\perp}{\bf 1}\big(\|\widetilde{\bGamma}\|\leq \delta_t\big)\big\|,
$$
following the same arguments. As a result,
\begin{align}
\Big\|\mathbb{E}\Big((\bP_k^{uu})^{\perp}\bS_k(\widetilde{\bGamma})\bP_k^{uu}&+\bP_k^{uu}\bS_k(\widetilde{\bGamma})(\bP_k^{uu})^{\perp}+(\bP_k^{uu})^{\perp}\bS_k(\widetilde{\bGamma})(\bP_k^{uu})^{\perp}\Big){\bf 1}\big(\|\widetilde{\bGamma}\|\leq \delta_t\big)\Big\|\nonumber\\
&\lesssim \nu_k\frac{\sigma^2m_2^{1/2}+\sigma\mu_1}{\bar{g}_k(\bA\bA^{\top})}\Big(\frac{2\delta_t}{\bar{g}_k(\bA\bA^{\top})}\Big).\label{eq:biasbound} 
\end{align}
By choosing $t=m_1$ such that $\PP(\|\widetilde{\bGamma}\|\geq \delta_{m_1})\leq e^{-m_1/2}$, we get
\begin{align*}
&\Big\|\mathbb{E}\Big((\bP_k^{uu})^{\perp}\bS_k(\widetilde{\bGamma})\bP_k^{uu}+\bP_k^{uu}\bS_k(\widetilde{\bGamma})(\bP_k^{uu})^{\perp}+(\bP_k^{uu})^{\perp}\bS_k(\widetilde{\bGamma})(\bP_k^{uu})^{\perp}\Big){\bf 1}\big(\|\widetilde{\bGamma}\|> \delta_{m_1}\big)\Big\|\\
&\leq \mathbb{E}\Big\|\Big((\bP_k^{uu})^{\perp}\bS_k(\widetilde{\bGamma})\bP_k^{uu}+\bP_k^{uu}\bS_k(\widetilde{\bGamma})(\bP_k^{uu})^{\perp}+(\bP_k^{uu})^{\perp}\bS_k(\widetilde{\bGamma})(\bP_k^{uu})^{\perp}\Big)\Big\|{\bf 1}\big(\|\widetilde{\bGamma}\|> \delta_{m_1}\big)\\
&\leq \EE\|\bS_k(\widetilde{\bGamma})\|{\bf 1}\big(\|\widetilde{\bGamma}\|> \delta_{m_1}\big)
\leq \mathbb{E}^{1/2}\|\bS_k(\widetilde{\bGamma})\|^2\PP^{1/2}\big(\|\widetilde{\bGamma}\|>\delta_{m_1}\big)\\
&\lesssim \Big(\frac{\delta_{m_1}}{\bar{g}_k(\bA\bA^{\top})}\Big)^2\PP^{1/2}\big(\|\widetilde{\bGamma}\|>\delta_{m_1}\big)
\lesssim \Big(\frac{\delta_{m_1}}{\bar{g}_k(\bA\bA^{\top})}\Big)^2 e^{-m_1/2},
\end{align*}
which is clearly dominated by (\ref{eq:biasbound}).
Substitute the above bounds into (\ref{eq:hatP-P}) and we get
\begin{align*}
\Big\|\EE\hat{\bP}_k^{uu}-\bP_k^{uu}-\bP_k^{uu}\bS_k(\widetilde{\bGamma})\bP_k^{uu}\Big\| &\leq \|\mathbb{E}\bS_k(\widetilde{\bGamma})-\bS_k(\hat{\bGamma})\|+D_1\nu_k\frac{\sigma^2m_2^{1/2}+\sigma\mu_1}{\bar{g}_k(\bA\bA^{\top})}\Big(\frac{2\delta(m_1,m_2)}{\bar{g}_k(\bA\bA^{\top})}\Big)\\
&\leq D_2\nu_k\frac{\sigma^2m_2^{1/2}+\sigma^2m_1+\sigma\mu_1}{\bar{g}_k(\bA\bA^{\top})}\Big(\frac{2\delta(m_1,m_2)}{\bar{g}_k(\bA\bA^{\top})}\Big).
\end{align*}
\end{proof}

\subsection{Proof of Theorem~\ref{thm:tildeu}}
The proof of Theorem~\ref{thm:tildeu} is identical to the proof of Corollary~$1.5$ in \cite{koltchinskii2016perturbation} and will be skipped here.

\subsection{Proof of Theorem~\ref{thm:linfty}}
It suffices to prove the upper bound of $\big|\widetilde{A}(i,j,k)-A(i,j,k)\big|$ for $i\in[d_1], j\in[d_2], k\in[d_3]$. To this end, denote by $\be_i$ the $i$-th canonical basis vectors. Observe that
\begin{align*}
\big<\widetilde{\bA}-\bA, \be_i\otimes \be_j\otimes \be_k \big>=&\Big<\bA\times_1\bP_{\widetilde{\bU}}\times_2\bP_{\widetilde{\bV}}\times_3\bP_{\widetilde{\bW}}-\bA, \be_i\otimes \be_j\otimes \be_k\Big>\\
+&\Big<\bZ\times_1\bP_{\widetilde{\bU}}\times_2\bP_{\widetilde{\bV}}\times_3\bP_{\widetilde{\bW}}, \be_i\otimes \be_j\otimes \be_k\Big>.
\end{align*}
Some preliminary facts shall be concluded from Theorem~\ref{thm:linearform}. 
By Theorem~\ref{thm:tildeu}, there exists an event $\calE_2$ with $\PP\big(\calE_2\big)\geq 1-\frac{1}{d^2}$ on which
$$
\big\|\be_i^{\top}\big(\widetilde{\bU}-\bU\big)\big\|_{\ell_2}\leq r^{1/2}\big\|\be_i^{\top}\big(\widetilde{\bU}-\bU\big)\big\|_{\ell_\infty}\lesssim \frac{\sigma\overline{\Lambda}(\bA)r^{1/2}+\sigma^2dr^{1/2}}{\bar{g}_{\min}^2(\bA)}\log^{1/2}d
$$
and 
$$
\big\|\widetilde{\bU}^{\top}\bU-\bI_{r_1}\big\|\leq \|\widetilde{\bU}^{\top}\bU-\bI_{r_1}\|_{\rm F}\lesssim r \|\widetilde{\bU}^{\top}\bU-\bI_{r_1}\|_{\ell_\infty}\lesssim \frac{\sigma\overline{\Lambda}(\bA)r+\sigma^2dr}{\bar{g}_{\min}^2(\bA)}\log^{1/2}d.
$$
The following decomposition is straightforward,
\begin{align*}
\bA\cdot\big(\bP_{\widetilde{\bU}},& \bP_{\widetilde{\bV}}, \bP_{\widetilde{\bW}}\big)-\bA\\
=&\bA\cdot \big(\bP_{\widetilde{\bU}}-\bP_\bU,\bP_{\bV}, \bP_\bW\big)+\bA\cdot\big(\bP_\bU, \bP_{\widetilde{\bV}}-\bP_\bV,\bP_\bW\big)\\
+&\bA\cdot\big(\bP_\bU,\bP_\bV,\bP_{\widetilde{\bW}}-\bP_\bW\big)+\bA\cdot\big(\bP_{\widetilde{\bU}}-\bP_{\bU}, \bP_{\widetilde{\bV}}-\bP_\bV,\bP_\bW\big)\\
+&\bA\cdot\big(\bP_{\widetilde{\bU}}-\bP_\bU, \bP_{\bV}, \bP_{\widetilde{\bW}}-\bP_\bW\big)+
\bA\cdot\big(\bP_\bU, \bP_{\widetilde{\bV}}-\bP_\bV, \bP_{\widetilde{\bW}}-\bP_\bW\big)\\
+&\bA\cdot\big(\bP_{\widetilde{\bU}}-\bP_\bU, \bP_{\widetilde{\bV}}-\bP_\bV, \bP_{\widetilde{\bW}}-\bP_\bW\big)
\end{align*}
Recall that $\bA=\bC\cdot(\bU,\bV,\bW)$ and we get
\begin{align*}
\Big<\bA\cdot\big(&\bP_{\widetilde{\bU}}-\bP_\bU, \bP_\bV, \bP_\bW\big), \be_i\otimes \be_j\otimes \be_k\Big>\\
=&\be_i^{\top}\Big(\widetilde{\bU}\big(\widetilde{\bU}^{\top}\bU\big)-\bU\Big)\calM_1(\bC)\big(\bV\otimes \bW\big)^{\top}(\be_j\otimes \be_k).
\end{align*}
Observe that
\begin{eqnarray*}
\be_i^{\top}\Big(\widetilde{\bU}\big(\widetilde{\bU}^{\top}\bU\big)-\bU\Big)=\be_i^{\top}\big(\widetilde{\bU}-\bU\big)\big(\widetilde{\bU}^{\top}\bU\big)+\be_i^{\top}\bU\big(\widetilde{\bU}^{\top}\bU-\bI_{r_1}\big)
\end{eqnarray*}
implying that on event $\calE_2$,
\begin{align*}
\Big\|\be_i^{\top}\Big(&\widetilde{\bU}\big(\widetilde{\bU}^{\top}\bU\big)-\bU\Big)\Big\|_{\ell_2}\\
\leq& \big\|(\widetilde{\bU}-\bU)^{\top}\be_i\big\|_{\ell_2}\|\widetilde{\bU}^{\top}\bU\|+\big\|\tilde\bU^{\top}\bU-\bI_{r_1}\big\|\|\bU^{\top}\be_i\|_{\ell_2}\\
\lesssim& \frac{\sigma\overline{\Lambda}(\bA)r^{1/2}+\sigma^2dr^{1/2}}{\bar{g}_{\min}^2(\bA)}\log^{1/2}d+\|\bU^{\top}\be_i\|_{\ell_2}\frac{\sigma\overline{\Lambda}(\bA)r+\sigma^2dr}{\bar{g}_{\min}^2(\bA)}\log^{1/2}d\\
\lesssim& \frac{\sigma\overline{\Lambda}(\bA)r+\sigma^2dr}{\bar{g}_{\min}^2(\bA)}\log^{1/2}d,
\end{align*}
where we used the facts $\|\widetilde{\bU}^{\top}\bU\|\leq \|\widetilde{\bU}\|\|\bU\|\leq (1+b_k)^{-1/2}=O(1)$ and
$$
\|\bU^{\top}\be_i\|_{\ell_2}=\big<\bU\bU^{\top}, \be_i\otimes \be_i\big>^{1/2}\leq 1.
$$
Therefore, on event $\calE_2$,
\begin{align*}
\big|\big<\bA\cdot\big(\bP_{\widetilde{\bU}}-\bP_\bU,&\bP_\bV,\bP_\bW\big), \be_i\otimes \be_j\otimes \be_k\big>\big|\\
\lesssim& \overline{\Lambda}(\bA)\bigg(\frac{\sigma\overline{\Lambda}(\bA)r+\sigma^2dr}{\bar{g}_{\min}^2(\bA)}\log^{1/2}d\bigg)\|\bV^{\top}\be_j\|_{\ell_2}\|\bW^{\top}\be_k\|_{\ell_2}.
\end{align*}
Similar bounds hold for
$$
\big|\big<\bA\cdot\big(\bP_\bU,\bP_{\widetilde{\bV}}-\bP_\bV,\bP_\bW\big), \be_i\otimes \be_j\otimes \be_k\big>\big|\quad {\rm and}\quad \big|\big<\bA\cdot\big(\bP_\bU,\bP_\bV,\bP_{\widetilde{\bW}}-\bP_\bW\big), \be_i\otimes \be_j\otimes \be_k\big>\big|.
$$
Following the same method, we can show that on event $\calE_2$,
\begin{align*}
\big|\big<\bA\cdot\big(\bP_{\widetilde{\bU}}-\bP_\bU,&\bP_{\widetilde{\bV}}-\bP_\bV,\bP_\bW\big), \be_i\otimes \be_j\otimes \be_k\big>\big|\\
\lesssim&\overline{\Lambda}(\bA)\bigg(\frac{\sigma\overline{\Lambda}(\bA)r+\sigma^2dr}{\bar{g}_{\min}^2(\bA)}\log^{1/2}d\bigg)^2\|\bW^{\top}\be_k\|_{\ell_2}
\end{align*}
and
\begin{align*}
\big|\big<\bA\cdot\big(\bP_{\widetilde{\bU}}-\bP_\bU,&\bP_{\widetilde{\bV}}-\bP_\bV,\bP_{\widetilde{\bW}}-\bP_\bW\big), \be_i\otimes \be_j\otimes \be_k\big>\big|\\
\lesssim&\overline{\Lambda}(\bA)\bigg(\frac{\sigma\overline{\Lambda}(\bA)r+\sigma^2dr}{\bar{g}_{\min}^2(\bA)}\log^{1/2}d\bigg)^3.
\end{align*}
We conclude that on event $\calE_2$,
\begin{align*}
\big|\big<\bA&\cdot\big(\bP_{\widetilde{\bU}}, \bP_{\widetilde{\bV}}, \bP_{\widetilde{\bW}}\big)-\bA, \be_i\otimes \be_j\otimes \be_k\big>\big|\\
\lesssim& \overline{\Lambda}(\bA)\bigg(\frac{\sigma\overline{\Lambda}(\bA)r+\sigma^2dr}{\bar{g}_{\min}^2(\bA)}\log^{1/2}d\bigg)\Big(\|\bV^{\top}\be_j\|_{\ell_2}\|\bW^{\top}\be_k\|_{\ell_2}\\
&+\|\bU^{\top}\be_i\|_{\ell_2}\|\bW^{\top}\be_k\|_{\ell_2}+\|\bU^{\top}\be_i\|_{\ell_2}\|\bV^{\top}\be_j\|_{\ell_2}\Big)\\
&+ \overline{\Lambda}(\bA)\bigg(\frac{\sigma\overline{\Lambda}(\bA)r+\sigma^2dr}{\bar{g}_{\min}^2(\bA)}\log^{1/2}d\bigg)^2\Big(\|\bV^{\top}\be_j\|_{\ell_2}+\|\bU^{\top}\be_i\|_{\ell_2}+\|\bW^{\top}\be_k\|_{\ell_2}\Big)\\
&+\overline{\Lambda}(\bA)\bigg(\frac{\sigma\overline{\Lambda}(\bA)r+\sigma^2dr}{\bar{g}_{\min}^2(\bA)}\log^{1/2}d\bigg)^3.
\end{align*}
Recall that for all $i\in[d_1], j\in[d_2], k\in[d_3]$
$$
\|\bU^{\top}\be_i\|_{\ell_2}\leq \mu_\bU\sqrt{\frac{r}{d}},\quad \|\bV^{\top}\be_j\|_{\ell_2}\leq \mu_\bV\sqrt{\frac{r}{d}},\quad \|\bW^{\top}\be_k\|_{\ell_2}\leq \mu_\bW\sqrt{\frac{r}{d}}
$$
and conditions (\ref{eq:denoisegk1}) (\ref{eq:denoisegk2}) (\ref{eq:denoisegk3}) imply 
$$
\frac{\sigma\overline{\Lambda}(\bA)r+\sigma^2dr}{\bar{g}_{\min}^2(\bA)}\log^{1/2}d\lesssim r\Big(\frac{\log d}{d}\Big)^{1/2}.
$$
We end up with a simpler bound on event $\calE_2$,
\begin{align}\label{eq:A-Aijk}
\big|\big<\bA\cdot\big(\bP_{\widetilde{\bU}}, &\bP_{\widetilde{\bV}}, \bP_{\widetilde{\bW}}\big)-\bA, \be_i\otimes \be_j\otimes \be_k\big>\big|\\
\lesssim& \sigma r^{3}\bigg(\frac{\sigma\widetilde{\kappa}(\bA)}{\bar{g}_{\min}(\bA)}+\frac{\widetilde{\kappa}^2(\bA)}{d}\bigg)(\mu_\bU\mu_\bV+\mu_\bU\mu_\bW+\mu_\bV\mu_\bW)\log^{3/2}d\nonumber
\end{align}
where $\widetilde{\kappa}(\bA)=\overline{\Lambda}(\bA)/\bar{g}_{\min}(\bA)$.

Next, we prove the upper bound of $\big|\big<\bZ\cdot(\bP_{\widetilde{\bU}}, \bP_{\widetilde{\bV}}, \bP_{\widetilde{\bW}}), \be_i\otimes \be_j\otimes \be_k\big>\big|$ and we proceed with the same decomposition. Observe that 
\begin{align*}
\big<\bZ\cdot(\bP_\bU,\bP_\bV,\bP_\bW), \be_i\otimes \be_j\otimes \be_k\big>=&\big<\bZ, (\bP_\bU\be_i)\otimes(\bP_\bV\be_j)\otimes(\bP_\bW\be_k)\big>\\
\sim&\calN\Big(0,\sigma^2\big\|\bP_\bU\be_i\big\|_{\ell_2}^2\big\|\bP_\bV\be_j\big\|_{\ell_2}^2\big\|\bP_\bW\be_k\big\|_{\ell_2}^2\Big)
\end{align*}
The standard concentration inequality of Gaussian random variables yields that with probability at least $1-\frac{1}{d^2}$,
\begin{align*}
\big|\big<\bZ\cdot(\bP_\bU,\bP_\bV,\bP_\bW), \be_i\otimes \be_j\otimes \be_k\big>\big|\lesssim& \sigma\|\bU^{\top}\be_i\|_{\ell_2}\|\bV^{\top}\be_j\|_{\ell_2}\|\bW^{\top}\be_k\|_{\ell_2}\log^{1/2}d\\
\lesssim& \sigma\Big(\frac{r}{d}\Big)^{3/2}\mu_\bU\mu_\bV\mu_\bW\log^{1/2}d.
\end{align*}
Similarly, with probability at least $1-\frac{1}{d^2}$,
\begin{align*}
\big|\big<\bZ\cdot\big(\bP_{\widetilde{\bU}}&-\bP_\bU, \bP_\bV, \bP_\bW\big), \be_i\otimes \be_j\otimes \be_k\big>\big|\\
=&\big|\be_i^{\top}\big(\bP_{\widetilde{\bU}}-\bP_\bU\big)\calM_1(\bZ)\big(\bV\otimes\bW\big)\big((\bV^{\top}\be_j)\otimes(\bW^{\top}\be_k)\big)\big|\\
\leq& \|(\bP_{\widetilde{\bU}}-\bP_\bU)\be_i\|_{\ell_2}\big\|\calM_1(\bZ)(\bV\otimes\bW)\big\|\|\bV^{\top}\be_j\|_{\ell_2}\|\bW^{\top}\be_k\|_{\ell_2}\\
\lesssim& \sigma d^{1/2}\|(\bP_{\widetilde{\bU}}-\bP_\bU)\be_i\|_{\ell_2}\big|\|\bV^{\top}\be_j\|_{\ell_2}\|\bW^{\top}\be_k\|_{\ell_2}
\end{align*}
where we used Lemma~\ref{lemma:gammabound} for the upper bound of $\big\|\calM_1(\bZ)(\bV\otimes\bW)\big\|$. Moreover, since $\mu_\bU\geq 1$,
\begin{align*}
\big\|\big(\bP_{\widetilde{\bU}}-\bP_\bU\big)\be_i\big\|_{\ell_2}\leq& \|(\widetilde{\bU}-\bU)\be_i\|_{\ell_2}+\|\widetilde{\bU}-\bU\|_{\ell_2}\|\bU^{\top}\be_i\|_{\ell_2}\\
\lesssim& \frac{\sigma\overline{\Lambda}(\bA)r+\sigma^2dr}{\bar{g}_{\min}^2(\bA)}\mu_\bU\log^{1/2}d.
\end{align*}
Denote the above event by $\calE_3$. On $\calE_2\cap\calE_3$,
\begin{eqnarray*}
\big|\big<\bZ\cdot\big(\bP_{\widetilde{\bU}}-\bP_\bU, \bP_\bV, \bP_\bW\big), \be_i\otimes \be_j\otimes \be_k\big>\big|\lesssim \frac{\sigma r}{d^{1/2}}\bigg(\frac{\sigma\overline{\Lambda}(\bA)r+\sigma^2dr}{\bar{g}_{\min}^2(\bA)}\bigg)\mu_\bU\mu_\bV\mu_\bW\log^{1/2}d.
\end{eqnarray*}
Similar bounds can be attained for 
$$
\big|\big<\bZ\cdot\big(\bP_\bU, \bP_{\widetilde{\bV}}-\bP_\bV, \bP_\bW\big), \be_i\otimes \be_j\otimes \be_k\big>\big|\quad{\rm and}\quad
\big|\big<\bZ\cdot\big(\bP_\bU, \bP_\bV, \bP_{\widetilde{\bW}}-\bP_\bW\big), \be_i\otimes \be_j\otimes \be_k\big>\big|.
$$
In an identical fashion, on event $\calE_2\cap \calE_3$,
\begin{align*}
\big|\big<\bZ\cdot\big(\bP_{\widetilde{\bU}}-\bP_\bU,& \bP_{\widetilde{\bV}}-\bP_\bV, \bP_\bW\big), \be_i\otimes \be_j\otimes \be_k\big>\big|\\
\lesssim& \sigma r^{1/2}\bigg(\frac{\sigma\overline{\Lambda}(\bA)r+\sigma^2dr}{\bar{g}_{\min}^2(\bA)}\bigg)^2\mu_\bU\mu_\bV\mu_\bW\log d.
\end{align*}
and
\begin{align*}
\big|\big<\bZ\cdot\big(\bP_{\widetilde{\bU}}-\bP_\bU, &\bP_{\widetilde{\bV}}-\bP_\bV, \bP_{\widetilde{\bW}}-\bP_\bW\big), \be_i\otimes \be_j\otimes \be_k\big>\big|\\
\lesssim& \sigma d^{1/2}\bigg(\frac{\sigma\overline{\Lambda}(\bA)r+\sigma^2dr}{\bar{g}_{\min}^2(\bA)}\bigg)^3\mu_\bU\mu_\bV\mu_\bW\log^{3/2}d.
\end{align*}
Observe by conditions (\ref{eq:denoisegk1}) (\ref{eq:denoisegk2}) (\ref{eq:denoisegk3}) that
$$
\frac{\sigma\overline{\Lambda}(\bA)r+\sigma^2dr}{\bar{g}_{\min}^2(\bA)}\lesssim \frac{r}{d^{1/2}}.
$$ 
We conclude on event $\calE_2\cap\calE_3$ with
\begin{eqnarray}\label{eq:Zijk}
\big|\big<\bZ\cdot\big(\bP_{\widetilde{\bU}}, \bP_{\widetilde{\bV}}, \bP_{\widetilde{\bW}}\big), \be_i\otimes \be_j\otimes \be_k\big>\big|\lesssim \frac{\sigma r^2}{d^{1/2}}\bigg(\frac{\sigma\overline{\Lambda}(\bA)r+\sigma^2dr}{\bar{g}_{\min}^2(\bA)}\bigg)\mu_\bU\mu_\bV\mu_\bW\log^{3/2}d.
\end{eqnarray}
By combining (\ref{eq:A-Aijk}) and (\ref{eq:Zijk}), we get on event $\calE_2\cap\calE_3$,
\begin{align*}
\big|\big<\widetilde{\bA}-&\bA, \be_i\otimes \be_j\otimes \be_k\big>\big|\\
\lesssim& \sigma r^{3}\bigg(\frac{\sigma\widetilde{\kappa}(\bA)}{\bar{g}_{\min}(\bA)}+\frac{\widetilde{\kappa}^2(\bA)}{d}\bigg)(\mu_\bU\mu_\bV+\mu_\bU\mu_\bW+\mu_\bV\mu_\bW)\log^{3/2}d\\
+&\frac{\sigma r^2}{d^{1/2}}\bigg(\frac{\sigma\overline{\Lambda}(\bA)r+\sigma^2dr}{\bar{g}_{\min}^2(\bA)}\bigg)\mu_\bU\mu_\bV\mu_\bW\log^{3/2}d\\
\lesssim& \sigma r^{3}\bigg(\frac{\sigma\widetilde{\kappa}(\bA)}{\bar{g}_{\min}(\bA)}+\frac{\widetilde{\kappa}^2(\bA)}{d}\bigg)(\mu_\bU\mu_\bV+\mu_\bU\mu_\bW+\mu_\bV\mu_\bW)\log^{3/2}d,
\end{align*}
where the last inequality is due to fact $\bar{g}_{\min}(\bA)\gtrsim \sigma d^{3/4}$ and $\max\big\{\mu_\bU, \mu_\bV, \mu_\bW\big\}\lesssim \sqrt{d}$. 

\bibliographystyle{plainnat}
\bibliography{refer}

\begin{thebibliography}{54}
\providecommand{\natexlab}[1]{#1}
\providecommand{\url}[1]{\texttt{#1}}
\expandafter\ifx\csname urlstyle\endcsname\relax
  \providecommand{\doi}[1]{doi: #1}\else
  \providecommand{\doi}{doi: \begingroup \urlstyle{rm}\Url}\fi

\bibitem[Acar and Yener(2009)]{acar2009unsupervised}
Evrim Acar and B{\"u}lent Yener.
\newblock Unsupervised multiway data analysis: A literature survey.
\newblock \emph{IEEE transactions on knowledge and data engineering},
  21\penalty0 (1):\penalty0 6--20, 2009.

\bibitem[Anandkumar et~al.(2013)Anandkumar, Ge, Hsu, and
  Kakade]{anandkumar2013tensor}
Animashree Anandkumar, Rong Ge, Daniel Hsu, and Sham Kakade.
\newblock A tensor spectral approach to learning mixed membership community
  models.
\newblock In \emph{Conference on Learning Theory}, pages 867--881, 2013.

\bibitem[Anandkumar et~al.(2014)Anandkumar, Ge, Hsu, Kakade, and
  Telgarsky]{anandkumar2014tensor}
Animashree Anandkumar, Rong Ge, Daniel Hsu, Sham~M Kakade, and Matus Telgarsky.
\newblock Tensor decompositions for learning latent variable models.
\newblock \emph{The Journal of Machine Learning Research}, 15\penalty0
  (1):\penalty0 2773--2832, 2014.

\bibitem[Ben-Dor et~al.(2003)Ben-Dor, Chor, Karp, and
  Yakhini]{ben2003discovering}
Amir Ben-Dor, Benny Chor, Richard Karp, and Zohar Yakhini.
\newblock Discovering local structure in gene expression data: the
  order-preserving submatrix problem.
\newblock \emph{Journal of computational biology}, 10\penalty0 (3-4):\penalty0
  373--384, 2003.

\bibitem[Bergqvist and Larsson(2010)]{bergqvist2010higher}
G{\"o}ran Bergqvist and Erik~G Larsson.
\newblock The higher-order singular value decomposition: Theory and an
  application [lecture notes].
\newblock \emph{IEEE Signal Processing Magazine}, 27\penalty0 (3):\penalty0
  151--154, 2010.

\bibitem[Brubaker and Vempala(2009)]{brubaker2009random}
S~Charles Brubaker and Santosh~S Vempala.
\newblock Random tensors and planted cliques.
\newblock In \emph{Approximation, Randomization, and Combinatorial
  Optimization. Algorithms and Techniques}, pages 406--419. Springer, 2009.

\bibitem[Cai and Zhang(2016)]{cai2016rate}
T~Tony Cai and Anru Zhang.
\newblock Rate-optimal perturbation bounds for singular subspaces with
  applications to high-dimensional statistics.
\newblock \emph{arXiv preprint arXiv:1605.00353}, 2016.

\bibitem[Cai et~al.(2015)Cai, Liang, and Rakhlin]{cai2015computational}
T~Tony Cai, Tengyuan Liang, and Alexander Rakhlin.
\newblock Computational and statistical boundaries for submatrix localization
  in a large noisy matrix.
\newblock \emph{arXiv preprint arXiv:1502.01988}, 2015.

\bibitem[Cape et~al.(2017)Cape, Tang, and Priebe]{cape2017two}
Joshua Cape, Minh Tang, and Carey~E Priebe.
\newblock The two-to-infinity norm and singular subspace geometry with
  applications to high-dimensional statistics.
\newblock \emph{arXiv preprint arXiv:1705.10735}, 2017.

\bibitem[Chaganty and Liang(2013)]{chaganty2013spectral}
Arun~T Chaganty and Percy Liang.
\newblock Spectral experts for estimating mixtures of linear regressions.
\newblock In \emph{Proceedings of the 30th International Conference on Machine
  Learning (ICML-13)}, pages 1040--1048, 2013.

\bibitem[Chen and Saad(2009)]{chen2009tensor}
Jie Chen and Yousef Saad.
\newblock On the tensor svd and the optimal low rank orthogonal approximation
  of tensors.
\newblock \emph{SIAM Journal on Matrix Analysis and Applications}, 30\penalty0
  (4):\penalty0 1709--1734, 2009.

\bibitem[Cichocki et~al.(2015)Cichocki, Mandic, De~Lathauwer, Zhou, Zhao,
  Caiafa, and Phan]{cichocki2015tensor}
Andrzej Cichocki, Danilo Mandic, Lieven De~Lathauwer, Guoxu Zhou, Qibin Zhao,
  Cesar Caiafa, and Huy~Anh Phan.
\newblock Tensor decompositions for signal processing applications: From
  two-way to multiway component analysis.
\newblock \emph{IEEE Signal Processing Magazine}, 32\penalty0 (2):\penalty0
  145--163, 2015.

\bibitem[Davis and Kahan(1970)]{davis1970rotation}
Chandler Davis and William~Morton Kahan.
\newblock The rotation of eigenvectors by a perturbation. iii.
\newblock \emph{SIAM Journal on Numerical Analysis}, 7\penalty0 (1):\penalty0
  1--46, 1970.

\bibitem[De~Lathauwer et~al.(2000{\natexlab{a}})De~Lathauwer, De~Moor, and
  Vandewalle]{de2000best}
Lieven De~Lathauwer, Bart De~Moor, and Joos Vandewalle.
\newblock On the best rank-1 and rank-(r1, r2,..., rn) approximation of
  higher-order tensors.
\newblock \emph{SIAM Journal on Matrix Analysis and Applications}, 21\penalty0
  (4):\penalty0 1324--1342, 2000{\natexlab{a}}.

\bibitem[De~Lathauwer et~al.(2000{\natexlab{b}})De~Lathauwer, De~Moor, and
  Vandewalle]{de2000multilinear}
Lieven De~Lathauwer, Bart De~Moor, and Joos Vandewalle.
\newblock A multilinear singular value decomposition.
\newblock \emph{SIAM journal on Matrix Analysis and Applications}, 21\penalty0
  (4):\penalty0 1253--1278, 2000{\natexlab{b}}.

\bibitem[Fan and Fan(2008)]{fan2008high}
Jianqing Fan and Yingying Fan.
\newblock High dimensional classification using features annealed independence
  rules.
\newblock \emph{Annals of statistics}, 36\penalty0 (6):\penalty0 2605, 2008.

\bibitem[Fan et~al.(2016)Fan, Wang, and Zhong]{fan2016ell}
Jianqing Fan, Weichen Wang, and Yiqiao Zhong.
\newblock An $\ell_\infty$ eigenvector perturbation bound and its application
  to robust covariance estimation.
\newblock \emph{arXiv preprint arXiv:1603.03516}, 2016.

\bibitem[Florescu and Perkins(2015)]{florescu2015spectral}
Laura Florescu and Will Perkins.
\newblock Spectral thresholds in the bipartite stochastic block model.
\newblock \emph{arXiv preprint arXiv:1506.06737}, 2015.

\bibitem[Friedman(1989)]{friedman1989regularized}
Jerome~H Friedman.
\newblock Regularized discriminant analysis.
\newblock \emph{Journal of the American statistical association}, 84\penalty0
  (405):\penalty0 165--175, 1989.

\bibitem[Gauvin et~al.(2014)Gauvin, Panisson, and Cattuto]{gauvin2014detecting}
Laetitia Gauvin, Andr{\'e} Panisson, and Ciro Cattuto.
\newblock Detecting the community structure and activity patterns of temporal
  networks: a non-negative tensor factorization approach.
\newblock \emph{PloS one}, 9\penalty0 (1):\penalty0 e86028, 2014.

\bibitem[Hastie et~al.(2009)Hastie, Tibshirani, and
  Friedman]{hastie2009unsupervised}
Trevor Hastie, Robert Tibshirani, and Jerome Friedman.
\newblock Unsupervised learning.
\newblock In \emph{The elements of statistical learning}, pages 485--585.
  Springer, 2009.

\bibitem[Hildebrand and R{\"u}egsegger(1997)]{hildebrand1997new}
T~Hildebrand and P~R{\"u}egsegger.
\newblock A new method for the model-independent assessment of thickness in
  three-dimensional images.
\newblock \emph{Journal of microscopy}, 185\penalty0 (1):\penalty0 67--75,
  1997.

\bibitem[Hillar and Lim(2013)]{hillar2013most}
Christopher~J Hillar and Lek-Heng Lim.
\newblock Most tensor problems are {NP}-hard.
\newblock \emph{Journal of the ACM (JACM)}, 60\penalty0 (6):\penalty0 45, 2013.

\bibitem[Hopkins et~al.(2015)Hopkins, Shi, and Steurer]{hopkins2015tensor}
Samuel~B Hopkins, Jonathan Shi, and David Steurer.
\newblock Tensor principal component analysis via sum-of-square proofs.
\newblock In \emph{COLT}, pages 956--1006, 2015.

\bibitem[Hore et~al.(2016)Hore, Vi{\~n}uela, Buil, Knight, McCarthy, Small, and
  Marchini]{hore2016tensor}
Victoria Hore, Ana Vi{\~n}uela, Alfonso Buil, Julian Knight, Mark~I McCarthy,
  Kerrin Small, and Jonathan Marchini.
\newblock Tensor decomposition for multiple-tissue gene expression experiments.
\newblock \emph{Nature Genetics}, 48\penalty0 (9):\penalty0 1094--1100, 2016.

\bibitem[Jin(2015)]{jin2015fast}
Jiashun Jin.
\newblock Fast community detection by score.
\newblock \emph{The Annals of Statistics}, 43\penalty0 (1):\penalty0 57--89,
  2015.

\bibitem[Kolar et~al.(2011)Kolar, Balakrishnan, Rinaldo, and
  Singh]{kolar2011minimax}
Mladen Kolar, Sivaraman Balakrishnan, Alessandro Rinaldo, and Aarti Singh.
\newblock Minimax localization of structural information in large noisy
  matrices.
\newblock In \emph{Advances in Neural Information Processing Systems}, pages
  909--917, 2011.

\bibitem[Kolda and Bader(2009)]{kolda2009tensor}
Tamara~G Kolda and Brett~W Bader.
\newblock Tensor decompositions and applications.
\newblock \emph{SIAM review}, 51\penalty0 (3):\penalty0 455--500, 2009.

\bibitem[Koltchinskii and Lounici(2016)]{koltchinskii2016asymptotics}
Vladimir Koltchinskii and Karim Lounici.
\newblock Asymptotics and concentration bounds for bilinear forms of spectral
  projectors of sample covariance.
\newblock In \emph{Annales de l'Institut Henri Poincar{\'e}, Probabilit{\'e}s
  et Statistiques}, volume~52, pages 1976--2013. Institut Henri Poincar{\'e},
  2016.

\bibitem[Koltchinskii and Lounici(2017)]{koltchinskii2017concentration}
Vladimir Koltchinskii and Karim Lounici.
\newblock Concentration inequalities and moment bounds for sample covariance
  operators.
\newblock \emph{Bernoulli}, 23\penalty0 (1):\penalty0 110--133, 2017.

\bibitem[Koltchinskii and Xia(2016)]{koltchinskii2016perturbation}
Vladimir Koltchinskii and Dong Xia.
\newblock Perturbation of linear forms of singular vectors under gaussian
  noise.
\newblock In \emph{High Dimensional Probability VII}, pages 397--423. Springer,
  2016.

\bibitem[Li and Li(2010)]{li2010tensor}
Nan Li and Baoxin Li.
\newblock Tensor completion for on-board compression of hyperspectral images.
\newblock In \emph{Image Processing (ICIP), 2010 17th IEEE International
  Conference on}, pages 517--520. IEEE, 2010.

\bibitem[Liu et~al.(2013)Liu, Musialski, Wonka, and Ye]{liu2013tensor}
Ji~Liu, Przemyslaw Musialski, Peter Wonka, and Jieping Ye.
\newblock Tensor completion for estimating missing values in visual data.
\newblock \emph{IEEE Transactions on Pattern Analysis and Machine
  Intelligence}, 35\penalty0 (1):\penalty0 208--220, 2013.

\bibitem[Liu et~al.(2017)Liu, Yuan, and Zhao]{liu2017characterizing}
Tianqi Liu, Ming Yuan, and Hongyu Zhao.
\newblock Characterizing spatiotemporal transcriptome of human brain via low
  rank tensor decomposition.
\newblock \emph{arXiv preprint arXiv:1702.07449}, 2017.

\bibitem[Ma and Wu(2015)]{ma2015computational}
Zongming Ma and Yihong Wu.
\newblock Computational barriers in minimax submatrix detection.
\newblock \emph{The Annals of Statistics}, 43\penalty0 (3):\penalty0
  1089--1116, 2015.

\bibitem[McCallum et~al.(2000)McCallum, Nigam, and
  Ungar]{mccallum2000efficient}
Andrew McCallum, Kamal Nigam, and Lyle~H Ungar.
\newblock Efficient clustering of high-dimensional data sets with application
  to reference matching.
\newblock In \emph{Proceedings of the sixth ACM SIGKDD international conference
  on Knowledge discovery and data mining}, pages 169--178. ACM, 2000.

\bibitem[Mitra(2009)]{mitra2009entrywise}
Pradipta Mitra.
\newblock Entrywise bounds for eigenvectors of random graphs.
\newblock \emph{Electronic journal of combinatorics}, 16\penalty0 (1):\penalty0
  R131, 2009.

\bibitem[Muralidhara et~al.(2011)Muralidhara, Gross, Gutell, and
  Alter]{muralidhara2011tensor}
Chaitanya Muralidhara, Andrew~M Gross, Robin~R Gutell, and Orly Alter.
\newblock Tensor decomposition reveals concurrent evolutionary convergences and
  divergences and correlations with structural motifs in ribosomal rna.
\newblock \emph{PloS one}, 6\penalty0 (4):\penalty0 e18768, 2011.

\bibitem[Newman(2004)]{newman2004detecting}
Mark~EJ Newman.
\newblock Detecting community structure in networks.
\newblock \emph{The European Physical Journal B-Condensed Matter and Complex
  Systems}, 38\penalty0 (2):\penalty0 321--330, 2004.

\bibitem[Omberg et~al.(2007)Omberg, Golub, and Alter]{omberg2007tensor}
Larsson Omberg, Gene~H Golub, and Orly Alter.
\newblock A tensor higher-order singular value decomposition for integrative
  analysis of {DNA} microarray data from different studies.
\newblock \emph{Proceedings of the National Academy of Sciences}, 104\penalty0
  (47):\penalty0 18371--18376, 2007.

\bibitem[Parsons et~al.(2004)Parsons, Haque, and Liu]{parsons2004subspace}
Lance Parsons, Ehtesham Haque, and Huan Liu.
\newblock Subspace clustering for high dimensional data: a review.
\newblock \emph{Acm Sigkdd Explorations Newsletter}, 6\penalty0 (1):\penalty0
  90--105, 2004.

\bibitem[Ponnapalli et~al.(2011)Ponnapalli, Saunders, Van~Loan, and
  Alter]{ponnapalli2011higher}
Sri~Priya Ponnapalli, Michael~A Saunders, Charles~F Van~Loan, and Orly Alter.
\newblock A higher-order generalized singular value decomposition for
  comparison of global mrna expression from multiple organisms.
\newblock \emph{PloS one}, 6\penalty0 (12):\penalty0 e28072, 2011.

\bibitem[Richard and Montanari(2014)]{richard2014statistical}
Emile Richard and Andrea Montanari.
\newblock A statistical model for tensor {PCA}.
\newblock In \emph{Advances in Neural Information Processing Systems}, pages
  2897--2905, 2014.

\bibitem[Rudelson and Vershynin(2015)]{rudelson2015delocalization}
Mark Rudelson and Roman Vershynin.
\newblock Delocalization of eigenvectors of random matrices with independent
  entries.
\newblock \emph{Duke Mathematical Journal}, 164\penalty0 (13):\penalty0
  2507--2538, 2015.

\bibitem[Vasilescu and Terzopoulos(2002)]{vasilescu2002multilinear}
M~Vasilescu and Demetri Terzopoulos.
\newblock Multilinear analysis of image ensembles: Tensorfaces.
\newblock \emph{Computer Vision?ECCV 2002}, pages 447--460, 2002.

\bibitem[Vershynin(2010)]{vershynin2010introduction}
Roman Vershynin.
\newblock Introduction to the non-asymptotic analysis of random matrices.
\newblock \emph{arXiv preprint arXiv:1011.3027}, 2010.

\bibitem[Vu and Wang(2015)]{vu2015random}
Van Vu and Ke~Wang.
\newblock Random weighted projections, random quadratic forms and random
  eigenvectors.
\newblock \emph{Random Structures \& Algorithms}, 47\penalty0 (4):\penalty0
  792--821, 2015.

\bibitem[Wang(2015)]{wang2015singular}
Rongrong Wang.
\newblock Singular vector perturbation under gaussian noise.
\newblock \emph{SIAM Journal on Matrix Analysis and Applications}, 36\penalty0
  (1):\penalty0 158--177, 2015.

\bibitem[Wedin(1972)]{wedin1972perturbation}
Perke Wedin.
\newblock Perturbation bounds in connection with singular value decomposition.
\newblock \emph{BIT Numerical Mathematics}, 12\penalty0 (1):\penalty0 99--111,
  1972.

\bibitem[Westin et~al.(2002)Westin, Maier, Mamata, Nabavi, Jolesz, and
  Kikinis]{westin2002processing}
C-F Westin, Stephan~E Maier, Hatsuho Mamata, Arya Nabavi, Ferenc~A Jolesz, and
  Ron Kikinis.
\newblock Processing and visualization for diffusion tensor mri.
\newblock \emph{Medical image analysis}, 6\penalty0 (2):\penalty0 93--108,
  2002.

\bibitem[Xia and Yuan(2019)]{xia2017polynomial}
Dong Xia and Ming Yuan.
\newblock On polynomial time methods for exact low rank tensor completion.
\newblock \emph{Foundations of Computational Mathematics, to appear}, 2019.

\bibitem[Xiong et~al.(2012)Xiong, Ancona, Hauser, Mukherjee, and
  Furey]{xiong2012integrating}
Qing Xiong, Nicola Ancona, Elizabeth~R Hauser, Sayan Mukherjee, and Terrence~S
  Furey.
\newblock Integrating genetic and gene expression evidence into genome-wide
  association analysis of gene sets.
\newblock \emph{Genome research}, 22\penalty0 (2):\penalty0 386--397, 2012.

\bibitem[Zhang and Xia(2018)]{zhang2017tensor}
Anru Zhang and Dong Xia.
\newblock Tensor svd: Statistical and computational limits.
\newblock \emph{IEEE Transactions on Information Theory}, 64\penalty0
  (11):\penalty0 7311--7338, 2018.

\bibitem[Zheng and Tomioka(2015)]{zheng2015interpolating}
Qinqing Zheng and Ryota Tomioka.
\newblock Interpolating convex and non-convex tensor decompositions via the
  subspace norm.
\newblock In \emph{Advances in Neural Information Processing Systems}, pages
  3106--3113, 2015.

\end{thebibliography}

\appendix
\section{Proof of Lemma~\ref{lemma:gammabound}}
Let $\bz_i\in\RR^{m_1}, i=1,\ldots,m_2$ denote the columns of $\bZ$. Then, we write
$$
\bZ\bZ^{\top}-\sigma^2m_2\bI_{m_1}=\sum_{i=1}^{m_2}\big(\bz_i\otimes \bz_i-\sigma^2\bI_{m_1}\big).
$$
Similarly, let $\tilde{\bz}_j\in\RR^{m_1}, j=1,\ldots,m_1$ denote the rows of $\bZ$ and observe that $\|\bB\bZ^{\top}\|=\|\bB\bZ^{\top}\bZ\bB^{\top}\|^{1/2}$ and
$$
\bB\bZ^{\top}\bZ\bB^{\top}=\sum_{j=1}^{m_1}\Big(\big(\bB\check{\bz}_j\big)\otimes \big(\bB\check{\bz}_j\big)-\sigma^2\bB\bB^{\top}\Big).
$$
The inequalities (\ref{eq:AZ}) and (\ref{eq:PAZ}) are on the concentration of sample covariance operator, where a sharp bound has been derived in \cite{koltchinskii2017concentration} and will be skipped here.

\section{Proof of Theorem~\ref{thm:xPky}}
Since $\EE\hat{\bGamma}={\bf 0}$, we immediately get $\mathbb{E}\bL_k(\hat{\bGamma})={\bf 0}$.
Then, 
\begin{eqnarray*}
\big<\bx, \hat{\bP}_k^{uu}\by\big>-\mathbb{E}\big<\bx,\hat{\bP}_k^{uu}\by \big>= \big<\bx, \bL_k(\hat{\bGamma})\by\big>
+\big<\bx, \bS_k(\hat{\bGamma})\by \big>-\mathbb{E}\big<\bx, \bS_k(\hat{\bGamma})\by \big>.
\end{eqnarray*}

\begin{lemma}\label{lemma:Lkcon}
For any $\bx,\by\in\RR^{m_1}$, there exists an absolute constant $D_1>0$ such that for all $0\leq t\leq m_1$, with probability at least $1-e^{-t}$,
\begin{eqnarray*}
\big|\langle\bx, \bL_k(\hat{\bGamma})\by \rangle\big|\leq D_1t^{1/2}\bigg(\frac{\sigma \mu_1+\sigma^2m_2^{1/2}}{\bar{g}_k(\bA\bA^{\top})}\bigg)\|\bx\|_{\ell_2}\|\by\|_{\ell_2}.
\end{eqnarray*}
\end{lemma}

\begin{proof}
    Recall that 
    $$
    \hat{\bGamma} = \bA\bZ^\top+\bZ\bA^\top+\bZ\bZ^\top- m_2\sigma^2 \bI_{m_1}.
    $$
    Then, we write $\big<\bx, \bL_k(\hat{\bGamma})\by \big>$ as
    \begin{align*}
        \langle\bx, \bL_k(\hat{\bGamma})\by \rangle = &\langle \hat{\bGamma} \bP_k^{uu}\bx, \bC_k^{uu} \by\rangle +\langle \hat{\bGamma}\bC_k^{uu}\bx,\bP^{uu}_k\by \rangle \\
         = &\langle ( \bA\bZ^\top+\bZ\bA^\top+\bZ\bZ^\top- m_2\sigma^2 \bI_{m_1})\bP_k^{uu}\bx, \bC_k^{uu} \by\rangle \\
       + & \langle ( \bA\bZ^\top+\bZ\bA^\top+\bZ\bZ^\top- m_2 \sigma^2 \bI_{m_1})\bC_k^{uu}\bx, \bP^{uu}_k \by\rangle.
    \end{align*}
    It suffices to consider the following terms separately for $\bx, \by\in\RR^{m_1}$:
    $$
    \langle \bZ\bA^\top\bx, \by\rangle,\quad \langle \bA\bZ^\top \bx,\by\rangle,\quad \big<\big(\bZ\bZ^\top- m_2\sigma^2\bI_{m_1}\big) \bx,\by\big>.
    $$
    It is straightforward to check that $\langle \bZ\bA^\top\bx, \by\rangle$ is a normal random variable with zero mean and variance
    $$
    \mathbb{E}\langle \bZ\bA^\top\bx, \by\rangle^2 = \mathbb{E}\langle\bZ, \by\otimes (\bA^\top\bx)\rangle^2 = \sigma^2 \|\by\otimes (\bA^\top\bx)\|_{\ell_2}^2 = \sigma^2 \|\by\|_{\ell_2}^2\|\bA^\top\bx\|_{\ell_2}^2,
    $$
    where we used the fact that $\bZ$ is a $m_1\times m_2$ matrix with i.i.d. $\mathcal{N}(0, \sigma^2)$ entries.
    Therefore, 
    $$
    \mathbb{E}\langle \bZ\bA^\top\bP_k^{uu}\bx, \bC_k^{uu}\by\rangle^2 \leq  \frac{\sigma^2\mu_k^2}{\bar{g}^2_k(\bA\bA^{\top})}\|\bx\|_{\ell_2}^2\|\by\|_{\ell_2}^2,  
    $$
    where we used the facts $\|\bC_k\|\leq \frac{1}{\bar{g}_k(\bA\bA^{\top})}$ and $\|\bA^\top\bP_{k}^{uu}\| \leq \mu_k$. By the standard concentration inequality of Gaussian random variables, we get for all $t\geq 0$,
    $$
    \PP\bigg(\big|\big<\bZ\bA^{\top}\bP_k^{uu}\bx, \bC_k^{uu}\by\big>\big|\geq 2t^{1/2}\frac{\sigma \mu_k}{\bar{g}_k(\bA\bA^{\top})}\|\bx\|_{\ell_2}\|\by\|_{\ell_2}\bigg)\leq e^{-t}.
    $$
    Similarly, for all $t\geq 0$,
    $$
    \PP\bigg(\big|\big<\bZ\bA^{\top}\bC_k^{uu}\bx, \bP_k^{uu}\by\big>\big|\geq 2t^{1/2}\frac{\sigma \mu_1}{\bar{g}_k(\bA\bA^{\top})}\|\bx\|_{\ell_2}\|\by\|_{\ell_2}\bigg)\leq e^{-t}.
    $$   
    We next turn to the bound of $\big|\big< \big(\bZ\bZ^\top - m_2\sigma^2\bI_{m_1}\big)\bP_k^{uu}\bx, \bC_k^{uu}\by\big>\big|$. Recall that $\bP_k^{uu}\bC_k^{uu}={\bf 0}$ implying that it suffices to consider $\big<\bZ\bZ^\top\bP_k^{uu}\bx, \bC_k^{uu}\by\big>$.
    Let $\bz_1,\ldots,\bz_{m_2}\in\mathbb{R}^{m_1}$ denote the columns of $\bZ$ such that $\bz_i\in\calN\big({\bf 0}, \sigma^2{\bf I}_{m_1}\big)$ for $1\leq i\leq m_2$. Write
    $$
    \big< \bZ\bZ^{\top}(\bP_k^{uu}\bx), \bC_k^{uu}\by \big>=\sum_{i=1}^{m_2}\big<\bz_i,\bP_k^{uu}\bx\big>\big<\bz_i, \bC_k^{uu}\by \big>.
    $$
    Observe that $\EE\big(\bP_k^{uu}\bz_i\big)\otimes\big(\bC_k^{uu}\bz_i\big)={\bf 0}$ implying that $\big<\bz_i,\bP_k^{uu}\bx\big>$ is independent of $\big<\bz_i, \bC_k^{uu}\by \big>$. By concentration inequalities of Gaussian random variables, for all $t\geq 0$,
    \begin{align*}
    \PP\bigg(\big|\big< \bZ\bZ^{\top}(\bP_k^{uu}\bx), \bC_k^{uu}\by \big>\big|\geq 2t^{1/2}\|\by\|_{\ell_2}&\frac{\sigma\big(\sum_{i=1}^{m_2}\langle\bz_i, \bP_k^{uu}\bx\rangle^2\big)^{1/2}}{\bar{g}_k(\bA\bA^{\top})}\\
    &\Big|\Big\{\langle\bz_i, \bP_k^{uu}\bx\rangle: i=1,\ldots, m_2\Big\}\bigg)\leq e^{-t}.
    \end{align*}
   By \cite[Prop~5.16]{vershynin2010introduction}, the following bound holds with probability at least $1-e^{-t}$,
   $$
   \big|\sum_{i=1}^{m_2}\langle\bz_i, \bP_k^{uu}\bx\rangle^2-\sigma^2m_2\|\bx\|_{\ell_2}^2\big|\lesssim \sigma\Big(m_2^{1/2}t^{1/2}+t\Big)\|\bx\|_{\ell_2}.
   $$
If $t\lesssim m_1\leq m_2$, we conclude that there exists an absolute constant $D_1>0$ such that
$$
\PP\bigg(\big|\big< \bZ\bZ^{\top}(\bP_k^{uu}\bx), \bC_k^{uu}\by \big>\big|\geq D_1\frac{\sigma^2m_2^{1/2}t^{1/2}}{\bar{g}_k(\bA\bA^{\top})}\|\bx\|_{\ell_2}\|\by\|_{\ell_2}\bigg)\leq e^{-t}.
$$
To sum up, for all $0\leq t\lesssim m_1$, the following bound holds with probability at least $1-e^{-t}$,
$$
\big|\big<\bx, \bL_k(\hat{\bGamma})\by\big>\big|\lesssim t^{1/2}\bigg(\frac{\sigma \mu_1+\sigma^2m_2^{1/2}}{\bar{g}_k(\bA\bA^{\top})}\bigg)\|\bx\|_{\ell_2}\|\by\|_{\ell_2}
$$
which concludes the proof.
\end{proof}

It remains to derive the upper bound of $\big|\langle\bx, \bS_k(\hat{\bGamma})\by \rangle-\mathbb{E}\langle\bx, \bS_k(\hat{\bGamma})\by \rangle\big|$.
The following lemma is due to \cite{koltchinskii2016asymptotics}.

\begin{lemma}\label{lemma:Skcon}
Let $\delta(m_1,m_2):=\sigma \mu_1m_1^{1/2}+\sigma^2(m_1m_2)^{1/2}$ and suppose that $\delta(m_1,m_2)\leq \frac{1-\gamma}{2(1+\gamma)}\bar{g}_k(\bA\bA^{\top})$ for some $\gamma\in(0,1)$. There exists a constant $D_{\gamma}>0$ such that, for all symmetric $\widehat{\bGamma}_1, \hat{\bGamma}_2\in\mathbb{R}^{m_1\times m_1}$ satisfying the condition $\max\big\{\|\hat{\bGamma}_1\|, \|\hat{\bGamma}_2\|\big\}\leq (1+\gamma)\delta(m_1,m_2)$,
$$
\|\bS_k(\hat{\bGamma}_1)-\bS_k(\hat{\bGamma}_2)\|\leq D_{\gamma}\frac{\delta(m_1,m_2)}{\bar{g}_k^2(\bA\bA^{\top})}\|\hat{\bGamma}_1-\hat{\bGamma}_2\|.
$$
\end{lemma}
Define function $\varphi(\cdot): \RR_+\mapsto [0, 1]$ such that $\varphi(t)=1$ for $0\leq t\leq 1$ and $\varphi(t)=0$ for $t\geq (1+\gamma)$ and $\varphi$ is linear in between. Then, function $\varphi$ is Lipschitz on $\RR_+$ with constant $\frac{1}{\gamma}$. To illustrate the dependence of $\hat\bGamma$ on $\bZ$, we write $\hat{\bGamma}(\bZ)$ instead of $\hat{\bGamma}$. To this end, fix $\bx, \by\in\RR^{m_1}$ and constants $\delta_1, \delta_2>0$ and define the function
$$
F_{\delta_1,\delta_2,\bx,\by}(\bZ):=\Big<\bx,\bS_k\big(\hat{\bGamma}(\bZ)\big)\by \Big>\varphi\Big(\frac{\|\hat{\bGamma}(\bZ)\|}{\delta_1}\Big)\varphi\Big(\frac{\|\bZ\|}{\delta_2}\Big).
$$
where we view $\bZ$ as a point in $\mathbb{R}^{m_1\times m_2}$ rather than a random matrix.
\begin{lemma}\label{lemma:FLip}
For any $\delta_1\leq \frac{1-\gamma}{2(1+\gamma)}\bar{g}_k(\bA\bA^{\top})$ for some $\gamma\in(0,1)$ and $\delta_2>0$, there exists an absolute constant $C_{\gamma}>0$ such that
\begin{equation*}
\big|F_{\delta_1,\delta_2,\bx,\by}(\bZ_1)-F_{\delta_1,\delta_2,\bx,\by}(\bZ_2)\big|\leq C_{\gamma}\frac{\delta_1}{\bar{g}_k^2(\bA\bA^{\top})}\Big(\mu_1+\delta_2+\frac{\delta_1}{\delta_2}\Big)\|\bZ_1-\bZ_2\|\|\bx\|_{\ell_2}\|\by\|_{\ell_2} \label{eq:fLip}
\end{equation*}
\end{lemma}

\begin{proof}[Proof of Lemma~\ref{lemma:FLip}]
	Since $\varphi(\frac{\|\hat{\bGamma}(\bZ)\|}{\delta_1})\varphi(\frac{\|\bZ\|}{\delta_2})\neq 0$ only if $\|\hat{\bGamma}(\bZ)\| \leq (1+\gamma) \delta_1$ and $\|\bZ\|\leq (1+\gamma)\delta_2$, Lemma \ref{lemma:hatPk} implies that 
	\begin{equation*}
	\big|F_{\delta_1,\delta_2,\bx,\by}(\bZ)\big| = \Big|\Big\langle\bx,\bS_k\big(\hat{\bGamma}(\bZ)\big)\by \Big\rangle\varphi\Big(\frac{\|\hat{\bGamma}(\bZ)\|}{\delta_1}\Big)\varphi\Big(\frac{\|\bZ\|}{\delta_2}\Big)\Big| \leq 14(1+\gamma)^2\frac{\delta_1^2}{\bar{g}_k^2(\bA\bA^{\top})}.
	\end{equation*}
	\paragraph*{Case 1.} If
	$\max\big\{\big\|\hat{\bGamma}(\bZ_1)\big\|, \big\|\hat{\bGamma}(\bZ_2)\big\|\big\} \leq (1+\gamma)\delta_1\quad {\rm and}\quad \max\big\{\|\bZ_1\|, \|\bZ_2\|\big\} \leq (1+\gamma)\delta_2.$
	
	By the Lipschitzity of function $\varphi$, Lemma~\ref{lemma:Skcon} and definition of $\hat{\bGamma}(\bZ)$, it is easy to check 
	\begin{align*}
	|F_{\delta_1,\delta_2,\bx,\by}&(\bZ_1)-F_{\delta_1,\delta_2,\bx,\by}(\bZ_2)|\\
	 \leq& \|\bS_k\big(\hat{\bGamma}(\bZ_1)\big) -\bS_k\big(\hat{\bGamma}(\bZ_2)\big)\|\|\bx\|_{\ell_2}\|\by\|_{\ell_2}\\
	+&\frac{14(1+\gamma)^2\delta_1}{\gamma \bar{g}_k^2(\bA\bA^{\top})} \big\|\hat{\bGamma}(\bZ_1)-\hat{\bGamma}(\bZ_2)\big\|\|\bx\|_{\ell_2}\|\by\|_{\ell_2}
	+ \frac{14(1+\gamma)^2\delta_1^2}{\delta_2\gamma\bar{g}_k^2(\bA\bA^{\top})}\|\bZ_1-\bZ_2\|\|\bx\|_{\ell_2}\|\by\|_{\ell_2} \\
	 \leq& D_{\gamma}\frac{\delta_1}{\bar{g}_k^2(\bA\bA^{\top})}\|\hat{\bGamma}(\bZ_1)-\hat{\bGamma}(\bZ_2)\|\|\bx\|_{\ell_2}\|\by\|_{\ell_2} +  \frac{14(1+\gamma)^2\delta_1^2}{\delta_2\gamma\bar{g}_k^2(\bA\bA^{\top})}\|\bZ_1-\bZ_2\|\|\bx\|_{\ell_2}\|\by\|_{\ell_2}\\
	\leq& D_{\gamma}\frac{\delta_1}{\bar{g}_k^2(\bA\bA^{\top})}\Big(\mu_1+\delta_2+\frac{\delta_1}{\delta_2}\Big)\|\bZ_1-\bZ_2\|\|\bx\|_{\ell_2}\|\by\|_{\ell_2}.
	\end{align*}
	\paragraph*{Case 2.} If 
	$\|\hat{\bGamma}(\bZ_1)\| \leq (1+\gamma)\delta_1,\quad \|\hat{\bGamma}(\bZ_2)\| \geq (1+\gamma)\delta_1\quad{\rm and}\quad \max\big\{\|\bZ_1\|, \|\bZ_2\|\big\} \leq (1+\gamma)\delta_2$.
	Since $\|\hat{\bGamma}(\bZ_2)\| \geq (1+\gamma)\delta_1$, we have $\varphi\big(\frac{\|\hat{\bGamma}(\bZ_2)\|}{\delta_1}\big)=0$ and $F_{\delta_1,\delta_2,\bx,\by}(\bZ_2)=0$. Then,
		\begin{align*}
	\label{lemma6_4}
	\big|F_{\delta_1,\delta_2,\bx,\by}&(\bZ_1)-F_{\delta_1,\delta_2,\bx,\by}(\bZ_2)\big|\\
	 =& \Big|\Big\langle\bx,\bS_k\big(\hat{\bGamma}(\bZ_1)\big)\by \Big\rangle\varphi\Big(\frac{\|\hat{\bGamma}(\bZ_1)\|}{\delta_1}\Big)\varphi\Big(\frac{\|\bZ_1\|}{\delta_2}\Big)\Big| \\
	=&\Big|\Big\langle\bx,\bS_k\big(\hat{\bGamma}(\bZ_1)\big)\by \Big\rangle\varphi\Big(\frac{\|\hat{\bGamma}(\bZ_1)\|}{\delta_1}\Big)\varphi\Big(\frac{\|\bZ_1\|}{\delta_2}\Big)
	-\Big\langle\bx,\bS_k\big(\hat{\bGamma}(\bZ_1)\big)\by \Big\rangle\varphi\Big(\frac{\|\hat{\bGamma}(\bZ_2)\|}{\delta_1}\Big)\varphi\Big(\frac{\|\bZ_1\|}{\delta_2}\Big)\Big| \\
	 \leq& \big\|\bS_k\big(\hat{\bGamma}(\bZ_1)\big)\big\| \frac{1}{\delta_1\gamma} \|\hat{\bGamma}(\bZ_1)-\hat{\bGamma}(\bZ_2)\| \|\bx\|_{\ell_2}\|\by\|_{\ell_2}\\
	 \leq& \frac{(1+\gamma)^2\delta_1^2}{\bar{g}_k^2(\bA\bA^{\top})\delta_1\gamma}  \big(2\mu_1+2(1+\gamma)\delta_2\big)\|\bZ_1-\bZ_2\|  \|\bx\|_{\ell_2}\|\by\|_{\ell_2}\\
	 \leq &D_{\gamma}\frac{\delta_1}{\bar{g}_k^2(\bA\bA^{\top})}(\mu_1+\delta_2)\|\bZ_1 - \bZ_2\| \|\bx\|_{\ell_2}\|\by\|_{\ell_2}.
	\end{align*}
	\paragraph*{Case 3.} If
	$
	\|\hat{\bGamma}(\bZ_1)\| \leq (1+\gamma)\delta_1,\quad \|\hat{\bGamma}(\bZ_2)\| \geq (1+\gamma)\delta_1, \quad\|\bZ_1\| \leq (1+\gamma)\delta_2,\quad \|\bZ_2\| \geq (1+\gamma) \delta_2.
	$	
	It can be proved similarly as {\it Case 2}.
	\paragraph*{Case 4.} If
	$
	\|\hat{\bGamma}(\bZ_1)\| \leq (1+\gamma)\delta_1,\quad \|\hat{\bGamma}(\bZ_2)\| \geq (1+\gamma)\delta_1, \quad \|\bZ_1\| \geq (1+\gamma)\delta_2,\quad \|\bZ_2\| \geq (1+\gamma) \delta_2.
	$
	It is a trivial case since $F_{\delta_1,\delta_2,\bx,\by}(\bZ_1)=F_{\delta_1,\delta_2,\bx,\by}(\bZ_2)=0$.
	
	\paragraph*{Case 5.} If
	$\max\big\{\|\hat{\bGamma}(\bZ_1)\|, \|\hat{\bGamma}(\bZ_2)\|\big\} \leq (1+\gamma)\delta_1,\quad \|\bZ_1\| \leq (1+\gamma)\delta_2,\quad \|\bZ_2\| \geq (1+\gamma)\delta_2.
	$
	Again, we have $F_{\delta_1,\delta_2,\bx,\by}(\bZ_2)=0$. Then,
	\begin{align*}
	\big|F_{\delta_1,\delta_2,\bx,\by}&(\bZ_1)-F_{\delta_1,\delta_2,\bx,\by}(\bZ_2)\big|\\
	 =& \Big|\Big\langle\bx,\bS_k\big(\hat{\bGamma}(\bZ_1)\big)\by \Big\rangle\varphi\Big(\frac{\|\hat{\bGamma}(\bZ_1)\|}{\delta_1}\Big)\varphi\Big(\frac{\|\bZ_1\|}{\delta_2}\Big)\Big| \\
	=&\Big|\Big\langle\bx,\bS_k\big(\hat{\bGamma}(\bZ_1)\big)\by \Big\rangle\varphi\Big(\frac{\|\hat{\bGamma}(\bZ_1)\|}{\delta_1}\Big)\varphi\Big(\frac{\|\bZ_1\|}{\delta_2}\Big)
	-\Big\langle\bx,\bS_k\big(\hat{\bGamma}(\bZ_1)\big)\by \Big\rangle\varphi\Big(\frac{\|\hat{\bGamma}(\bZ_1)\|}{\delta_1}\Big)\varphi\Big(\frac{\|\bZ_2\|}{\delta_2}\Big)\Big| \\
	 \leq& \big\|\bS_k\big(\hat{\bGamma}(\bZ_1)\big)\big\|  \frac{1}{\delta_2\gamma} \|\bZ_1 - \bZ_2\| \|\bx\|_{\ell_2}\|\by\|_{\ell_2}
	 \leq \frac{(1+\gamma)^2\delta_1^2}{\bar{g}_k^2(\bA\bA^{\top})\delta_2\gamma}  \|\bZ_1-\bZ_2\|  \|\bx\|_{\ell_2}\|\by\|_{\ell_2}\\
	 \leq& D_{\gamma}\frac{\delta_1}{\bar{g}_k^2(\bA\bA^{\top})}\frac{\delta_1}{\delta_2}\|\bZ_1 - \bZ_2\| \|\bx\|_{\ell_2}\|\by\|_{\ell_2}.
	\end{align*}
	All the other cases shall be handled similarly and we conclude the proof.
	\end{proof}
	
Note that $\|\bZ_1-\bZ_2\|\leq \|\bZ_1-\bZ_2\|_{\ell_2}$, Lemma~\ref{lemma:FLip} indicates that $F_{\delta_1,\delta_2,\bx,\by}(\bZ)$ is Lipschitz with constant 
$$
D_{\gamma}\frac{\delta_1}{\bar{g}_k^2(\bA\bA^{\top})}\Big(\mu_1+\delta_2+\frac{\delta_1}{\delta_2}\Big)\|\bx\|_{\ell_2}\|\by\|_{\ell_2}.
$$
\begin{lemma}\label{lemma:Skxycon}
Let $\delta(m_1,m_2):=\sigma \mu_1m_1^{1/2}+\sigma^2(m_1m_2)^{1/2}$ and
suppose that $\EE\|\hat{\bGamma}\|  \leq \frac{1-\gamma}{2}\bar{g}_k(\bA\bA^{\top})$ for some $\gamma\in(0,1)$. There exists some constant $D_{\gamma}$ such that for any $\bx,\by\in\mathbb{R}^{m_1}$ and all $\log 8\leq t\leq m_1$, the following inequality holds with probability at least $1-e^{-t}$,
$$
\big|\langle\bx,\bS_k(\hat{\bGamma})\by\rangle-\mathbb{E}\langle\bx,\bS_k(\hat{\bGamma})\by\rangle\big|\leq D_{\gamma}t^{1/2}\frac{\sigma\mu_1+\sigma^2m_2^{1/2}}{\bar{g}_k(\bA\bA^{\top})} \bigg(\frac{\delta(m_1,m_2)}{\bar{g}_k(\bA\bA^{\top})}\bigg)\|\bx\|_{\ell_2}\|\by\|_{\ell_2}.
$$
\end{lemma}
\begin{proof}[Proof of Lemma~\ref{lemma:Skxycon}]
Choose $\delta_1=\delta_1(m_1,m_2)$ and $\delta_2=\delta_2(m_1,m_2)$ as follows where $\log 8\leq t\leq m_1$ is to be determined:
	\begin{align*}
	\delta_1(m_1,m_2):&=\delta_1(m_1,m_2,t):= \mathbb{E}\|\widetilde{\bGamma}\| + D_1t^{1/2}(\sigma\mu_1 + \sigma^2 m_2^{1/2})\\
	\delta_2(m_1,m_2):&=\delta_2(m_1,m_2,t):= \mathbb{E}\|\bZ\| + D_2\sigma t^{1/2}
	\end{align*}
	and the constants $D_1,D_2>0$ are chosen such that $\PP\big(\|\hat{\bGamma}\|\geq \delta_1(m_1,m_2,t)\big)\leq e^{-t}$ and $\PP\big(\|\bZ\|\geq \delta_2(m_1,m_2,t)\big)\leq e^{-t}$.
	Let $M:={\rm Med}(\langle\bx,\bS_k(\hat{\bGamma})\by\rangle)$ denote its median. 
	\paragraph*{Case 1.} If
	$
	D_1t^{1/2}(\mu_1\sigma + \sigma^2 m_2^{1/2}) \leq \frac{\gamma}{4}\bar{g}_k(\bA\bA^{\top}).
	$
	Then, $\delta_1\leq (1-\frac{\gamma}{2})\frac{\bar{g}_k(\bA\bA^{\top})}{2}=\frac{1-2\gamma'}{1+2\gamma'}\frac{\bar{g}_k(\bA\bA^{\top})}{2}$ for some $\gamma'\in(0,1/2)$. By Lemma~\ref{lemma:FLip}, $F_{\delta_1,\delta_2,\bx,\by}(\cdot)$ satisfies the Lipschitz condition. By definition of $F_{\delta_1,\delta_2,\bx,\by}(\bZ)$, we have $F_{\delta_1,\delta_2,\bx,\by}(\bZ)=\langle\bx,\bS_k(\hat{\bGamma})\by\rangle$ on the event $\{\|\hat{\bGamma}\| \leq \delta_1, \|\bZ\|\leq \delta_2\}$. By Lemma~\ref{lemma:gammabound} and $t\geq \log8$,
	\begin{align*}
	\mathbb{P}\Big\{F_{\delta_1,\delta_2,\bx,\by}&(\bZ)\geq M \Big\}\\
	\geq& \mathbb{P}\Big\{F_{\delta_1,\delta_2,\bx,\by}(\bZ)\geq M,\quad \|\hat{\bGamma}\| \leq \delta_1,\quad \|\bZ\| \leq \delta_2\Big\} \\
	\geq& \mathbb{P}\Big\{\langle\bx,\bS_k(\hat{\bGamma})\by\rangle \geq M\Big\} - \mathbb{P}\{ \|\hat{\bGamma}\| \leq \delta_1, \|\bZ\| \leq \delta_2\}\\
	 \geq&  \mathbb{P}\Big\{\langle\bx,\bS_k(\hat{\bGamma})\by\rangle \geq M\Big\} - \mathbb{P}\Big\{ \|\hat{\bGamma}\| \leq \delta_1\Big\}- \mathbb{P}\Big\{ \|\bZ\| \leq \delta_2\Big\}\\
	   \geq&  \frac{1}{2}-\frac{1}{8}-\frac{1}{8}= 1/4,
	\end{align*}
	and similarly, 
	$$
	\mathbb{P}\Big\{F_{\delta_1,\delta_2,\bx,\by}(\bZ)\leq M) \Big\} \geq 1/4.
	$$
	It follows from Gaussian isoperimetric inequality (see \cite[Lemma~2.6]{koltchinskii2016perturbation}) and Lemma \ref{lemma:FLip} that with some constant $D_{\gamma}>0$, for all $t\geq \log 8$ with probability at least $1-e^{-t}$, 
	$$
	\big|F_{\delta_1,\delta_2,\bx,\by}(\bZ)- M\big| \leq D_{\gamma}\frac{\sigma\delta_1t^{1/2}}{\bar{g}_k^2(\bA\bA^{\top})}\Big(\mu_1+\delta_2 +\frac{\delta_1}{\delta_2}\Big)\|\bx\|_{\ell_2}\|\by\|_{\ell_2}.
	$$
	Since $t\leq m_1\leq m_2$, it is easy to check by Lemma~\ref{lemma:gammabound} that $\delta_1\asymp \sigma\mu_1 m_1^{1/2}+\sigma^2(m_1m_2)^{1/2}$ and $\delta_2\asymp \sigma m_2^{1/2}$. Moreover, $\PP\big\{\|\hat{\bGamma}\| \leq \delta_1, \|\bZ\|\leq \delta_2\big\}\geq 1-2e^{-t}$. As a result, with probability at least $1-e^{-3t}$,
	\begin{equation}\label{eq:SkGammacase1}
	\big|\langle\bx,\bS_k(\hat{\bGamma})\by\rangle-M\big| \leq D_{\gamma}\frac{\sigma\mu_1t^{1/2}+\sigma^2m_2^{1/2}t^{1/2}}{\bar{g}_k(\bA\bA^{\top})} \bigg(\frac{\delta(m_1,m_2)}{\bar{g}_k(\bA\bA^{\top})}\bigg)\|\bx\|_{\ell_2}\|\by\|_{\ell_2}.
	\end{equation}
	\paragraph*{Case 2.} If
	$
	D_1t^{1/2}(\sigma\mu_1+\sigma^2m_2^{1/2}) > \frac{\gamma}{4}\bar{g}_k(\bA\bA^{\top}).
	$
	It implies that
	$$
	\mathbb{E}\|\hat{\bGamma}\| \leq D_1\frac{(1-\gamma)}{\gamma}t^{1/2}(\sigma\mu_1+\sigma^2m_2^{1/2}),
	$$
	and
	$
	\delta_1 \leq D_{\gamma}t^{1/2}(\sigma\mu_1+\sigma^2m_2^{1/2}).
	$
	By Lemma~\ref{lemma:gammabound} and Lemma~\ref{lemma:hatPk}, with probability at least $1-e^{-t}$,
	\begin{equation*}
		|\langle\bx,\bS_k(\hat{\bGamma})\by\rangle| \leq \|\bS_k(\hat{\bGamma})\| \leq D_{\gamma}t\frac{(\sigma\mu_1+\sigma^2 m_2^{1/2})^2}{\bar{g}_k^2(\bA\bA^{\top})}\|\bx\|_{\ell_2}\|\by\|_{\ell_2},
	\end{equation*}
   which immediately yields that 
   \begin{equation*}
   M\leq D_{\gamma} \frac{(\sigma\mu_1+\sigma^2 m_2^{1/2})^2}{\bar{g}_k^2(\bA\bA^{\top})}\|\bx\|_{\ell_2}\|\by\|_{\ell_2}. 
   \end{equation*}
   The above inequalities imply that with probability at least $1-e^{-t}$ for $\log 8\leq t\leq m_1$, 
   \begin{align}
       	|\langle\bx,\bS_k(\hat{\bGamma})\by\rangle-M|  \leq& D_{\gamma}t\frac{(\sigma\mu_1+\sigma^2m_2^{1/2})^2}{\bar{g}_k^2(\bA\bA^{\top})}\|\bx\|_{\ell_2}\|\by\|_{\ell_2}\nonumber \\
       	 \leq& D_{\gamma} \frac{\sigma\mu_1t^{1/2}+\sigma^2m_2^{1/2}t^{1/2}}{\bar{g}_k(\bA\bA^{\top})} \bigg(\frac{\delta(m_1,m_2)}{\bar{g}_k(\bA\bA^{\top})}\bigg)\|\bx\|_{\ell_2}\|\by\|_{\ell_2}. \label{eq:SkGammacase2}
   \end{align}
   Therefore, bounds (\ref{eq:SkGammacase1}) and (\ref{eq:SkGammacase2}) hold in both cases. The rest of the proof is quite standard by integrating the exponential tails and will be skipped here, see \cite{koltchinskii2016perturbation}.
	\end{proof}
	\begin{proof}[Proof of Theorem~\ref{thm:xPky}]
	By Lemma~\ref{lemma:Lkcon} and Lemma~\ref{lemma:Skxycon}, if $D_1\delta(m_1,m_2)\leq \bar{g}_k(\bA\bA^{\top})$ for a large enough constant $D_1>0$ such that $\gamma\leq 1/2$, we conclude that for all $\log 8\leq t\leq m_1$, with probability at least $1-2e^{-t}$,
	$$
	\big|\big<\bx, \hat{\bP}_k\by\big>\big|\leq D t^{1/2}\frac{\sigma\mu_1+\sigma^2m_2^{1/2}}{\bar{g}_k(\bA\bA^{\top})}\|\bx\|_{\ell_2}\|\by\|_{\ell_2}
	$$
	which concludes the proof after adjusting the constant $D$ accordingly.
	\end{proof}

	\section{Proof of Lemma~\ref{lemma:hatS-tildeS}}
	Observe that for any $\bx,\by\in\RR^{m_1}$ with $\|\bx\|_{\ell_2}=\|\by\|_{\ell_2}=1$ and  $\delta_t=\mathbb{E}\|\hat{\bGamma}\|+D_1\sigma\mu_1t^{1/2}+D_2\sigma^2m_2^{1/2}t^{1/2}$ with $t\leq m_1$ and some $\gamma\in (0,1/2]$,
\begin{align*}
\Big|\mathbb{E}\big\langle\bx,\big(\bS_k(\widetilde{\bGamma})-\bS_k(\hat{\bGamma})\big)\by\big\rangle\Big|
&\leq \mathbb{E}\Big\|\bS_k(\widetilde{\bGamma})-\bS_k(\hat{\bGamma})\Big\|\\
&= \mathbb{E}\Big\|\bS_k(\widetilde{\bGamma})-\bS_k(\hat{\bGamma})\Big\|{\bf 1}\Big(\|\widetilde{\bGamma}\|\leq (1+\gamma)\delta_t\Big){\bf 1}\Big(\|\hat{\bGamma}\|\leq (1+\gamma)\delta_t\Big)\\
&+ \mathbb{E}\Big\|\bS_k(\widetilde{\bGamma})-\bS_k(\hat{\bGamma})\Big\|{\bf 1}\Big(\|\widetilde{\bGamma}\|\leq (1+\gamma)\delta_t\Big){\bf 1}\Big(\|\hat{\bGamma}\|> (1+\gamma)\delta_t\Big)\\
&+ \mathbb{E}\Big\|\bS_k(\widetilde{\bGamma})-\bS_k(\hat{\bGamma})\Big\|{\bf 1}\Big(\|\widetilde{\bGamma}\|>(1+\gamma)\delta_t\Big){\bf 1}\Big(\|\hat{\bGamma}\|\leq (1+\gamma)\delta_t\Big)\\
&+ \mathbb{E}\Big\|\bS_k(\widetilde{\bGamma})-\bS_k(\hat{\bGamma})\Big\|{\bf 1}\Big(\|\widetilde{\bGamma}\|> (1+\gamma)\delta_t\Big){\bf 1}\Big(\|\hat{\bGamma}\|>(1+\gamma)\delta_t\Big)
\end{align*}
where the constants $D_1,D_2>0$ are chosen such that $\max\big\{\PP\big(\|\widetilde\bGamma\|\geq \delta_t\big), \PP\big(\|\hat\bGamma\|\geq \delta_t\big)\big\}\leq e^{-t}$.
By Lemma~\ref{lemma:Skcon},
\begin{align*}
\mathbb{E}\Big\|\bS_k(\widetilde{\bGamma})-&\bS_k(\hat{\bGamma})\Big\|{\bf 1}\Big(\|\widetilde{\bGamma}\|\leq (1+\gamma)\delta_t\Big){\bf 1}\Big(\|\hat{\bGamma}\|\leq (1+\gamma)\delta_t\Big)\\
\leq &D_{\gamma}\frac{\delta_t}{\bar{g}_k^2(\bA\bA^{\top})}\mathbb{E}\|\widetilde{\bGamma}-\hat{\bGamma}\|\leq D_{\gamma}\frac{\delta_t}{\bar{g}_k^2(\bA\bA^{\top})}\mathbb{E}\|\bZ\bP_k^{hh}\bZ^{\top}-\nu_k\sigma^2\bI_{m_1}\|.
\end{align*}
By writing $\bP_k^{hh}:=\sum_{j\in \Delta_k} \bh_{j}\otimes \bh_{j}$, we obtain
\begin{align*}
\bZ\bP_k^{hh}\bZ^{\top}-\sigma^2\nu_k\bI_{m_1}=&\sum_{j\in\Delta_k}(\bZ\bh_{j})\otimes (\bZ\bh_{j})-\sigma^2\nu_k\bI_{m_1}\\
=&\nu_k\Big(\frac{1}{\nu_k}\sum_{j\in\Delta_k}(\bZ\bh_{j})\otimes (\bZ\bh_{j})-\sigma^2\bI_{m_1}\Big).
\end{align*}
where $\nu_k={\rm Card}(\Delta_k)$.
The vectors $\bZ\bh_{j}\sim \calN(0, \sigma^2\bI_{m_1})$ and $\{\bZ\bh_{j}: \ldots, j\in\Delta_k\}$ are independent. By \cite{koltchinskii2017concentration},
$$
\mathbb{E}\Big\|\frac{1}{\nu_k}\sum_{j\in\Delta_k}(\bZ\bh_j)\otimes (\bZ\bh_j)-\sigma^2\bI_{m_1}\Big\|\lesssim \sigma^2\Big(\sqrt{\frac{m_1}{\nu_k}}\vee \frac{m_1}{\nu_k}\Big).
$$
Since $\nu_k\leq m_1$, we conclude with
\begin{align}\label{eq:hatS-tildeS-eq1}
\mathbb{E}\Big\|\bS_k(\widetilde{\bGamma})-\bS_k(\hat{\bGamma})\Big\|{\bf 1}\Big(\|\widetilde{\bGamma}\|\leq (1+\gamma)\delta_t\Big){\bf 1}&\Big(\|\hat{\bGamma}\|\leq (1+\gamma)\delta_t\Big)\\
\lesssim_{\gamma}& \frac{\delta_t}{\bar{g}_k(\bA\bA^{\top})}\bigg(\frac{m_1\sigma^2}{\bar{g}_k(\bA\bA^{\top})}\bigg).\nonumber
\end{align}
Choose $t=m_1$, by Lemma~\ref{lemma:hatPk} and Lemma~\ref{lemma:gammabound},
\begin{align*}
\mathbb{E}\Big\|\bS_k(\widetilde{\bGamma})-&\bS_k(\hat{\bGamma})\Big\|{\bf 1}\Big(\|\hat{\bGamma}\|\leq (1+\gamma)\delta_{m_1}\Big){\bf 1}\Big(\|\widetilde{\bGamma}\|>(1+\gamma)\delta_{m_1}\Big)\\
&\leq D_{\gamma} \frac{\delta_{m_1}^2}{\bar{g}_k^2(\bA\bA^{\top})}\EE \frac{\|\widetilde{\bGamma}\|^2}{\bar{g}_k^2(\bA\bA^{\top})}{\bf 1}\Big(\|\widetilde{\bGamma}\|>(1+\gamma)\delta_{m_1}\Big)\\
&\lesssim_{\gamma} \frac{\delta_{m_1}^2}{\bar{g}_k^4(\bA\bA^{\top})}e^{-m_1/2}\EE^{1/2}\|\widetilde{\bGamma}\|^4\lesssim  \frac{\delta_{m_1}^4}{\bar{g}_k^4(\bA\bA^{\top})}e^{-m_1/2}\\
&\lesssim  \frac{\delta(m_1,m_2)}{\bar{g}_k(\bA\bA^{\top})}\bigg(\frac{\sigma\mu_1+\sigma^2m_1}{\bar{g}_k(\bA\bA^{\top})}\bigg)
\end{align*}
which is clearly dominated by (\ref{eq:hatS-tildeS-eq1}) for $t=m_1$ and $m_2e^{-m_1/2}\leq 1$. The other terms are bounded in a similar fashion. 
To sum up, we obtain
 $$
 \|\mathbb{E}\bS_k(\widetilde{\bGamma})-\EE\bS_k(\hat{\bGamma})\|\lesssim \frac{\sigma\mu_1+\sigma^2m_1}{\bar{g}_k(\bA\bA^{\top})}\bigg( \frac{\delta(m_1,m_2)}{\bar{g}_k(\bA\bA^{\top})}\bigg).
 $$

\end{document}